\documentclass{amsart}
\allowdisplaybreaks[4]
\usepackage{graphicx}
\usepackage{color}

\newcommand{\R}{\mathbb{R}}

\newcommand{\sig}{\sigma}

\newcommand{{\ba}}{\bf a}
\newcommand{\ve}{\varepsilon}
\newcommand{\la}{\lambda}
\newcommand{\La}{\Lambda}
\newcommand{\ga}{\gamma}
\newcommand{\Ga}{\Gamma}
\newcommand{\pa}{\partial}
\newcommand{\ra}{\rightarrow}

\newcommand{\del}{\delta}

\newcommand{\al}{\alpha}

\newcommand{\be}{\begin{equation}}
\newcommand{\ee}{\end{equation}}

\newtheorem{lem}{Lemma}{\bf}{\it}
{\it}{\rm}
\newtheorem{rem}{Remark}{\it}{\rm}
{\it}{\rm}

\newtheorem{theorem}{Theorem}
\newtheorem{proposition}{Proposition}
\newtheorem{corollary}{Corollary}
\numberwithin{theorem}{section}
\numberwithin{lem}{section}
\numberwithin{equation}{section}
\numberwithin{proposition}{section}
\numberwithin{corollary}{section}
\numberwithin{rem}{section}
\title[Green's function on half line]{ On Properties of the Dirichlet Green's function for linear diffusions on a half line}

\author{Joseph G. Conlon and Michael Dabkowski}

\address{Joseph G. Conlon: University of Michigan\\ Department of Mathematics\\ Ann Arbor,
  MI 48109-1109}
\email{conlon@umich.edu}

\address{Michael Dabkowski: University of Michigan-Dearborn\\ Department of Mathematics and Statistics\\ Dearborn,
  MI 48128}
\email{mgdabkow@umich.edu}


\begin{document}

\maketitle

\begin{abstract}
This paper is concerned with the study of Green's functions for one dimensional diffusions with constant diffusion coefficient and  linear time inhomogeneous drift.
It is well know that the whole line Green's function is given by a Gaussian.  Formulas for the Dirichlet Green's function on the half line are only known in special cases. 
The main object of study in the paper is the ratio of the Dirichlet to whole line Green's functions.  Bounds, asymptotic behavior in the limit as the diffusion coefficient vanishes, and a log concavity result are obtained for this ratio.
\end{abstract}

\section{Introduction}
In this paper we shall be concerned with obtaining properties of Green's functions for one dimensional diffusions with linear drift. Let $b:\mathbb{R}\times \mathbb{R}\ra\mathbb{R}$ be a continuous function $(y,t)\ra b(y,t)$, which is linear in the space variable $y$. For $\ve>0$, the terminal value problem 
\be \label{A1}
\frac{\pa u_\ve(y,t)}{\pa t}+b(y,t)\frac{\pa u_\ve(y,t)}{\pa y}+\frac{\ve}{2}\frac{\pa^2 u_\ve(y,t)}{\pa y^2} \ = \ 0, \quad y\in\R, \ t<T,
\ee
\be \label{B1}
u_\ve(y,T) \ = \ u_T(y), \quad y\in\mathbb{R} \ ,
\ee
has a unique solution $u_\ve$ which has the representation
\be \label{C1}
u_\ve(y,t) \ = \ \int_{-\infty}^\infty G_\ve(x,y,t,T) u_T(x) \ dx, \quad y\in\R, \ t<T,
\ee
where $G_\ve$ is the Green's function for the problem.  
The adjoint problem to (\ref{A1}), (\ref{B1}) is the initial value problem
\be \label{I1}
\frac{\pa v_\ve(x,t)}{\pa t}+\frac{\pa}{\pa x}\left[b(x,t)v_\ve(x,t)\right] \ = \ \frac{\ve}{2}\frac{\pa^2 v_\ve(x,t)}{\pa x^2} \ , \quad x\in\R, \ t>0,
\ee
\be \label{J1}
v_\ve(x,0) \ = \ v_0(x), \quad y\in\R.
\ee
The solution to (\ref{I1}), (\ref{J1}) is given by the formula
\be \label{K1}
v_\ve(x,T) \ = \ \int_{-\infty}^\infty G_\ve(x,y,0,T) v_0(y) \ dy, \quad x\in\R, \  T>0.
\ee

Since the drift $b(\cdot,\cdot)$ is linear, $G_\ve$ is Gaussian, so the function $(x,y)\ra \log G_\ve(x,y,t,T)$ is quadratic in $(x,y)$.   Here we shall obtain properties of the corresponding Dirichlet Green's function $(x,y)\ra G_{\ve,D}(x,y,t,T)$ on the half line $x,y>0$. Thus
\be \label{D1}
u_{\ve,D}(y,t) \ = \ \int_0^\infty G_{\ve,D}(x,y,t,T) u_T(x) \ dx, \quad y>0, \ t<T,
\ee
is the solution to (\ref{A1}), (\ref{B1}) in the domain  $\{(y,t): \ y>0, \ t<T\}$ with Dirichlet boundary condition $u_{\ve,D}(0,t)=0, \ t<T$. 

The drifts $b(\cdot,\cdot)$ we consider are of the form 
\be \label{E1}
b(y,t)=A(t)y-1, \quad {\rm where \ }A:\mathbb{R}\ra\mathbb{R} \  {\rm is \  a \  continuous \  function,}
\ee
but the methods of the paper may be extended to more general linear drifts.
In the case $A(\cdot)\equiv 0$ there are simple explicit formulas for $G_\ve$ and $G_{\ve,D}$. These are given by 
\be \label{F1}
G_\ve(x,y,t,T) \ = \ \frac{1}{\sqrt{2\pi\ve(T-t)}}\exp\left[-\frac{(x+T-t-y)^2}{2\ve (T-t)}\right] \ ,
\ee
\be \label{G1}
G_{\ve,D}(x,y,t,T)  \ = \ \left\{1-\exp\left[-\frac{2xy}{\ve(T-t)}\right]\right\}G_\ve(x,y,t,T) \ . 
\ee
For non-trivial $A(\cdot)$, we write
\be \label{H1}
 G_{\ve,D}(x,y,t,T)  \ = \  \left\{1-\exp\left[-\frac{q_\ve(x,y,t,T)}{\ve}\right]\right\}G_\ve(x,y,t,T) \ ,
\ee
and study the properties of the function $q_\ve$. When $A(\cdot)\equiv 0$ we have from (\ref{G1}) that the function $(x,y)\ra q_\ve(x,y,t,T)$ is independent of $\ve$ and bilinear.  One can also obtain explicit formulas for $q_\ve$ in some other cases of linear drift, in particular for the Ornstein-Uhlenbeck process where $b(y,t)=-\ga y$ with constant $\ga$ (see Proposition 20 of \cite{svy} and Remark 3.1 of the present paper).  However there appears not to be an explicit formula for $q_\ve$ in the case of general function $A(\cdot)$.  
   
We are able to obtain linear bounds on the function $x\ra q_\ve(x,y,t,T), \ x>0,$ and its first two $x$ derivatives, which are uniform in $\ve>0$,  when the function $A(\cdot)$ is assumed to be non-negative:
\begin{theorem}
Assume the function $A(\cdot)$ of (\ref{E1}) is continuous and non-negative, and $q_\ve$ is defined by (\ref{H1}). Then there exists a continuous positive function
$\al_1(\cdot)$ and continuous non-negative functions $\beta_1(\cdot),\beta_2(\cdot)$, with domain $\{(t,T): \ t,T\in\mathbb{R}, \ t<T\}$, which bound $q_\ve$  as follows:  
\be \label{L1}
\al_1(t,T)xy \ \le q_\ve(x,y,t,T) \ \le \ [\al_1(t,T)x+\beta_1(t,T)]y \ , \quad x,y>0, \ t<T,
\ee
\be \label{M1}
\lim_{x\ra \infty}\{[\al_1(t,T)x+\beta_1(t,T)]y-q_\ve(x,y,t,T)\} \ = \ 0 \ , \quad  y>0, \ t<T,
\ee
\be \label{N1}
q_\ve(x,y,t,T) \ \le \ [\al_1(t,T)y +\beta_2(t,T)]x\ , \quad x,y>0, \ t<T,
\ee
\be \label{O1}
\al_1(t,T)y \ \le \frac{\pa q_\ve(x,y,t,T)}{\pa x} \ \le \ \al_1(t,T)y+\beta_2(t,T) \ , \quad x,y>0, \ t<T,
\ee
\be \label{P1}
\lim_{x\ra\infty}\left\{ \frac{\pa q_\ve(x,y,t,T)}{\pa x} - \al_1(t,T)y\right\} \ = \ 0 \ , \quad y>0, \ t<T,
\ee
\be \label{Q1}
{\rm The \  function}  \quad  x\ra q_\ve(x,y,t,T), \ x,y>0,  \ t<T,\quad {\rm is  \ concave.}
\ee
\end{theorem}  
Since $q_\ve(0,y,t,T)=0$ the lower bound in (\ref{L1}) is implied by the lower bound in (\ref{O1}). Similarly the upper bound in (\ref{O1}) implies the upper bound in (\ref{N1}).  However our proof of (\ref{O1}) in Proposition 5.1  uses the inequalities (\ref{L1}), (\ref{N1}), which have previously been established in Proposition 3.3. 
The proof of (\ref{M1}) is given in Proposition 5.2,  the proof of (\ref{P1}) in Proposition 5.3, and the proof of (\ref{Q1}) in Theorem 5.1. 

Theorem 1.1 tells us that the graph of the function   $x\ra q_\ve(x,y,t,T), \ x>0,$ lies between two parallel lines and is asymptotic to the upper line at large $x$. 
The graph also lies  in a wedge formed by two lines through the origin and is concave.  This geometric picture gives us a rather precise understanding of the global behavior of the function $x\ra q_\ve(x,y,0,T), \ x>0$.  
The significance of the upper bounds (\ref{N1}), (\ref{O1}) can be understood by considering the situation when $\ve \ra 0$.  The function $[x,T]\ra q_\ve(x,y,0,T), \ x,T>0,$ is a solution to the Hamilton-Jacobi-Bellman equation (\ref{V2}). One expects then that $\lim_{\ve\ra 0}q_\ve(x,y,t,T)=q_0(x,y,t,T)$ exists,  and in the case $t=0$ is a solution to the Hamilton-Jacobi equation (\ref{AC2}). In $\S2$ we show that  the limit does exist and  $q_0(x,y,0,T)$ is given by the  variational formula (\ref{AB2}) corresponding to the Hamilton-Jacobi equation.  In $\S4$ we study this variational problem in great detail, establishing in particular that if $A(\cdot)$ is continuous and non-negative then the function $x\ra q_0(x,y,t,T)$ is differentiable in a neighborhood of $x=0$ and $\pa q_0(0,y,t,T)/\pa x=\al_1(t,T)y+\beta_2(t,T)$.  More precisely we have the following:
\begin{theorem}
Assume the function $A(\cdot)$ of (\ref{E1}) is continuous and and $q_\ve$ is defined by (\ref{H1}). Then there is a continuous function $[x,y,t,T]\ra q_0(x,y,t,T)$ with domain $\{[x,y,t,T]: \ x,y\ge 0, \ t,T\in\mathbb{R}, \ t<T\}$ such that $\lim_{\ve\ra 0}q_\ve(x,y,t,T)=q_0(x,y,t,T)$ for all $x,y\ge 0, \ t<T$. If $A(\cdot)$ is non-negative then
the function  $x\ra q_0(x,y,t,T)$ is differentiable at $x=0$ and
\be \label{R1}
\frac{\pa q_0(x,y,t,T)}{\pa x} \Big|_{x=0} \ = \ \al_1(t,T)y+\beta_2(t,T) \ .
\ee
Furthermore, for any $t\in\mathbb{R},T_0>0,$ there are constants $C_1,C_2>0$, depending only on $T_0$ and $\sup_{t\le s\le t+T_0}A(s)$, such that
\be \label{S1}
\left|\frac{\pa q_\ve(0,y,t,T)}{\pa x}- [\al_1(t,T)y+\beta_2(t,T)]\right| \ \le \  \frac{\ve C_2(T-t)^2}{y^2} \quad {\rm for \ } y\ge C_1(T-t)^2, \ T-t\le T_0,
\ee
provided $\ve\le (T-t)^3$. 
\end{theorem}

In Theorem 2.1 we prove that  $\lim_{\ve\ra 0}q_\ve(x,y,t,T)=q_0(x,y,t,T)$ exists and is the solution to the variational problem (\ref{AB2}). In Proposition 4.2 we show if $A(\cdot)$ is non-negative that $q_0(x,y,0,T)$ may be obtained by the method of characteristics in a subdomain of $\{[x,T]:  \ x>0,T>0\}$, which includes  a neighborhood of the boundary $\{[0,T]: \ T>0\}$. The function $[x,T]\ra q_0(x,y,0,T)$ is then a classical solution of the Hamilton-Jacobi equation (\ref{AC2}) in this region, and the derivative $\pa q_0(x,y,0,T)/\pa x$ is given by the formula (\ref{Z4}). 

We use the methods of stochastic control theory to prove that $\lim_{\ve\ra0}q_\ve=q_0$ and (\ref{S1}).  The Bellman equation (\ref{V2})  corresponds to the stochastic variational problem (\ref{Z2}). The proofs of $\lim_{\ve\ra0}q_\ve=q_0$ and (\ref{S1})  are then obtained by comparing the solution of this stochastic variational problem to the solution of the classical variational problem (\ref{AB2}). The proof of (\ref{S1}) is given in Proposition 6.1, and uses in a crucial way the regularity of the function $x\ra q_0(x,y,0,T)$ in a neighborhood of $x=0$. 

The ratio of Green's functions $G_{\ve,D}(x,y,t,T)/G_\ve(x,y,t,T)$ given in (\ref{H1})  is the probability that a generalized Brownian bridge, beginning at $y$ at time $t$ and ending at $x$ at time $T$, lies entirely in the positive half line. We may therefore try to estimate $q_\ve(x,y,t,T)$ by comparing this generalized bridge to the standard Brownian bridge. This method of bridge comparison is used in the proof of Proposition 3.2. However for the most part we use the fact that the bridge process is a Gaussian Markov process with the linear drift (\ref{O2}) in order to estimate $q_\ve(x,y,t,T)$.   There is a considerable literature on the study of bridges.
In \cite{conf,cr} Conforti et al study bridges associated to diffusions with a gradient drift, using the fact that it is the {\it reciprocal characteristics} which determine the bridge uniquely.  In particular, diffusions with differing drifts may have the same bridge processes. A simple example of this is the case of the drift  (\ref{E1}) with $A(\cdot)\equiv 0$. The bridge process associated with the constant drift is the same as the Brownian bridge.  In Proposition 3 of \cite{gn}  formulas for  first passage time  for  diffusions with time-inhomogeneous drift are given. However in these cases there needs to be a relation between the graph of the boundary and the drift and diffusion coefficients of the process. The first passage time  for the half line is given in terms of the Dirichlet Green's function by the function $t\ra\int_0^\infty G_{\ve,D}(x,y,t,T) \ dx, \ t<T$. 
 An alternative approach to understanding the limit $\lim_{\ve\ra 0}q_\ve$ may be taken using the techniques of large deviation theory \cite{fw}. This is the approach in
 Baldi et al \cite{bc,bcr}, which considers the asymptotic behavior of the ratio of Green's functions in the limit  $T-t\ra 0$ for time homogeneous diffusions. 
 
 The results obtained in this paper have been motivated by some issues which occur in the problem of coarsening for the diffusive Carr-Penrose model \cite{cdw}. In order to understand the large time asymptotic behavior of this model, it is necessary to have very good control over the ratio of the Dirchlet to whole line Green's functions.  In particular, it is important for the coarsening problem that constants in inequalities such as (\ref{S1}) depend only on 
 $T_0>0$ and $\sup_{t<s<t+T_0}A(s)$.

\vspace{.1in}
\section{ Representation and convergence of the function $q_\ve$}
For any $t\in\mathbb{R}$ let $Y_\ve(s), \ s>t,$ be the solution to the initial value problem for the stochastic differential equation (SDE)
\be \label{A2}
dY_\ve(s) \ = \ b(Y_\ve(s),s) ds+\sqrt{\ve} \ dB(s), \quad Y_\ve(t)=y,
\ee
where $B(\cdot)$ is Brownian motion.  Then the Green's function $G_\ve(\cdot,y,t,T)$ defined by (\ref{C1}) is the probability density for the random variable $Y_\ve(T)$.  In the case when the function $(y,t)\ra b(y,t)$ is linear in $y$ it is easy to see that (\ref{A2}) can be explicitly solved. The solution to (\ref{A2}) with $b(y,t)=A(t)y-1$ as in (\ref{E1}) is  given by 
\begin{multline} \label{B2}
Y_\ve(s) \ = \ \exp\left[\int_t^s A(s') ds'\right] y-\int_t^s \exp\left[\int_{s'}^s A(s'') ds''\right]  \ ds' \\
+\sqrt{\ve}\int_t^s \exp\left[\int_{s'}^s A(s'') ds''\right] dB(s') \ . 
\end{multline}
Hence the random variable $Y_\ve(T)$ conditioned on $Y_\ve(0)=y$ is Gaussian with mean
$m_{1,A}(T)y-m_{2,A}(T)$ and variance $\ve\sig_A^2(T)$, where $m_{1,A},m_{2,A}$ are given by 
\be \label{C2}
m_{1,A}(T)= \exp\left[\int_0^T A(s') ds'\right], \quad 
m_{2,A}(T)= \int_0^T \exp\left[\int_{s}^T A(s') ds'\right]  \  ds \ , 
\ee
 and $\sig_A^2(T)$ by
\be \label{D2}
\sig_A^2(T)= \int_0^T \exp\left[2\int_{s}^T A(s') ds'\right] ds \ .  
\ee
The Green's function $G_\ve(x,y,0,T)$ is therefore explicitly given by the formula
\be \label{E2}
G_\ve(x,y,0,T)=\frac{1}{\sqrt{2\pi\ve\sig_A^2(T)}}\exp\left[-\frac{\{x+m_{2,A}(T)-m_{1,A}(T)y\}^2}{2\ve\sig_A^2(T)}\right] \ .
\ee

It is useful to recall that the ratio $G_{\ve,D}/G_\ve$ is a probability for a generalized Brownian bridge process. Thus
\be \label{F2}
\frac{G_{\ve,D}(x,y,0,T)}{G_\ve(x,y,0,T)} \ = \  P(\inf_{0<s<T} Y_\ve(s)>0 \ | \ Y_\ve(0)=y, \ Y_\ve(T)=x) \ ,
\ee
where $Y_\ve(\cdot)$ is the solution to the SDE  (\ref{A2}). The process $Y_\ve(\cdot)$ of (\ref{A2}), conditioned on $Y_\ve(0)=y, \ Y_\ve(T)=x$  is  Gaussian and has variance independent of $x,y$. We may obtain a formula for it by extending the functions  $m_{1,A},m_{2,A}, \sig_A^2$ of (\ref{C2}), (\ref{D2}),  defined with respect to the interval $[0,T]$, to any interval $[t,T]$ with $t<T$. Thus we define $m_{1,A}(t,T),m_{2,A}(t,T)$ by
\be \label{G2}
m_{1,A}(t,T)= \exp\left[\int_t^T A(s') ds'\right], \quad 
m_{2,A}(t,T)= \int_t^T \exp\left[\int_{s}^T A(s') ds'\right]  \  ds \ , 
\ee
 and $\sig_A^2(t,T)$ by
\be \label{H2}
\sig_A^2(t,T)= \int_t^T \exp\left[2\int_{s}^T A(s') ds'\right] ds \ .  
\ee
The variance of $Y_\ve(s)$ is then given by the formula
\be \label{I2}
{\rm Var}[ Y_\ve(s) \ | \ Y_\ve(0)=y, \ Y_\ve(T)=x] \ = \  \ve \sig_A^2(s)\sig_A^2(s,T)/\sig_A^2(T) \ .
\ee
More generally, the covariance of $Y_\ve(\cdot)$ is also independent of $x,y$ and is given by the formula
\be \label{J2}
{\rm Covar}[ Y_\ve(s_1),Y_\ve(s_2) \ | \ Y_\ve(0)=y, \ Y_\ve(T)=x] \ = \  \ve \Gamma_A(s_1,s_2) \ , \quad 0\le s_1, s_2\le T,
\ee 
where the symmetric function $\Gamma:[0,T]\times[0,T]\ra\R$   is defined by
\be \label{K2}
\Gamma_A(s_1,s_2) \ = \ \frac{m_{1,A}(s_1,s_2)\sig_A^2(s_1)\sig_A^2(s_2,T)}{\sig_A^2(T)} \ , \quad 0\le s_1\le s_2\le T  \ .
\ee
Let $y_{\rm class}(s), \ 0\le s\le T$, be the path going from $y$ at $s=0$ to $x$ at $s=T$ defined by
\begin{multline} \label{L2}
\sig_A^2(T)y_{\rm class}(s) \ = \  xm_{1,A}(s,T)\sig_A^2(s) +ym_{1,A}(s)\sig_A^2(s,T) \\
+ \ m_{1,A}(s,T)m_{2,A}(s,T)\sig_A^2(s)-m_{2,A}(s)\sig_A^2(s,T) \ .
\end{multline}
Then the mean of $Y_\ve(\cdot)$ conditioned on $Y_\ve(0)=y, \ Y_\ve(T)=x$ is given by the formula
\be \label{M2}
E[ Y_\ve(s) \ | \ Y_\ve(0)=y, \ Y_\ve(T)=x] \ = \ y_{\rm class} (s), \quad 0\le s\le T \ .
\ee
In the case $A(\cdot)\equiv 0$ the process $Y_\ve(\cdot)$, conditioned on $Y_\ve(0)=y, \ Y_\ve(T)=x$  is the standard Brownian Bridge (BB) from $y$ at time $0$ to $x$ at time $T$.  

It is well known that the conditioned process $Y_\ve(\cdot)$  is also Markovian.  In \cite{cdw} we showed that
it is the solution to an SDE  with a linear drift depending on $x$ and with initial condition $Y_\ve(0)=y$ (see (4.43), (4.46) of \cite{cdw}). The SDE in this case is run {\it forwards} in time. Here we observe that the conditioned process is also the solution of an SDE with a linear drift depending on $y$, which is run {\it backwards} in time. Denoting by $X_\ve(s), \ 0<s<T,$ the solution to this SDE with terminal condition $X_\ve(T)=x$, we have that
\be \label{N2}
dX_\ve(s) \ = \ \la(X_\ve(s),y,s) \ ds+\sqrt{\ve} \ dB(s) \ , \ \ 0<s<T, \quad X_\ve(T)=x \ .
\ee
The function $\la(x,y,s)$ is given by the formula
\be \label{O2}
\la(x,y,s) \ = \ \left[A(s)+\frac{1}{\sig_A^2(s)}\right]x-1+\frac{m_{2,A}(s)}{\sig_A^2(s)}-\frac{m_{1,A}(s)y}{\sig_A^2(s)} \ .
\ee
Integrating (\ref{N2}), (\ref{O2}), we have that 
\be \label{P2}
X_\ve(s) \ = \ y_{\rm class}(s)-\sqrt{\ve}\frac{\sig^2_A(s)}{m_{1,A}(s)}Z(s), \quad {\rm with \ } Z(s)=\int_s^T \frac{m_{1,A}(s') \ dB(s')}{\sig^2_A(s')} \ ,
\ee
where $y_{\rm class}(\cdot)$ is given in  (\ref{L2}).  Note that the drift $\la(\cdot)$ is independent of $\ve$ and  $\la(x,y,s)$ becomes singular as $s\ra 0$.  The singularity is necessary in order to ensure that $X_\ve(0)=y$ with probability $1$, 

We define now
\be \label{Q2}
v_\ve(x,y,T) \ = \ P\left(\inf_{0\le s\le T} X_\ve(s)<0 \ \big| \ X_\ve(T)=x\right) \ .
\ee
Comparing (\ref{F2}), (\ref{Q2}) we see that
\be \label{R2}
v_\ve(x,y,T) \ = \ 1-\frac{G_{\ve,D}(x,y,0,T)}{G_\ve(x,y,0,T)} \ .
\ee
The function $v_\ve$ is a solution to the PDE
\be \label{S2}
\frac{\pa v_\ve(x,y,T)}{\pa T} \ = \ -\la(x,y,T)\frac{\pa v_\ve(x,y,T)}{\pa x}+\frac{\ve}{2}\frac{\pa^2 v_\ve(x,y,T)}{\pa x^2} \ , \quad 
T>0,x>0,
\ee
with initial and boundary conditions
\be \label{T2}
\lim_{x\ra0}v_\ve(x,y,T) \ = \ 1, \ T>0,  \quad \lim_{T\ra 0}v_\ve(x,y,T) \ = \ 0,  \ x>0 \ .
\ee
Next we set $q_\ve$ to be
\be \label{U2}
q_\ve(x,y,T) \ = \ -\ve\log v_\ve(x,y,T) \ .
\ee
Then $q_\ve$ is a solution to the PDE
\be \label{V2}
\frac{\pa q_\ve(x,y,T)}{\pa T} \ = \ -\la(x,y,T)\frac{\pa q_\ve(x,y,T)}{\pa x}-\frac{1}{2}\left[\frac{\pa q_\ve(x,y,T)}{\pa x}\right]^2+\frac{\ve}{2}\frac{\pa^2 q_\ve(x,y,T)}{\pa x^2} \ , \quad 
T>0,x>0,
\ee
with initial and boundary conditions
\be \label{W2}
q_\ve(0,y,T) \ = \ 0, \ T>0,  \quad q_\ve(x,y,0) \ = \ +\infty,  \ x>0 \ .
\ee
In the case $A(\cdot)\equiv 0$ it is easy to see that the drift $\la$ and solution $q_\ve$ to (\ref{V2}), (\ref{W2}) are given by the formulae
\be \label{X2}
\la(x,y,s) \ = \ \frac{(x-y)}{s} \ , \quad q_\ve(x,y,T) \ = \ \frac{2xy}{T} \ . 
\ee
Evidently (\ref{X2}) is consistent with (\ref{G1}). 

The PDE (\ref{V2}) is the Hamilton-Jacobi -Bellman (HJB) equation for a stochastic control problem. Thus consider solutions $X_\ve(\cdot)$ to the SDE
\be \label{Y2}
dX_\ve(s) \ = \ \mu_\ve(X_\ve(s),y,s) \ ds +\sqrt{\ve}dB(s) \ , 
\ee
run {\it backwards} in time with controller $\mu_\ve(\cdot)$ and given terminal data.  For $x,y,T>0$ define $q_\ve(x,y,T)$  by
\begin{multline} \label{Z2}
q_\ve(x,y,T) \ = \\
 \min_{\mu_\ve}E\left[\frac{1}{2}\int_\tau^T[\mu_\ve(X_\ve(s),y,s)-\la(X_\ve(s),y,s)]^2 \ ds \ \Big| \ X_\ve(T)=x, \ 0<\tau<T, \ ,  \ X_\ve(\cdot)>0, \ X_\ve(\tau)=0 \  \right] \ ,
\end{multline}
where the function $\la(\cdot)$ is given by (\ref{O2}). The class of controllers $\mu_\ve(\cdot)$ in (\ref{Y2}) are those which have the property that paths $X_\ve(s), \ s<T,$ with $X_\ve(T)=x>0$, exit the half line $(0,\infty)$ before time $0$ with probability $1$.    The  HJB equation for $q_\ve$ is then given by   
(\ref{V2}), with initial and boundary conditions (\ref{W2}). The optimal controller $\mu^*_\ve(\cdot)$ in (\ref{Y2}), (\ref{Z2}) is given by the formula
\be \label{AA2}
\mu^*_\ve(x,y,T) \ = \ \la(x,y,T)+\frac{\pa q_\ve(x,y,T)}{\pa x} \ .
\ee

 The  zero noise limit $\ve\ra 0$ of (\ref{Y2}), (\ref{Z2}) yields the classical variational formula
\begin{multline} \label{AB2}
q_0(x,y,T) \ = \\
 \min\left\{\frac{1}{2}\int_\tau^T\left[\frac{dx(s)}{ds}-\la(x(s),y,s)\right]^2 \ ds \ \Big| \ 0<\tau<T, \ x(T)=x, \ x(\cdot)>0, \ x(\tau)=0 \right\} \ .
\end{multline}
At least formally, the function $q_0$ is the solution to the Hamilton-Jacobi (HJ) equation
 \be \label{AC2}
\frac{\pa q_0(x,y,T)}{\pa T} \ = \ -\la(x,y,T)\frac{\pa q_0(x,y,T)}{\pa x}-\frac{1}{2}\left[\frac{\pa q_0(x,y,T)}{\pa x}\right]^2 \ , \quad 
T>0,x>0,
\ee
with initial and boundary conditions
\be \label{AD2}
q_0(0,y,T) \ = \ 0, \ T>0,  \quad q_0(x,y,0) \ = \ +\infty,  \ x>0 \ .
\ee
When $A(\cdot)\equiv 0$ the function $q_0(x,y,T)=2xy/T$ is a classical $C^1$ solution to (\ref{AC2}), (\ref{AD2}). However in general we can only expect $q_0$ to be a  viscosity solution of the HJ equation (see Chapter 10 of \cite{evans}).  

To obtain an upper bound on $q_\ve$ by $q_0$ plus a constant which vanishes as $\ve\ra0$, we observe that the variational problem (\ref{AB2}) for {\it fixed} $\tau$ with $0<\tau<T$, without the positivity constraint $x(\cdot)>0$,  is quadratic  with a linear constraint, which may be easily solved. The Euler-Lagrange equation for the minimization problem with  fixed $\tau$ is 
\be \label{AE2}
\left\{\frac{d}{ds}+\frac{\pa \la(x(s),y,s)}{\pa x}\right\}\left[\frac{dx(s)}{ds}-\la(x(s),y,s)\right] \ = \ 0 \ .
\ee
The minimizing trajectory is then the solution to (\ref{AE2}) with initial and terminal conditions $x(\tau)=0, \ x(T)=x$. To obtain a formula for this trajectory we observe that the solution to the equation 
\be \label{AF2}
\frac{d\phi(s)}{ds}+\frac{\pa \la(x(s),y,s)}{\pa x}\phi(s) \ = \ 0 \ ,
\ee
is given by the formula
\be \label{AG2}
\phi(s) \ = \ \frac{C_1m_{1,A}(s)}{\sig^2_A(s)} \quad {\rm where \  }C_1 \ {\rm is  \ constant.}
\ee
We then need to obtain the solution to
\be \label{AH2}
\frac{dx(s)}{ds}-\la(x(s),y,s) \ = \ \phi(s) \ ,
\ee
with initial and terminal conditions $x(\tau)=0, \ x(T)=x$, and this determines the constant $C_1$ in (\ref{AG2}).
Observe that the function $x(s)=y_{\rm class}(s), \ 0<s<T,$ where $y_{\rm class}(\cdot)$ is defined by (\ref{L2}), is the solution to (\ref{AH2}) with terminal condition $x(T)=x$ in the case $\phi(\cdot)\equiv 0$.
Let $y_p(\cdot)$ be the solution to the terminal value problem
\be \label{AI2}
\frac{dy_p(s)}{ds}-\left[A(s)+\frac{1}{\sig_A^2(s)}\right]y_p(s) \ + \ \frac{m_{1,A}(s)}{\sig^2_A(s)} \ = \ 0  \ , \quad y_p(T)=0 \ .
\ee
The solution to (\ref{AH2}) with initial and terminal conditions $x(\tau)=0, \ x(T)=x$, is then $x(s)=y_{\rm class}(s)-C_1y_p(s)$, where $C_1$ is chosen so that $x(\tau)=0$.  The solution to (\ref{AI2}) is given by the formula
\be \label{AJ2}
\sig^2_A(T)y_p(s) \ = \ m_{1,A}(s)\sig_A^2(s,T) \  .
\ee
We have then from (\ref{AI2}), (\ref{AJ2}) that
\begin{multline} \label{AK2}
\sig_A^2(T)x(s) \ = \  xm_{1,A}(s,T)\sig_A^2(s) +[y-\ga(\tau)]m_{1,A}(s)\sig_A^2(s,T) \\
+ \ m_{1,A}(s,T)m_{2,A}(s,T)\sig_A^2(s)-m_{2,A}(s)\sig_A^2(s,T) \ .
\end{multline}
where $\ga(\tau)$ is chosen so that $x(\tau)=0$.  The optimal controller $\mu^*_{0,\tau}$ for the variational problem with fixed $\tau$  is obtained by evaluating $dx(s)/ds$ at $s=T$.  Thus  $\mu^*_{0,\tau}(x,y,T)$ is given by the formula 
\begin{multline} \label{AL2}
\mu^*_{0,\tau}(x,y,T) \ = \ \la(x,y,T)+\frac{\ga(\tau)m_{1,A}(T)}{\sig^2_A(T)} \\
= \ \la(x,y,T)+\frac{m_{1,A}(T)}{\sig^2_A(T)}\left[y+ g_{1,A}(\tau,T)x+g_{2,A}(\tau,T)\right] \ ,
\end{multline}
where the functions $g_{1,A},g_{2,A}$ are given by the formulae
\be \label{AM2}
g_{1,A}(s,T) \ = \ \frac{m_{1,A}(s,T)\sig_A^2(s)}{m_{1,A}(s)\sig_A^2(s,T)} \ ,  \quad  s<T \ ,
\ee
\be \label{AN2}
g_{2,A}(s,T) \ = \ \frac{m_{1,A}(s,T)m_{2,A}(s,T)\sig_A^2(s)-m_{2,A}(s)\sig_A^2(s,T) }{m_{1,A}(s)\sig_A^2(s,T)} \ , \quad s<T \ .
\ee
\begin{lem}
Let $\tau>0$ and $X_\ve(s), \ s>\tau,$ be the solution to the SDE (\ref{Y2}) with $\mu_\ve$ given by $\mu_\ve(x,y,s)=\mu^*_{0,\tau}(x,y,s), \ s>\tau$.  For $x>0, \ T>\tau$ let $\tau_{\ve,x,T}$ be the first exit time from the interval $(0,\infty)$  of $X_\ve(\cdot)$ with terminal condition $X_\ve(T)=x$.  Then $\tau_{\ve,x,T}>\tau$ with probability $1$ and
\be \label{AO2}
q_\ve(x,y,T) \ \le \  E\left[\frac{1}{2}\int_{\tau_{\ve,x,T}}^T[\mu_\ve(X_\ve(s),y,s)-\la(X_\ve(s),y,s)]^2 \ ds \ \Big| \ X_\ve(T)=x \  \right] \ .
\ee
\end{lem}
\begin{proof}
Since the function $A(\cdot)$ is continuous, we have from (\ref{AL2})-(\ref{AN2}) that
\be \label{AP2}
\mu_\ve(x,y,s) \ = \  \left[\frac{1}{s-\tau}+A_\tau(s)\right]x+(s-\tau)B_\tau(s) \ ,  \quad s>\tau,
\ee
where $A_\tau,B_\tau$ are  continuous functions on the closed interval $[\tau,\infty)$.  Let $m_{1,A_\tau}$ be defined as in (\ref{G2}). The solution to (\ref{Y2}) with $\mu_\ve$ as in (\ref{AP2}) and terminal condition $X_\ve(T)=x$ is given by the formula
\be \label{AQ2}
X_\ve(s) \ = \  \frac{s-\tau}{(T-\tau)m_{1,A_\tau}(s,T)}\left[X_{\rm class}(s)-\sqrt{\ve}Z(s)\right]  \ , \quad \tau<s<T \ ,
\ee
where $X_{\rm class}(\cdot), \ Z(\cdot)$  are given by the formulae
\be \label{AR2}
X_{\rm class}(s) \ = \ x-(T-\tau)\int_s^T m_{1,A_\tau}(s',T)B_\tau(s') \ ds' \ ,
\ee
\be \label{AS2}
Z(s) \ = \ (T-\tau)\int_s^T \frac{m_{1,A_\tau}(s',T)}{s'-\tau}dB(s') \ .
\ee
Since $Z(\cdot)$ is by a change of variable equivalent to Brownian motion, the reflection principle applies to it. Hence for any $a>0, \ \tau<s<T$,
\be \label{AT2}
P(\sup _{s<s'<T}Z(s')>a) \ = \ 2P(Z(s)>a) \ .
\ee
From (\ref{AS2}) we see that $Z(s)$ is Gaussian with mean zero. The variance ${\rm Var}[Z(s)]$  satisfies the inequality
\be \label{AU2}
\frac{c_1(T-\tau)(T-s)}{s-\tau} \ \le {\rm Var}[Z(s)] \ \le \ \frac{C_1(T-\tau)(T-s)}{s-\tau} \ ,  \quad \tau<s<T \ ,
\ee
for some positive constants $c_1,C_1$. From (\ref{AT2}), (\ref{AU2}) we conclude that
\be \label{AV2}
P(\sup_{s<s'<T}Z(s')>a) \ \ge \ 1-C_2a\frac{\sqrt{s-\tau}}{\sqrt{(T-\tau)(T-s)}} \ , \quad \tau<s<T \ ,
\ee
where $C_2>0$ is a constant.  We also have from (\ref{AR2}) there is a constant $C_3$ such that
$\sup_{\tau<s<T} X_{\rm class}(s) \ \le \ C_3$.
Choosing $a=\ve^{-1/2}C_3$ in (\ref{AV2}),  we conclude from (\ref{AQ2}), (\ref{AV2}) that $\lim_{s\ra\tau}P(\tau_{\ve,x,T}>s)=1$.  Hence $\tau_{\ve,x,T}>\tau$ with probability $1$. 

To prove (\ref{AO2}) we first observe from Ito's lemma that the mapping $s\ra M(s)$ on the interval $\tau<s\le T$, where
\begin{multline} \label{AX2}
 M(s) \ = \ q_\ve(x,y,T)-q_\ve(X_\ve(s),y,s)\\
-\int_s^T\frac{\pa q_\ve(X_\ve(s'),y,s')}{\pa s'}+\frac{\pa q_\ve(X_\ve(s'),y,s')}{\pa x}\mu_\ve(X_\ve(s'),y,s')-
\frac{\ve}{2} \frac{\pa^2 q_\ve(X_\ve(s'),y,s')}{\pa x^2} \ ds' 
\end{multline}
 is a (backwards in time) stochastic integral.  From (\ref{V2}) we see that $M(s)$ can be written as
 \begin{multline} \label{AY2}
 M(s) \ = \ q_\ve(x,y,T)-q_\ve(X_\ve(s),y,s) \\
 -\int_s^T\frac{\pa q_\ve(X_\ve(s'),y,s')}{\pa x}\left[\mu_\ve(X_\ve(s'),y,s')-\la(X_\ve(s'),y,s')\right]-
 \frac{1}{2}\left[\frac{\pa q_\ve(X_\ve(s'),y,s')}{\pa x}\right]^2 \ ds' \ .
 \end{multline}
 For $K>0$ and $0<x<K$ let $\tau_{\ve,x,T,K}$ be the first exit time of $X_\ve(s), \ s<T,$ with $X_\ve(T)=x$ from the interval $(0,K)$. By the optional sampling theorem we have that $E[M(s\vee\tau_{\ve,x,T,K})]=0$ for all $s$ in the interval $\tau<s\le T$.  Hence on using the Schwarz inequality in (\ref{AY2}) we have that
 \begin{multline} \label{AZ2}
 q_\ve(x,y,T) \ \le E[q_\ve(X_\ve(s\vee\tau_{\ve,x,T,K}),y,s\vee\tau_{\ve,x,T,K})] \\
 + E\left[\frac{1}{2}\int_{s\vee\tau_{\ve,x,T,K}}^T[\mu_\ve(X_\ve(s),y,s)-\la(X_\ve(s),y,s)]^2 \ ds \ \Big| \ X_\ve(T)=x \  \right] \ , \quad \tau<s\le T \ .
 \end{multline}
 
 Observe that for $\tau<s\le T$ we have $s\vee\tau_{\ve,x,T,K}\ge\tau_{\ve,x,T,K}\ge \tau_{\ve,x,T}>\tau$ with probability $1$.  Hence the second expectation on the RHS of (\ref{AZ2}) is bounded above by the expectation on the RHS of (\ref{AO2}). Thus it is sufficient to show that the first term on the RHS of (\ref{AZ2}) converges to $0$ as $s\ra \tau$ and $K\ra \infty$.  We first consider the limit $s\ra\tau$. Since $\tau_{\ve,x,T,K}>\tau$ with probability 
 $1$, we have by path continuity of $X_\ve(\cdot)$  that $\lim_{s\ra\tau}X_\ve(s\vee\tau_{\ve,x,T,K})=X_\ve(\tau_{\ve,x,T,K})$ with probability $1$.  Using the continuity of the function $q_\ve$, it follows by dominated convergence that
 \begin{multline} \label{BA2}
 \lim_{s\ra\tau} E[q_\ve(X_\ve(s\vee\tau_{\ve,x,T,K}),y,s\vee\tau_{\ve,x,T,K})] \ = \\ 
 E[q_\ve(X_\ve(\tau_{\ve,x,T,K}),y,\tau_{\ve,x,T,K})]  \ \le \ \sup_{\tau<s\le T} q_\ve(K,y,s)P(X_\ve(\tau_{\ve,x,T,K})=K) \ .
 \end{multline}
 From (\ref{AT2}), (\ref{AU2}) we have for some constants $C_2,c_2>0$ that
 \be \label{BB2}
 P(\inf_{s<s'<T}Z(s')<-a) \ \le \frac{C_2}{a\sqrt{s-\tau}}\exp\left[-c_2a^2(s-\tau)\right] \ , \quad a>0 \ .
 \ee
 We have from (\ref{AQ2}) there is a constant $c_3>0$ such that for all large $K$,
 \be \label{BC2}
 P(X_\ve(\tau_{\ve,x,T,K})=K)  \  \le \ P(\inf_{\tau<s\le T}(s-\tau)Z(s)<-c_3K/\sqrt{\ve}) \ .
 \ee
 From (\ref{BB2}) we see that
 \begin{multline} \label{BD2}
 P(\inf_{\tau<s\le T}(s-\tau)Z(s)<-c_3K/\sqrt{\ve}) \\
  \le  
 \sum_{n=0}^\infty P\left(\inf_{2^n<(T-\tau)/(s-\tau)\le 2^{n+1}}Z(s)<-\frac{2^nc_3K}{\sqrt{\ve}(T-\tau)}\right) \
 \le \ C_2\sum_{n=0}^\infty\exp\left[-\frac{c_2c_3^2K^2 2^{n-1}}{\sqrt{\ve}(T-\tau)}\right] \ ,
 \end{multline}
if $K$ is large.  From our bound (\ref{I3}) on $q_\ve$ we conclude from (\ref{BC2}), (\ref{BD2}) that  \\
$\lim_{K\ra\infty} \sup_{\tau<s\le T} q_\ve(K,y,s)P(X_\ve(\tau_{\ve,x,T,K})=K) =0$, whence (\ref{AO2}) follows from (\ref{AZ2}). 
\end{proof}
\begin{lem}
For  $x,y,T$  positive  one has
$\limsup_{\ve\ra 0}[q_\ve(x,y,T) -q_0(x,y,T)]=0$.
\end{lem}
\begin{proof}
We first observe that a minimizing $\tau=\tau_{0,x,T}$ (which may not be unique) in (\ref{AB2}) satisfies $0<\tau_{0,x,T}<T$.  To see this we use that fact that the minimizing trajectory for fixed $\tau$ is given by (\ref{AQ2}), (\ref{AR2}) with $\ve=0$.  From (\ref{O2}) we have that  $\la(x,y,s)\simeq(x-y)/s$ as $s\ra 0$, whence the minimum action integral diverges as $\tau\ra 0$. When $\tau$ is close to $T$ we have from (\ref{AP2}), (\ref{AQ2}) that $\mu^*_{0,\tau}(X_0(s),y,s)\simeq x/(T-\tau)$, whence the minimum action integral again diverges as $\tau\ra T$. 

From (\ref{AQ2})-(\ref{AS2}) we have that 
\be \label{BF2}
X_\ve(s)-X_0(s) \ = \  -\frac{s-\tau}{(T-\tau)m_{1,A_\tau}(s,T)}\sqrt{\ve}Z(s) \ , \quad \tau<s\le T \ .
\ee
From (\ref{O2}), (\ref{AO2}), (\ref{AP2}) we have that
\begin{multline} \label{BG2}
q_\ve(x,y,T) \ \le \  \sqrt{\ve}E\left[\int_{\tau_{\ve,x,T}}^T f_\tau(s)Z(s) \  ds  \right] +C(\tau)\ve  E\left[\frac{1}{2}\int_{\tau_{\ve,x,T}}^T Z(s)^2 \  ds  \right] 
 \\ +E\left[\frac{1}{2}\int_{\tau_{\ve,x,T}}^T[\mu^*_{0,\tau}(X_0(s),y,s)-\la(X_0(s),y,s)]^2 \ ds \ \Big| \ X_0(T)=x \  \right]  \ ,
\end{multline}
where the function $f_\tau:[\tau,T]\ra\mathbb{R}$ is a continuous function depending on $\tau$, and $C(\tau)$ is a constant also depending on $\tau$.  We choose now $\tau$ to be a minimizer of the action (\ref{AB2}). Since $\tau_{\ve,x,T}>\tau$ with probability $1$, it follows that the last expectation on the RHS of (\ref{BG2}) is bounded above by $q_0(x,y,T)$.  

To bound the first term on the RHS of (\ref{BG2}) we note  from the optional sampling theorem that
 $E[Z(s\vee\tau_{\ve,x,T})]=0$. Hence $E\left[ Z(s); \ \tau_{\ve,x,T}<s\right]=-E[Z(\tau_{\ve,x,T}); \ \tau_{\ve,x,T}>s]$.
Letting $M=\sup_{\tau<s\le T}X_{\rm class}(s)$, we see from (\ref{AQ2}) that $0<Z(\tau_{\ve,x,T})<M/\sqrt{\ve}$.  Hence we have that
\be \label{BH2}
\sqrt{\ve}E\left[\int_{\tau_{\ve,x,T}}^T f_\tau(s)Z(s) \  ds  \right]  \ \le \ C_1E[\tau_{\ve,x,T}-\tau] \ , 
\ee
for some constant $C_1$ independent of $\ve$. For $0<\la<T-\tau$ we let $\ga(\la)=\inf_{\tau+\la<s<T}X_{\rm class}(s)$, whence $\ga:(0,T-\tau]\ra\mathbb{R}$ is a positive decreasing function.  We have from (\ref{BB2}) and reflection symmetry that
\begin{multline} \label{BI2}
P(\tau_{\ve,x,T}-\tau>\la) \\ \le \
 P(\sup_{\tau+\la<s<T}Z(s)>\ga(\la)/\sqrt{\ve}) \
\le \ \frac{C_2\sqrt{\ve}}{\sqrt{\la}\ga(\la)}\exp\left[-\frac{c_2\la\ga(\la)^2}{\ve}\right] \ .
\end{multline}
We see from (\ref{BI2}) that $\lim_{\ve\ra 0} P(\tau_{\ve,x,T}-\tau>\la)=0$ for all $\la>0$, whence $\lim_{\ve\ra0} E[\tau_{\ve,x,T}-\tau]=0$ by dominated convergence.

We see from (\ref{AU2}) that the second term on the RHS of (\ref{BG2}) diverges if we replace $\tau_{\ve,x,T}$ by $\tau$. Therefore it is again necessary to estimate the distribution of the variable $\tau_{\ve,x,T}-\tau>0$ as $\ve\ra0$.  With $M$ as in the previous paragraph, and using (\ref{AT2}), (\ref{AU2}) we have from (\ref{AQ2}) that
\begin{multline} \label{BJ2}
P(\tau_{\ve,x,T}-\tau<\la) \ \le \ P\left(\sup_{\tau+\la<s< T} Z(s)<M/\sqrt{\ve}\right) \\
 = \ P\left(|Z(\tau+\la)|<M/\sqrt{\ve}\right) \ \le \  C_2(\la/\ve)^{1/2} \ , \quad 0<\la<T-\tau \ ,
\end{multline}
where $c_2,C_2>0$ are constants. We write now
\be \label{BK2}
E\left[\int_{\tau_{\ve,x,T}}^T Z(s)^2 \  ds  \right]  \ \le \ E\left[\int_{\tau+\ve}^T Z(s)^2 \  ds  \right]+\sum_{n=0}^\infty a_n \ ,
\ee
where
\be \label{BL2}
a_n \ = \  E\left[\int_{\tau+2^{-(n+1)}\ve}^{\tau+2^{-n}\ve} Z(s)^2 \  ds \ ; \ \tau_{\ve,x,T}-\tau<2^{-n}\ve  \right] \ , \quad n=0,1,\dots
\ee
It follows from (\ref{AU2}) that the first term on the RHS of (\ref{BK2}) is bounded by $C_3|\log\ve|$ for some constant $C_3$.  From the Schwarz inequality we have that
\be \label{BM2}
a_n \ \le \ 2^{-(n+1)/2}\sqrt{\ve}P\left(  \tau_{\ve,x,T}-\tau<2^{-n}\ve \right)^{1/2}E\left[\int_{\tau+2^{-(n+1)}\ve}^{\tau+2^{-n}\ve} Z(s)^4 \  ds\right]^{1/2} \ .
\ee
It follows from (\ref{AU2}), (\ref{BJ2}), (\ref{BM2}) that $a_n\le C_4P\left(  \tau_{\ve,x,T}-\tau<2^{-n}\ve\right)^{1/2}\le C_52^{-n/2}$ for some constants $C_4,C_5$.  We have shown that the second term on the RHS of (\ref{BG2}) converges to $0$ as $\ve\ra 0$. 
\end{proof}
\begin{rem}
One can obtain a rate of convergence $\ve\log\ve$ in Lemma 2.2 as $\ve\ra0$ by making a further assumption that the classical trajectory $X_0(\cdot)$ has the property $X'_0(\tau)>0$. In that case $\lim_{\la\ra0}\ga(\la)\ge c_1$ for some  constant $c_1>0$, whence (\ref{BI2}) implies that $E[\tau_{\ve,x,T}-\tau]\le C_2\ve$ for some constant $C_2$. 
\end{rem}
To obtain a lower bound for $q_\ve$ by  $q_0$ plus a constant which vanishes as $\ve\ra 0$ we need to show that
the variational formula (\ref{Z2}) yields a lower bound when $\mu_\ve$ is chosen to be the optimal controller $\mu^*_\ve$ given by (\ref{AA2}).  We have from propositions  3.1, 3.2   that the function $(x,T)\ra\mu_\ve^*(x,y,T)$ is $C^1$ on the domain $x,T>0$ and $\mu_\ve^*(x,y,T)\ge \la(x,y,T)$ for $x,y,T>0$. Hence the SDE (\ref{Y2}) with $\mu_\ve=\mu_\ve^*$   may be solved backwards in time. Letting $X^*_\ve(s), \ s\le T,$ be the solution with terminal condition $X^*_\ve(T)=x$, then $X^*_\ve(s)\le X_\ve(s), \ 0<s\le T$, where $X_\ve(\cdot)$ is given by (\ref{P2}). 
\begin{lem}
For $x,T>0$ and paths $X^*_\ve(s), \ s<T$, with $X^*_\ve(T)=x$ we define $\tau^*_{\ve,x,T}=\inf\{s>0: \ X^*_\ve(s')>0, \ s\le s'\le T\}$.  Then $\tau^*_{\ve,x,T}>0$ with probability $1$ and
\be \label{BO2}
E\left[\left(\frac{1}{\tau^*_{\ve,x,T}}\right)^{1/2-\nu}\right] \ < \ \infty \quad {\rm for \ any \ }\nu \ {\rm with \ } 0<\nu\le 1/2 \ ,
\ee
\be \label{BP2}
q_\ve(x,y,T) \ \ge \ 
  E\left[\frac{1}{2}\int_{\tau^*_{\ve,x,T}}^T[\mu^*_\ve(X^*_\ve(s),y,s)-\la(X^*_\ve(s),y,s)]^2 \ ds \ \Big| \ X^*_\ve(T)=x \  \right] \ .
\ee
\end{lem}
\begin{proof}
We consider the stochastic integral $s\ra M(s), \ 0<s\le T$, defined by (\ref{AX2}), (\ref{AY2}) with $\mu_\ve=\mu_\ve^*$. For $K>0$ and $0<x<K$ let $\tau^*_{\ve,x,T,K}$ be the first exit time of $X^*_\ve(s), \ s<T,$ with $X^*_\ve(T)=x$ from the interval $(0,K)$. Since $q_\ve$ is non-negative, we have from (\ref{AA2}) and the optional sampling theorem that
 \begin{multline} \label{BQ2}
 q_\ve(x,y,T) \\
  \ge \ 
  E\left[\frac{1}{2}\int_{s\vee\tau^*_{\ve,x,T,K}}^T[\mu^*_\ve(X^*_\ve(s),y,s)-\la(X^*_\ve(s),y,s)]^2 \ ds \ \Big| \ X^*_\ve(T)=x \  \right] \ , \quad 0<s\le T \ .
 \end{multline}
 Since $X^*_\ve(\cdot)\le X_\ve(\cdot)$, it follows that $s\vee\tau^*_{\ve,x,T,K}\ra s\vee\tau^*_{\ve,x,T}$ with probability $1$ as $K\ra\infty$, whence the inequality (\ref{BQ2}) holds with $\tau^*_{\ve,x,T}$ in place of $\tau^*_{\ve,x,T,K}$. 
 
 Letting  $X_\ve(\cdot)$ be the solution to (\ref{Y2}) with terminal condition $X_\ve(T)=x>0$,  we have from (\ref{O2}), (\ref{Y2}) that
\begin{multline} \label{BR2}
\frac{m_{1,A}(s)}{\sig^2_A(s)}\left[\mu_\ve(X_\ve(s),y,s)-\la(X_\ve(s),y,s)\right]  ds \\
 = \ d\left[\frac{m_{1,A}(s)X_\ve(s)}{\sig^2_A(s)}\right]-\sqrt{\ve}\frac{m_{1,A}(s)}{\sig^2_A(s)}dB(s) 
 -d\left[\frac{m_{1,A}(s)^2\sig_A^2(s,T)}{\sig^2_A(T)\sig^2_A(s)}\right]y \\
- \ d\left[\frac{m_{1,A}(T)m_{2,A}(s,T)}{\sig^2_A(T)}-\frac{m_{1,A}(s)m_{2,A}(s)\sig_A^2(s,T)}{\sig^2_A(T)\sig^2_A(s)}\right] \ , \quad 0<s\le T \ .
\end{multline}
Let  $\del$ satisfy $0<\del\le T$ and $\tau_\del$ be the stopping time $\tau_\del=\del\vee\tau^*_{\ve,x,T}$. On integrating (\ref{BR2}) with 
$\mu_\ve=\mu^*_\ve$ over the interval 
$\tau_\del<s<T$, we obtain the identity
\begin{multline} \label{BS2}
\int_{\tau_\del}^T\frac{m_{1,A}(s)}{\sig^2_A(s)}\left[\mu^*_\ve(X^*_\ve(s),y,s)-\la(X^*_\ve(s),y,s)\right]  ds  \ = \\
 \frac{m_{1,A}(T)x}{\sig^2_A(T)}-\frac{m_{1,A}(\tau_\del)X^*_\ve(\tau_\del)}{\sig^2_A(\tau_\del)}-\sqrt{\ve}Z(\tau_\del) \\
 +\frac{m_{1,A}(\tau_\del)^2\sig_A^2(\tau_\del,T)}{\sig^2_A(T)\sig^2_A(\tau_\del)}y 
+\frac{m_{1,A}(T)m_{2,A}(\tau_\del,T)}{\sig^2_A(T)} \\
-\frac{m_{1,A}(\tau_\del)m_{2,A}(\tau_\del)\sig_A^2(\tau_\del,T)}{\sig^2_A(T)\sig^2_A(\tau_\del)} \ ,
\end{multline}
where the martingale $Z(\cdot)$ is defined in (\ref{P2}). From the Schwarz inequality we see that the LHS of (\ref{BS2}) is bounded above by
\be \label{BT2}
\frac{\al y}{2}\int_{\tau_\del}^T\frac{m_{1,A}(s)^2}{\sig^4_A(s)} \ ds+
\frac{1}{2\al y}\int_{\tau_\del}^T \left[\mu^*_\ve(X^*_\ve(s),y,s)-\la(X^*_\ve(s),y,s)\right]^2  ds  \ , 
\ee
for any $\al>0$. Choosing $\al$ sufficiently small, we conclude from (\ref{BS2}) that
\begin{multline} \label{BU2}
\frac{y}{\tau^*_{\ve,x,T}} \ \le \ C_1\left[1+\sqrt{\ve} \ |Z(\tau_\del)| \ \right] \\
+ \frac{C_1}{2\al y}\int_{\tau^*_{\ve,x,T}}^T \left[\mu^*_\ve(X^*_\ve(s),y,s)-\la(X^*_\ve(s),y,s)\right]^2  ds\quad {\rm if  \ }\del\le\tau^*_{\ve,x,T}\le T/2 \ ,
\end{multline}
for some constants $C_1,\al$ depending only on $T$.  Multiplying  (\ref{BU2}) by $(\tau^*_{\ve,x,T})^{1/2+\nu}$  and taking the expected value,  we conclude from the limit of (\ref{BQ2}) as $K\ra\infty$  that
\begin{multline} \label{BV2}
yE\left[\left(\frac{1}{\tau^*_{\ve,x,T}}\right)^{1/2-\nu}; \ \del<\tau^*_{\ve,x,T}<T/2 \right] \ \le \\ 
 C_1\left[T^{1/2+\nu}+\sqrt{\ve} \ E\left[ \tau_\del^{1/2+\nu} |Z(\tau_\del)|\right] \ \right]
 +\frac{C_1T^{1/2+\nu}q_\ve(x,y,T)}{\al y} \ .
\end{multline}

Let $\tau$ be a stopping time for the martingale $s\ra Z(s), \ 0<s<T$, of (\ref{P2}).  Then for $\nu>0$ there is a constant $C_\nu$, depending on $\nu$, such that
\be \label{BW2}
E\left[\tau^{1/2+\nu}; \ |Z(\tau)|>a \ \right] \ \le  \ \frac{C_\nu}{a^{1+2\nu}}   \quad {\rm for \ } a>0 \ .
\ee
To show (\ref{BW2}), let $\tau_a=\sup\{s: \ 0<s<T, \ |Z(s)|\ge a \ \}$. Then any stopping time $\tau$ for $Z(\cdot)$ has the property
\be \label{BX2}
\{\tau>s, \ |Z(\tau)|>a\} \ \subset \ \{\tau_a>s\} \ ,
\ee
since $\{\tau_a>s\}$ is the largest Borel set in $\mathcal{F}_{s,T}$, the $\sig-$field generated by $B(s'), \ s<s'<T,$ on which $\sup_{s<s'<T}|Z(s')|>a$.  Hence
\begin{multline} \label{BY2}
E[\tau^{1/2+\nu}; \ |Z(\tau)|>a] \ = \ \left(\frac{1}{2}+\nu\right)\int_0^T s^{\nu-1/2}P\left(   \tau>s, \ |Z(\tau)|>a  \right) \ ds \\
 \le \ \left(\frac{1}{2}+\nu\right)\int_0^T s^{\nu-1/2} P(\tau_a>s) \ ds \ = \ E[\tau^{1/2+\nu}_a] \ .
\end{multline}
The reflection principle applies to $Z(\cdot)$ and also ${\rm Var}[Z(s)]$ satisfies an inequality (\ref{AU2}) with $\tau=0$.  We have therefore from (\ref{BB2}) that
\be \label{BZ2}
P(\tau_a>n/a^2) \ = \ P\left(\sup_{n/a^2<s<T}|Z(s)|\ge a\right) \ \le \ \frac{2C_2}{\sqrt{n}}e^{-c_2n} \ , \quad n=1,2,\dots
\ee
The inequality (\ref{BW2}) follows now from (\ref{BY2}), (\ref{BZ2}). 

We easily see from (\ref{BW2}) that $E\left[ \tau_\del^{1/2+\nu} |Z(\tau_\del)|\right]$ is bounded by a constant, uniformly in $\del$ as $\del\ra 0$.  Thus for any $a>0$, 
\begin{multline} \label{CA2}
E\left[\tau^{1/2+\nu}_\del |Z(\tau_\del)| \ \right] \ \le \ aE\left[\tau^{1/2+\nu}_\del  \ \right] \\+\sum_{n=0}^\infty 2^{n+1} aE[\tau_\del^{1/2+\nu}; \ |Z(\tau_\del)|>2^na]  \
\le \ aT^{1/2+\nu}+\frac{C_{1,\nu}}{a^{2\nu}}  \ ,
\end{multline}
where $C_{1,\nu}$ depends on $\nu>0$, but not on $\del$. 
Choosing $a$ to minimize the RHS of (\ref{CA2}) we conclude that $E\left[ \tau_\del^{1/2+\nu} |Z(\tau_\del)|\right]\le C_{2,\nu}T^\nu$ for some constant depending only on $\nu$.  Hence the RHS of (\ref{BV2}) is bounded by a constant independent of $\del$. Letting $\del\ra 0$ we conclude that $\tau^*_{\ve,x,T}>0$ with probability $1$ and (\ref{BO2}) holds. 

The inequality (\ref{BP2}) follows from (\ref{BQ2}) and the monotone convergence theorem by letting $K\ra\infty$ first and then $s\ra0$. 
\end{proof}
\begin{lem}
For  $x,y,T$  positive  one has
$\liminf_{\ve\ra 0}[q_\ve(x,y,T) -q_0(x,y,T)]=0$.
\end{lem}
\begin{proof}
Let $X^*_\ve(s), \ s\le T,$ be a solution to (\ref{Y2}) with $\mu_\ve=\mu^*_\ve$ and terminal condition $X^*_\ve(T)=x$. We associate with $X^*_\ve(\cdot)$ the differentiable path $X^*_{\ve,c}(\cdot)$ defined by 
\be \label{CB2}
\frac{dX^*_{\ve,c}(s)}{ds} \ = \ \la(X^*_{\ve,c}(s),y,s)+\left[\mu^*_\ve(X^*_\ve(s),y,s)-\la(X^*_\ve(s),y,s)\right] \ , \quad s<T\  ,
\ee
with terminal condition $X^*_{\ve,c}(T)=x$. From (\ref{O2}), (\ref{Y2}), (\ref{CB2}) we see that
\be \label{CC2}
d\left\{X^*_\ve(s)-X^*_{\ve,c}(s)\right\} \ = \ \left[A(s)+\frac{1}{\sig^2_A(s)}\right]\left\{X^*_\ve(s)-X^*_{\ve,c}(s)\right\} +\sqrt{\ve} \ dB(s) \ , \quad s<T \ ,
\ee
with zero terminal condition at $s=T$.  Comparing (\ref{CC2}) to (\ref{N2}), we conclude from (\ref{P2})  that
\be \label{CD2}
X^*_{\ve,c}(s) \ =  \ X^*_\ve(s)+\sqrt{\ve}\frac{\sig^2_A(s)}{m_{1,A}(s)} Z(s) \ , \quad s<T \ , 
\ee
with $Z(\cdot)$ as given in (\ref{P2}). 
We define a classical action which generalizes (\ref{AB2}).  Thus for $x,y,T>0$ and $z\in\mathbb{R}$ we define $q_0(x,y,z,T)$ just as in (\ref{AB2}) but with terminal condition $x(\tau)=z$ instead of $x(\tau)=0$, and without the positivity constraint $x(\cdot)>0$.  It follows from (\ref{CB2}), (\ref{CD2}) and (\ref{BP2}) of Lemma 2.3 that 
\be \label{CE2}
q_\ve(x,y,T) \ \ge \ E\left[q_0(x,y,\sqrt{\ve}Z_\ve,T)\right]  \ , \quad Z_\ve=\frac{\sig^2_A(\tau^*_{\ve,x,T})}{m_{1,A}(\tau^*_{\ve,x,T})} Z(\tau^*_{\ve,x,T}) \ ,
\ee
where $Z(\cdot)$ is as in (\ref{P2}). 

We have already observed  that a minimizing $\tau$ in (\ref{AB2}) satisfies $0<\tau<T$. It follows from this there exist constants $C,\del>0$, depending on $x,y,T$, such that if $|z|<\del$ then $|q_0(x,y,z,T)-q_0(x,y,T)|\le C|z|$.
Hence from (\ref{CE2}) we have that
\be \label{CF2}
q_\ve(x,y,T) \ \ge \ q_0(x,y,T)\left[1-P(|Z_\ve|>\del/\sqrt{\ve})\right]-C\sqrt{\ve}E\left[ \ |Z_\ve| \ \right] \ .
\ee
From (\ref{C2}), (\ref{D2}) we see that $|Z_\ve|\le C_1\tau^*_{\ve,x,T}|Z(\tau^*_{\ve,x,T})|$ for some constant $C_1$. Hence from the proof of Lemma 2.3 we have that $ E\left[ \ |Z_\ve| \ \right] \le C_2$ for some constant 
$C_2$ independent of $\ve$ as $\ve\ra0$. We conclude from (\ref{CF2}) and the Chebyshev inequality that 
$q_\ve(x,y,T)  \ge  q_0(x,y,T)-C_3\sqrt{\ve}$ for some constant $C_3$. 
\end{proof}
We summarize the main result of this section:
\begin{theorem}
Assume $A(\cdot)$ is continuous and the function $q_\ve$ is defined by (\ref{R2}), (\ref{U2}). Then For  $x,y,T$  positive  one has  $\lim_{\ve\ra 0}q_\ve(x,y,T)=q_0(x,y,T)$, where the function $q_0$ is defined by (\ref{AB2}). 
\end{theorem}

\vspace{.1in}

\section{Regularity and bounds on the function $q_\ve$}
We first prove a regularity result for the function $(x,y,t,T)\ra G_{\ve,D}(x,y,t,T)$, which will imply the regularity results for the function $(x,y,T)\ra q_\ve(x,y,T)$ we shall need. 
\begin{proposition}
Let $A:[0,\infty)\ra\R$ be a continuous function and $G_{\ve,D}(x,y,t,T), \ x,y,>0, \ 0\le t<T<\infty,$ be the Dirichlet Green's function  for  the PDE (\ref{A1}) with drift (\ref{E1}).  Then  the derivatives
\be \label{A3}
\frac{\pa^n}{\pa x^n}\frac{\pa^m}{\pa y^m} G_{\ve,D}(x,y,t,T)  \quad {\rm    with \ } 0\le n,m\le 2, \ \  n+m\le 3 \ ,
\ee
and
\be \label{B3}
\frac{\pa^k}{\pa t^k} \frac{\pa^l}{\pa T^l}\frac{\pa^n}{\pa x^n}\frac{\pa^m}{\pa y^m} G_{\ve,D}(x,y,t,T)  \quad {\rm    with \ } 0\le k+l, \ k+m, \ l+n\le 1 \ ,
\ee
exist and are continuous in the region $\mathcal{D}=\{(x,y,t,T): \ x,y\ge 0, \ 0\le t<T<\infty\}$. 

Let $G(x,t)$ be the Gaussian distribution with mean $0$ and variance $t$, 
\be \label{C3}
G(x,t) \ = \ \frac{1}{\sqrt{2\pi t}}\exp\left[-\frac{x^2}{2t} \right]\  , \quad x\in\mathbb{R}, \ t>0 \ .
\ee
For any $L_0,T_0>0$ define $\mathcal{D}_{L_0,T_0}$ to be the region $\mathcal{D}_{L_0,T_0}=\{(x,y,t,T): \ 0\le x,y,\le L_0, \ 0\le t<T\le T_0, \ T-t\le L_0^2\}$. Then there is a constant $C(L_0,T_0,\ve)$ such that if $m,n$ satisfy the conditions of (\ref{A3}), then
\be \label{D3}
\left|\frac{\pa^n}{\pa x^n}\frac{\pa^m}{\pa y^m} G_{\ve,D}(x,y,t,T) \right| \ \le \ \frac{C(L_0,T_0,\ve)}{(T-t)^{(n+m)/2}} G(x-y,2\ve(T-t)) \ ,
\ee
for $(x,y,t,T)\in\mathcal{D}_{L_0,T_0}$. Similarly if $k,l,m,n$ satisfies the conditions of (\ref{B3}), then
\be \label{E3}
\left|\frac{\pa^k}{\pa t^k} \frac{\pa^l}{\pa T^l}\frac{\pa^n}{\pa x^n}\frac{\pa^m}{\pa y^m} G_{\ve,D}(x,y,t,T)\right| \ \le \ 
\frac{C(L_0,T_0,\ve)}{(T-t)^{k+l+(n+m)/2}} G(x-y,2\ve(T-t)) \ ,
\ee
for $(x,y,t,T)\in\mathcal{D}_{L_0,T_0}$. 

For any $y>0$ the function $(x,T)\ra\pa^2G_{\ve,D}(x,y,0,T)/\pa x^2,$ with domain $\{(x,T): \ x,T>0\}$ is continuous up to the boundary $x=0$, and is also continuously differentiable in $T$,  twice continuously differentiable in $x$. 
\end{proposition}
\begin{proof}
Since the drift $b(\cdot,\cdot)$ is continuous and satisfies for any $T_0>0$ the bound $\sup\{|\pa b(y,t)/\pa y|: \ y\ge 0, 0\le t\le T_0\}<\infty$ we may apply the perturbation argument of Lemma 3.4 of \cite{cg}. From this we see that the derivatives (\ref{A3}) with $n\le 1,m\le 2$ are continuous and satisfy the inequality (\ref{D3}). In making this conclusion we are using the backwards in time PDE (\ref{A1}). Since the adjoint PDE (\ref{I1}) is similar to (\ref{A1}) except run forwards in time, we  conclude that (\ref{D3}) also holds with $n\le 2,m\le 1$.  The continuity of the derivatives (\ref{B3}) and bounds (\ref{E3}) follow from the fact that $G_{\ve,D}$ is a solution to the PDEs (\ref{A1}), (\ref{I1}). Hence the derivatives in (\ref{B3}) can be expressed as a sum of the derivatives in (\ref{A3}). 

Letting $v_\ve(x,T)=\exp\left[2\int_0^T A(s) \ ds\right]\pa^2G_{\ve,D}(x,y,0,T)/\pa x^2$, we see by differentiating twice the PDE (\ref{I1}) that $v_\ve$ is also a solution to (\ref{I1}). It  is also continuous up to the boundary $x=0$. For any $L_0>0$ we consider $v_\ve$ to be a solution to (\ref{I1}) in the region $\{(x,T): \ 0<x<L_0, \ T>0\}$. Let $G_{\ve,D,L_0}(x,y,t,T)$ be the corresponding Dirichlet Green's function.  The function $G_{\ve,D,L_0}$ has the same differentiability properties as $G_{\ve,D}$  given in (\ref{A3})-(\ref{E3}) We also have  the integral representation,
\begin{multline} \label{F3}
v_\ve(x,T) \ = \ \int_0^{L_0} G_{\ve,D,L_0}(x,y,t,T) v_\ve(y,t) \ dy +\ve\int_t^T\frac{\pa G_{\ve,D,L_0}(x,0,s,T)}{\pa y} v_\ve(0,s) \ ds \\
-\ve\int_t^T\frac{\pa G_{\ve,D,L_0}(x,L_0,s,T)}{\pa y} v_\ve(L_0,s) \ ds \ , \quad 0<x<L_0, \ 0\le t<T \ .
\end{multline}
The differentiability properties of the function $v_\ve$ follow from (\ref{F3}) and the differentiability properties of $G_{\ve,D,L_0}$ by differentiating under the integral and using the estimates (\ref{D3}), (\ref{E3}).   Note that while the function $(x,T)\ra v_\ve(x,T)$ itself is continuous up to the boundary $x=0$, our proof does not establish that the derivatives are continuous up to the boundary. 
\end{proof}
In the case $A(\cdot)\equiv 0$ the process $X_\ve(\cdot)$ of (\ref{P2}) is the standard Brownian Bridge (BB) process from $x$ at time $T$ to $y$ at time $0$. In that case
\be \label{G3}
X_\ve(s) \ = \ \frac{sx+(T-s)y}{T}-\sqrt{\ve} \ s\int_s^T \frac{dB(s')}{s'} \ , \quad 0<s<T \ ,
\ee
and from (\ref{G1}) we have that
\be \label{H3}
 P\left(\inf_{0\le s\le T} X_\ve(s)<0 \ \big| \ X_\ve(T)=x\right) \ = \ \exp\left[-\frac{2xy}{\ve T}\right] \ .
\ee
We can obtain a linear upper bound on the function $x\ra q_\ve(x,y,T)$ in the case of non-trivial $A(\cdot)$ by comparing $X_\ve(\cdot)$ to the BB process.
\begin{proposition}
Let $A:[0,\infty)\ra\mathbb{R}$ be continuous and $q_\ve$  be defined by (\ref{R2}), (\ref{U2}).  Then for any $y,T>0$ the function $x\ra q_\ve(x,y,T),\ x\ge0,$ is continuous increasing with $q_\ve(0,y,T)=0$. For any $T_0>0$ there exists a constant $C_A(T_0)$, depending only on $T_0$ and 
$\sup_{0\le t\le T_0}|A(t)|$, such that  $q_\ve$ satisfies the inequality 
\be \label{I3}
\left| q_\ve(x,y,T)-\frac{2m_{1,A}(T)xy}{\sig_A^2(T)}\right| \ \le \ C_A(T_0)Ty  \ , \quad x,y>0, \ 0<T\le T_0 \ .
\ee
\end{proposition}
\begin{proof}
The monotonicity of the function $x\ra q_\ve(x,y,T)$  follows from (\ref{Q2}), (\ref{U2}). Since $X_\ve(s), \   s\le T,$ is the solution to (\ref{N2}), it follows from the non-intersection of paths property that the function $x\ra v_\ve(x,y,T)$ is decreasing. 

To prove (\ref{I3}) we make the change of variable $s\leftrightarrow t$ in (\ref{P2})  defined by 
\be \label{J3}
t^2\frac{ds}{dt} \ = \ g(s)^2 \ , \quad g(s) \ = \ \frac{\sig_A^2(s)}{m_{1,A}(s)} \ , \ \ 0<s<T \ ,  \ s(T)=T \ . 
\ee
Since $\sig_A(\cdot)$ is strictly positive and $\lim_{s\ra 0}g(s)/s=1$, the function $s(\cdot)$ is continuous, strictly monotonic and $\lim_{t\ra 0}s(t)/t=1$. We see that the stochastic integral in (\ref{P2}) becomes
\be \label{K3}
\int_s^T \frac{m_{1,A}(s') \ dB(s')}{\sig^2_A(s')} \ = \ \int_t^T\frac{d\tilde{B}(t')}{t'} \  ,
\ee
where $\tilde{B}(\cdot)$ is a Brownian motion.  We define the stochastic process $\tilde{X}_\ve(t), \ 0<t<T$, by
\be \label{L3}
\tilde{X}_\ve(t) \ = \ \frac{m_{1,A}(s(t))}{\sig^2_A(s(t))}t \ y_{\rm class}(s(t))-\sqrt{\ve}t \int_t^T\frac{d\tilde{B}(t')}{t'}  \ ,
\quad 0<t\le T \ ,
\ee
so that the events $\{\inf_{0<s<T} X_\ve(s)<0\}$ and  $\{\inf_{0<t<T} \tilde{X}_\ve(t)<0\}$ are the same.  Observe that
\be \label{M3}
g_{3,A}(s,T) \ = \ \int_s^T \frac{m_{1,A}(s')^2  \ ds'}{\sig_A^4(s')} \ = \ \frac{m_{1,A}(s)^2}{\sig_A^2(s)}-\frac{m_{1,A}(T)^2}{\sig_A^2(T)} \ = \ \frac{m_{1,A}(s)^2\sig_A^2(s,T)}{\sig_A^2(T)\sig_A^2(s)}  \ .
\ee
From (\ref{J3}) we see that $g_{3,A}(s(t),T)=(T-t)/tT, \ 0<t<T$. We conclude then from (\ref{L2}), (\ref{J3}), (\ref{M3}) that
\be \label{N3}
\frac{m_{1,A}(s(t))}{\sig^2_A(s(t))}t \ y_{\rm class}(s(t)) \ = \ t\frac{m_{1,A}(T)}{\sig^2_A(T)}\left\{x+\frac{g_{2,A}(s(t),T)}{g_{1,A}(s(t),T)}\right\}+(T-t)\frac{y}{T} \ , \quad 0<t<T \ ,
\ee
where $g_{1,A},g_{2,A}$ are defined by (\ref{AM2}), (\ref{AN2}). 

The expression (\ref{AN2}) for $g_{2,A}$ can be simplified by observing that
\begin{multline} \label{O3}
g_{2,A}(s,T) \ = \ \frac{m_{2,A}(s,T)}{m_{1,A}(T)}\left[\frac{\sig_A^2(T)-\sig_A^2(s,T)}{\sig^2_A(s,T)}\right]-\frac{m_{2,A}(s)}{m_{1,A}(s)} \\
= \ \frac{\sig_A^2(T)}{m_{1,A}(T)}\frac{m_{2,A}(s,T)}{\sig^2_A(s,T)}-\frac{m_{2,A}(T)}{m_{1,A}(T)} \ .
\end{multline}
We have then from (\ref{AM2}), (\ref{O3}) that
\be \label{P3}
\frac{g_{2,A}(s,T)}{g_{1,A}(s,T)} \ = \ m_{2,A}(T)- \frac{\sig^2_A(T)m_{2,A}(s)}{m_{1,A}(s,T)\sig^2_A(s)}  \ .
\ee 
We have from (\ref{P3}) that
\be \label{Q3}
\lim_{s\ra T} \frac{g_{2,A}(s,T)}{g_{1,A}(s,T)} \ = \ 0, \quad \frac{d}{ds} \left[\frac{g_{2,A}(s,T)}{g_{1,A}(s,T)}\right] \ = \ -\frac{\sig^2_A(T)}{m_{1,A}(s,T)}\frac{\left[\sig^2_A(s)-m_{2,A}(s)\right]}{\sig^4_A(s)} \ .
\ee
It is easy to see from (\ref{Q3}) that for any $T_0>0$, 
\be \label{R3}
\left|\frac{g_{2,A}(s,T)}{g_{1,A}(s,T)}\right| \ \le \ C_{A,1}(T_0)T(T-s) \ , \quad 0<s<T\le T_0 \ , 
\ee
where the constant $C_{A,1}(T_0)$ depends only on $T_0$ and $\sup_{0\le t\le T_0}|A(t)|$. 

It follows from (\ref{L3}), (\ref{N3}), (\ref{R3}) that
\begin{multline} \label{S3}
 P\left(\inf_{0\le t\le T} \tilde{X}_\ve(t)<0 \ \big| \ \tilde{X}_\ve(T)=x\right) \ \ge \\
 P\left(\inf_{0\le t\le T}\left[ \frac{tx_1+(T-t)y}{T}-\sqrt{\ve} t\int_t^T\frac{d\tilde{B}(t')}{t'} \right]<0\right) \ , 
 \quad x_1=\frac{Tm_{1,A}(T)}{\sig^2_A(T)}\left[x+C_{A,1}(T_0)T^2\right] \ .
\end{multline}
We conclude from (\ref{G3}), (\ref{H3}), (\ref{S3}) that
\be \label{T3}
q_\ve(x,y,T) \ \le \ \frac{2x_1y}{T} \ \le \ \frac{2m_{1,A}(T)xy}{\sig^2_A(T)}+C_{A,2}(T_0)Ty \ , \quad 0<T\le T_0 \ ,
\ee
for some constant $C_{A,2}(T_0)$ depending only on $T_0$ and $\sup_{0\le t\le T_0}|A(t)|$. A similar argument yields a lower bound corresponding to (\ref{T3}), whence (\ref{I3}) follows. 
\end{proof}
For general $A(\cdot)$ one can construct a linear solution to the HJ equation (\ref{AC2}), which is therefore also  a solution to the HJB equation (\ref{V2}). To find it  we set
\be \label{U3}
q_0(x,y,T) \ = \ a(y,T)+b(y,T)x \ .
\ee
Equating the coefficients of $x$ in (\ref{AC2}), we obtain the ODE
\be \label{V3}
\frac{db(y,T)}{dT} \ = \ -\left[A(T)+\frac{1}{\sig_A^2(T)}\right]b(y,T) \ .
\ee
Integrating (\ref{V3}), we conclude that
\be \label{W3}
b(y,T) \ = \ C(y)\frac{m_{1,A}(T)}{\sig^2_A(T)} \quad {\rm for \  a \  constant \ } C(y) {\rm \ depending  \ on \  } y \ . 
\ee
We choose the constant $C(y)$ in (\ref{W3})  so that our linear solution gives the function $x\ra 2xy/T$, corresponding to (\ref{G1}), when $A(\cdot)\equiv 0$.  Thus we choose $C(y)=2y$. 
Equating the terms independent of $x$ in (\ref{AC2}), we obtain the ODE
\be \label{X3}
\frac{da(y,T)}{dT} \ = \ \left[1-\frac{m_{2,A}(T)}{\sig_A^2(T)}+\frac{m_{1,A}(T)y}{\sig_A^2(T)}       \right]b(y,T)-\frac{1}{2}b(y,T)^2 \ .
\ee
With the choice of $C(y)=2y$ in (\ref{W3}), this reduces (\ref{X3}) to the equation.
\be \label{Y3}
\frac{da(y,T)}{dT} \ = \ \left[1-\frac{m_{2,A}(T)}{\sig_A^2(T)}      \right]b(y,T)  \ .
\ee
Integrating (\ref{Y3}) with initial condition $a(y,0)=0$ yields the solution
\be \label{Z3}
a(y,T) \ = \ \frac{2y}{\sig_A^2(T)}\left[m_{1,A}(T)m_{2,A}(T)-\sig_A^2(T)\right] \  .
\ee
We conclude from (\ref{W3})-(\ref{Z3}) that
\be \label{AA3}
q_{\rm linear}(x,y,T) \ = \  \frac{2y}{\sig_A^2(T)}\left[m_{1,A}(T)m_{2,A}(T)-\sig_A^2(T)+m_{1,A}(T)x\right]  
\ee
is a linear solution to (\ref{AC2}). 
We shall show that the linear solution (\ref{AA3}) tightly bounds the function $q_\ve$ defined by (\ref{R2}), (\ref{U2}) when $A(\cdot)$ is non-negative. 
\begin{proposition}
Assume the function $A(\cdot)$ is continuous  non-negative, and let $q_\ve(x,y,T)$ be defined by (\ref{R2}), (\ref{U2}). Then
\be \label{AB3}
\left\{\frac{\pa q_{\rm linear}(x',y,T)}{\pa x'}\Big|_{x'=0}\right\}x  \ \le \ q_\ve(x,y,T) \ \le \ q_{\rm linear}(x,y,T)  \quad {\rm for \ } x,y,T>0 \ ,
\ee
and also
\be \label{AC3}
 q_\ve(x,y,T) \ \le \ -2\la(0,y,T)x  \quad {\rm for \ } x,y,T>0 \ .
\ee
\end{proposition}
\begin{proof}
We see using the formula (\ref{E2}) for $G_\ve(x,y,0,T)$  and the fact that the function $(x,t)\ra G_{\ve,D}(x,y,0,t)$ is a solution to the PDE (\ref{I1}) with drift $b(x,t)=A(t)x-1$,  that the function $(x,T)\ra v_\ve(x,y,T)$ defined by (\ref{R2}) is a solution to the PDE (\ref{S2}). Furthermore, $v_\ve$ satisfies the initial and boundary conditions (\ref{T2}).  Since the function $q_{\rm linear}$  is a solution of (\ref{V2}) it follows that $v_{\ve,1}(x,y,T)=\exp[-q_{\rm linear}(x,y,T)/\ve]$ is   a solution to the PDE (\ref{S2}). From the non-negativity of $A(\cdot)$ we also have that $q_{\rm linear}(0,y,T)\ge 0$ for $T>0$. In addition, one has for $T$ small that  $q_{\rm linear}(x,y,T)\simeq 2xy/T$. We conclude  that $v_{\ve,1}$ satisfies the initial and boundary conditions
\be \label{AD3}
v_{\ve,1}(0,y,T) \ \le \ 1, \ T>0,  \quad v_{\ve,1}(x,y,0) \ = \ 0,  \ x>0 \ .
\ee
Comparing (\ref{T2}) and (\ref{AD3}), we expect that an application of the maximum principle for linear parabolic PDE \cite{pw} implies that $v_{\ve,1}(x,y,T)\le v_\ve(x,y,T)$ for all $x,T>0$, whence the upper bound in (\ref{AB3}). 

In the application of the maximum principle we  need to take account of the fact that the domain $\{x\in\mathbb{R}: x>0\}$ is unbounded, and  also that the drift $\la(x,y,T)$ of (\ref{O2}) becomes unbounded as $T\ra 0$.  To deal with this we apply for any $M,T_0>0, \ 0<\del<T_0,$ the maximum principle to a bounded domain $\mathcal{D}_{M,T_0,\del}=\{(x,T): \ 0<x<M, \ \del<T<T_0\}$ on which the drift is continuous and bounded.  Then we let $M\ra\infty,  \del\ra 0$. 

We first consider the case $M\ra\infty$. It is evident from (\ref{AA3}) that \\
$\lim_{M\ra\infty}\sup_{0<T<T_0}v_{\ve,1}(M,y,T)=0$. It follows from (\ref{I3}) of Proposition 3.2  that also
$\lim_{M\ra\infty}\sup_{0<T<T_0}v_\ve(M,y,T)=0$.
Next we consider the case $\del\ra0$. We see from (\ref{AA3}) that for any $m>0$ then $\lim_{\del\ra0}\sup_{ x\ge m}v_{\ve,1}(x,y,\del)=0$. Observe  from  (\ref{C2}), (\ref{D2}) that since the function $s\ra A(s)$ is continuous at $s=0$  then
\be \label{AE3}
\frac{\sig_A^2(s)}{m_{1,A}(s)} \ = \ s[1+so(s)] \ , \quad \frac{m_{2,A}(s)}{m_{1,A}(s)^{1/2}} \ = \  s[1+so(s)] \   .
\ee
It follows from (\ref{AA3}), (\ref{AE3}) that
\be \label{AF3}
q_{\rm linear}(x,y,T) \ = \  \frac{2xy}{T}+o(T) \quad {\rm as \ } T\ra 0 \ .
\ee
 We conclude from (\ref{AF3})  that
\be \label{AG3}
\lim_{\del\ra0}\sup_{ 0<x\le m}\left\{v_{\ve,1}(x,y,\del)-\exp\left[-\frac{2xy}{\ve \del}\right]\right\}=0 \  .
\ee 
From Proposition 3.2 we see that  $\lim_{\del\ra0}\sup_{ x\ge m}v_\ve(x,y,\del)=0$ and a similar result to (\ref{AG3}) holds for $v_\ve$. We conclude that
\be \label{AH3}
\lim_{\del\ra0}\sup_{ 0<x<\infty}\left\{v_{\ve,1}(x,y,\del)-v_\ve(x,y,\del)\right\}=0 \ .
\ee

From Proposition 3.1 it follows that the function $(x,T)\ra u(x,T)=v_{\ve,1}(x,y,T)-v_\ve(x,y,T)$ is continuously differentiable in $T$ and twice continuously differentiable in $x$ on the domain $\mathcal{D}_{M,T_0,\del}$.  It is also continuous up to the boundaries $x=0, \ x=M, \ T=\del$. It  follows then from the maximum principle (Theorem 2 of  Chapter 3 of \cite{pw})  applied to the solution $u(x,T)$ of (\ref{S2}) that the maximum of  $u(\cdot,\cdot)$ on $\mathcal{D}_{M,T_0,\del}$ occurs on one of those boundaries.  The upper bound in (\ref{AB3}) follows by letting $\del\ra 0, M\ra\infty$ and using (\ref{T2}), (\ref{AD3}), (\ref{AH3}) and the fact that $\lim_{M\ra\infty}\sup_{0<T<T_0}u(M,T)=0$.. 

The lower bound in (\ref{AB3}) can be established similarly. Thus we define the function $v_{\ve ,2}$ by 
\be \label{BI3}
v_{\ve,2}(x,y,T) \ = \ \exp\left[    -  \frac{2m_{1,A}(T)xy}{\ve\sig^2_A(T)}   \right] \  .
\ee
Let $\mathcal{L}$ be the linear differential operator $\mathcal{L}=-\pa/\pa T+\cdots$ such that the PDE (\ref{S2}) is $\mathcal{L}v_\ve=0$.  Then we have from (\ref{BI3}) that
\be \label{BJ3}
\mathcal{L}v_{\ve,2}(x,y,T) \ = \ -\frac{2m_{1,A}(T)y}{\ve\sig^2_A(T)}\left[1-\frac{m_{2,A}(T}{\sig^2_A(T)}\right]v_{\ve,2}(x,y,T) \ . 
\ee
Since $A(\cdot)$ is non-negative the RHS of (\ref{BJ3}) is less than or equal to $0$. Arguing as in the previous paragraphs we also see that
\begin{eqnarray} \label{BK3}
\lim_{M\ra\infty}\sup_{0<T<T_0}[v_\ve(M,y,T)-v_{\ve,2}(M,y,T)] \ &=& \ 0 \ , \\
\lim_{\del\ra0}\sup_{ 0<x<\infty}\left\{v_\ve(x,y,\del)-v_{\ve,2}(x,y,\del)\right\} \ &=&  \ 0 \ . \nonumber
\end{eqnarray}
Setting $u(x,T)= v_\ve(M,y,T)-v_{\ve,2}(M,y,T)$, we see from (\ref{S2}), (\ref{BJ3}) that $\mathcal{L}u(x,T)\ge 0$ for $(x,T)\in \mathcal{D}_{M,T_0,\del}$.
Applying Theorem 2 of Chapter 3 of \cite{pw} again,  we conclude  that $u(\cdot,\cdot)$ takes its maximum on the boundaries $x=0, \ x=M, \ T=\del$ of $\mathcal{D}_{M,T_0,\del}$.  The lower bound in (\ref{AB3}) then follows from (\ref{BK3}) upon letting $\del\ra 0, \ M\ra\infty$. 

The upper bound (\ref{AC3}) also follows in a similar way. We define the function $v_{\ve ,3}$ by 
\begin{multline} \label{BL3}
v_{\ve,3}(x,y,T) \ = \ \exp\left[\frac{2\la(0,y,T)x}{\ve}\right] \\
 = \ \exp\left[    - \frac{2x}{\ve}\left\{1-\frac{m_{2,A}(T)}{\sig^2_A(T)}+ 
 \frac{m_{1,A}(T)y}{\sig^2_A(T)}\right\}   \right] \ .
\end{multline}
Taking derivatives in (\ref{BL3}) we see that
\be \label{BM3}
\mathcal{L}v_{\ve,3}(x,y,T) \ = \ 2x\ve^{-1}A(T)v_{\ve,3}(x,y,T) \  .
\ee
Setting $u(x,T)= v_{\ve,3}(x,y,T) -v_\ve(x,y,T)$, we have from (\ref{S2}), (\ref{BM3}) that  $\mathcal{L}u(x,T)\ge 0$ for $(x,T)\in \mathcal{D}_{M,T_0,\del}$. As with (\ref{BK3}), we have that
\be \label{BN3}
\lim_{M\ra\infty} \sup_{0<T<T_0}u(M,T) \ = \ 0  , \quad \lim_{\del\ra 0}\sup_{0<x<\infty}u(x,\del) \ = \ 0 \ .
\ee
Applying Theorem 2 of Chapter 3 of \cite{pw} once more, and taking the limits $\del\ra0, \ M\ra\infty$ using (\ref{BN3}), then yields the upper bound (\ref{AC3}). 
\end{proof}
\begin{rem}
In the case $A(\cdot)\equiv 0$ the upper and lower bounds in (\ref{AB3}) are identical, yielding the function $q_\ve(x,y,T)$ of (\ref{X2}). A similar situation also occurs when the drift in (\ref{A2}) has the form $b(y,s)=-\ga y$, where $\ga$ is constant.  In that case the solution to (\ref{A2}) is given by
\be \label{BQ3}
Y_\ve(s) \ = \  e^{-\ga(s-t)}y+\sqrt{\ve}\int_t^s e^{-\ga(s-s')} \ dB(s') \ .
\ee
Hence the random variable $Y_\ve(T)$ conditioned on $Y_\ve(t)=y$ is Gaussian with mean $m(T-t)y$  and variance $\ve\sig^2(T-t)$, where
\be \label{BR3}
m(T) \ = \ e^{-\ga T}, \quad  \sig^2(T) \ = \ \frac{1}{2\ga}\left[1-e^{-2\ga T}\right] \ .
\ee
 The whole line Green's function $G_\ve(x,y,t,T)$ is explicitly given by the formula
\be \label{BS3}
G_\ve(x,y,t,T)=\frac{1}{\sqrt{2\pi\ve\sig^2(T-t)}}\exp\left[-\frac{\{x-m(T-t)y\}^2}{2\ve\sig^2(T-t)}\right] \ .
\ee
We again define $q_\ve(x,y,T)$ in terms of the Dirichlet Green's function  by (\ref{R2}), (\ref{U2}).  Then the function $[x,T]\ra q_\ve(x,y,T), \ x,T>0,$ is a solution to (\ref{V2}), (\ref{W2}) with $\la(x,y,T)$ given by 
\be \label{BT3}
\la(x,y,T) \ = \ -\ga x -\ve\frac{\pa}{\pa x}\log G_\ve(x,y,0,T) \ = \ -\ga x+\frac{x-m(T)y}{\sig^2(T)} \ .
\ee
We  may solve (\ref{V2}), (\ref{W2}) in the case of (\ref{BR3}), (\ref{BT3})  by looking for a solution of the form $q_\ve(x,y,T)=a(T)xy$. Then $a(\cdot)$ is given by the formula
\be \label{BU3}
a(T) \ = \ \frac{2m(T)}{\sig^2(T)} \ = \ \frac{2\ga}{\sinh \ga T} \ .
\ee 
The relation with the function $\hat{p}$ of Proposition 20 of \cite{svy} is
\begin{multline} \label{BV3}
\hat{p}(T,x,y) \ = \ \exp\left[\frac{\ga x^2}{\ve}\right]G_{\ve, D}(x,y,0,T) \\
 = \ \exp\left[\frac{\ga x^2}{\ve}\right]G_\ve(x,y,0,T)\left[1-v_\ve(x,y,T)\right] \quad {\rm with \ } \ve=1 \ , 
\end{multline}
where $G_\ve$ is given by (\ref{BS3}), and $v_\ve, q_\ve$ are related by (\ref{U2}). 
\end{rem}
\begin{rem}
The upper bound (\ref{AC3}) suggests that the function $x\ra q_0(x,y,T)$ is concave. To see this  consider solutions to the HJ equation (\ref{AC2})  
with the initial and boundary conditions given by (\ref{W2}).  In view of the boundary condition at $x=0$ we have that $\pa q_0(0,y,T)/\pa T=0$ for $T>0$.  It follows then from the PDE (\ref{AC2}) that 
\be \label{BO3}
\frac{\pa q_0(x,y,T)}{\pa x} \Big|_{x=0} \ = \ -2\la(0,y,T)  \ .
\ee
Letting $\ve\ra 0$ in the inequality (\ref{AC3}) we see that the graph of the function $x\ra q_0(x,y,T)$  lies below the line through the origin with slope (\ref{BO3}). 
\end{rem}
\begin{corollary}
Assume the function $A(\cdot)$ is continuous  non-negative, and let $q_\ve(x,y,T)$ be defined by (\ref{R2}), (\ref{U2}). Then the function $x\ra q_\ve (x,y,T)$ is twice continuously differentiable in $x$ for $x\ge 0$ and  $\pa^2 q_\ve(x,y,T)/\pa x^2\le 0$ at $x=0$ and $y,T>0$. 
\end{corollary}
\begin{proof}
The regularity of $q_\ve$ follows from Proposition 3.1. 
Since $q_\ve(0,y,T)=0$ and the function $x\ra q_\ve(x,y,T)$ is non-negative,  we have using the inequality (\ref{AC3})  that $0\le\pa q_\ve(x,y,T)/\pa x\le -2\la(0,y,T)$ at $x=0$. Observing that $q_\ve$  satisfies the PDE (\ref{V2})  and $\pa q_\ve(x,y,T)/\pa T=0$ at $x=0$, we also have that
\be \label{BP3}
\ve \frac{\pa^2 q_\ve(x,y,T)}{\pa x^2} \ \Big|_{x=0} \ = \   \frac{\pa q_\ve(x,y,T)}{\pa x} \Big|_{x=0} 
\left[2\la(0,y,T)+ \frac{\pa q_\ve(x,y,T)}{\pa x} \Big|_{x=0}\right] \ .
\ee
We conclude from (\ref{BP3}) and our bounds on  $\pa q_\ve(x,y,T)/\pa x$ at $x=0$ that \\
$\pa^2 q_\ve (x,y,T)/\pa x^2\le 0$ at $x=0$. 
\end{proof}

\vspace{.1in}

\section{Estimating solutions of the Hamilton-Jacobi PDE}
In $\S2$ we already observed that the infinite dimensional variational problem (\ref{AB2}) may be reduced to a single variable variational problem in the first hitting time parameter $\tau, \ 0<\tau<T$.  From (\ref{AG2}), (\ref{AH2}) we have that
\be \label{A4}
q_0(x,y,T) \ = \ \min_{0<\tau<T}\frac{\ga(\tau)^2}{2}\int_\tau^T \frac{m_{1,A}(s)^2  \ ds}{\sig_A^4(s)} \ ,
\ee
where the function $\ga(\cdot)$ is defined by (\ref{AK2}).
The RHS of (\ref{A4}) can be expressed in terms of the functions $g_{1,A},g_{2,A},g_{3,A}$ of (\ref{AM2}), (\ref{AN2}), (\ref{M3}).  Thus we have that
\be \label{B4}
q_0(x,y,T) \ = \ \min_{0<\tau<T} \frac{g_{3,A}(\tau,T)}{2} \left[y+g_{1,A}(\tau,T)x+g_{2,A}(\tau,T)\right]^2 \ .
\ee 
One sees from (\ref{AM2}) that $\lim_{\tau\ra 0} g_{1,A}(\tau,T)=0$ and $\lim_{\tau\ra T} g_{1,A}(\tau,T)=\infty$. 
From (\ref{O3}) we see that the function $\tau\ra g_{2,A}(\tau,T), \ 0<\tau<T,$ satisfies $\lim_{\tau\ra 0} g_{2,A}(\tau,T)=0$ and $\lim_{\tau\ra T} g_{2,A}(\tau,T)=[\sig_A^2(T)-m_{2,A}(T)]/m_{1,A}(T)$.  From (\ref{M3}) we see that the function $\tau\ra g_{3,A}(\tau,T), \ 0<\tau<T,$ has the properties  $\lim_{\tau\ra 0} g_{3,A}(\tau,T)=\infty$ and $\lim_{\tau\ra T} g_{3,A}(\tau,T)=0$. It follows then from (\ref{B4}) that $q_0(x,0,T)=q_0(0,y,T)=0$. In the case of $x\ra0$ with fixed $y>0$,  the minimizer $\tau(x,y,T)$  in (\ref{B4})  satisfies $\tau(x,y,T)\ra T$, with the minimum in (\ref{B4}) converging to  $0$.  
In the case of $y\ra0$ with fixed $x>0$,  the minimizer $\tau(x,y,T)$    satisfies $\tau(x,y,T)\ra 0$, with the minimum in (\ref{B4}) also converging to  $0$.  For general $x,y>0$, there may not be a unique minimizer $\tau(x,y,T)$, so one does not expect the function $q_0(x,y,T)$ of (\ref{B4}) to be a $C^1$ solution to the HJ equation  (\ref{AC2}).  Note however from (\ref{B4}) that the function $x\ra \sqrt{q_0(x,y,T)}, \ x>0,$ is concave for all $y,T>0$.  This is a simple consequence of the fact that the function is the minimum of a set of linear functions. Concavity of the function  $x\ra q_0(x,y,T)$ implies concavity of the function $x\ra \sqrt{q_0(x,y,T)}$. We shall prove  concavity of  $x\ra q_0(x,y,T)$  in the case when $A(\cdot)$ is non-negative.   

When $A(\cdot)\equiv0$ the formula (\ref{B4}) becomes
\begin{multline} \label{C4}
q_0(x,y,T) \ = \ \min_{0<\tau<T} \frac{T-\tau}{2\tau T} \left[y+\frac{\tau x}{(T-\tau)}\right]^2 \\
= \ \frac{1}{2T}\left[2xy+\min_{\al>0}\{\al x^2+y^2/\al\} \right]\ = \  \frac{2xy}{T} \ , \quad {\rm with \ } \tau(x,y,T) \ = \ \frac{yT}{x+y}\ .
\end{multline}
In this case the minimization problem (\ref{C4})  is convex in $\al$, but one does not expect for general $A(\cdot)$ that (\ref{B4}) is a convex minimization problem. 

The solution of the variational problem (\ref{AB2}) with {\it fixed} $\tau$ and without the positivity constraint on $x(\cdot)$ is given by the expression on the RHS of (\ref{B4}). In the case when the function $A(\cdot)$ is non-negative this is also the solution to the fixed $\tau$ variational problem {\it with} the positivity constraint on $x(\cdot)$.  We see this by observing that the optimizing trajectory (\ref{AK2}) for the unconstrained problem is positive. This follows from the fact that the functions $s\ra g_{1,A}(s,T)$ and $s\ra g_{2,A}(s,T), \ 0<s<T,$ are increasing if $A(\cdot)$ is non-negative.  The monotonicity of $g_{1,A}$ follows by noting that it may be written as
\be \label{D4}
g_{1,A}(s,T) \ = \ \frac{1}{m_{1,A}(T)}\left[   \frac{\sig_A^2(T)}{\sig^2_A(s,T)}-1        \right]\ .
\ee
To show monotonicity of $g_{2,A}$ we differentiate (\ref{O3}) to obtain the formula
\be \label{E4}
\frac{\pa g_{2,A}(s,T)}{\pa s} \ = \ \frac{\sig_A^2(T)}{m_{1,A}(s)\sig^4_A(s,T)}
\left[m_{1,A}(s,T)m_{2,A}(s,T)-\sig^2_A(s,T)\right] \ .
\ee
It is easy to see that $m_{1,A}(s,T)m_{2,A}(s,T)-\sig^2_A(s,T)\ge 0$ for  $0<s<T$ if the function $A(\cdot)$ is non-negative. 

For fixed $y>0$ and $x$ large, minimizers $\tau(x,y,T)$ for (\ref{B4}) are close to $0$.  We can use this observation to show that when $x$ is large, $q_0(x,y,T)$ is well approximated by $q_{\rm linear}(x,y,T)$ of (\ref{AA3}). To see this first observe from (\ref{M3}) that $\lim_{\tau\ra 0}\tau g_3(\tau,T)=1$.  We also have on differentiating (\ref{D4}) that
 \be\label{F4}
\frac{\pa g_{1,A}(s,T)}{\pa s} \ = \  \frac{\sig^2_A(T)m_{1,A}(s,T)^2}{m_{1,A}(T)\sig^4_A(s,T)} \ .
\ee
Upon setting $s=0$ in (\ref{E4}), (\ref{F4}) we see  that
\be \label{G4}
g_{1,A}(0,T)=0, \ \ \frac{\pa g_{1,A}(0,T)}{\pa s} =\frac{m_{1,A}(T)}{\sig_A^2(T)} \ , \quad g_{2,A}(0,T)=0, \ \ \frac{\pa g_{2,A}(0,T)}{\pa s} =\frac{m_{1,A}(T)m_{2,A}(T)}{\sig_A^2(T)}-1 \  .
\ee
Hence if the minimizer $\tau(x,y,T)$ of (\ref{B4}) is close to $0$, the minimization problem is given to leading order by
\be \label{H4}
q_0(x,y,T)\simeq \frac{1}{2}\left\{\min_{\tau>0}\left[\frac{y}{\sqrt{\tau}}+\sqrt{\tau}\left\{\frac{\pa g_{1,A}(0,T)}{\pa s} x+\frac{\pa g_{2,A}(0,T)}{\pa s} \right\}\right]\right\}^2 \ = \  q_{\rm linear}(x,y,T)  \ . 
\ee
The minimizer in (\ref{H4}) gives the leading order term in an expansion of $\tau(x,y,T)$, whence
\be \label{I4}
\tau(x,y,T) \ \simeq \ \frac{2y^2}{q_{\rm linear}(x,y,T)} \ \simeq\ \frac{Ty}{x} \quad {\rm for \ large \ } x.
\ee
Note from (\ref{I4})  that $\tau(x,y,T)=O(1/x)$ as $x\ra\infty$.  We make this argument precise in the following:
\begin{proposition}
Assume the function $A:[0,\infty)\ra\mathbb{R}$ is continuous and non-negative. Then for any $T_0>0$ there exists a constant $C_A(T_0)$, depending only on  
$T_0$ and $\sup_{0\le t\le T_0}A(t)$, such that   the function $(x,T)\ra q_0(x,y,T)$ satisfies the inequalities
\be \label{J4}
-\frac{C_A(T_0)Ty^2}{x}  \ \le \ q_0(x,y,T)-q_{\rm linear}(x,y,T) \ \le \ 0 \quad {\rm for \ } x\ge \max\{2y,T^2\}, \ 0<T<T_0 \ ,
\ee
and also
\be \label{K4}
0 \ \le \ q_0(x,y,T)-\frac{2m_{1,A}(T)xy}{\sig^2_A(T)} \ \le \ C_A(T_0)Tx \quad {\rm for \ } x,y>0, \ 0<T<T_0 \ . 
\ee
\end{proposition}
\begin{proof}
All the constants $C_1,C_2,...,$ in the following can be chosen to depend only on $T_0$ and $\sup_{0\le t\le T_0}A(t)$.
We first observe from (\ref{AM2}), (\ref{M3}) that
\be \label{L4}
g_{3,A}(s,T)g_{1,A}(s,T) \ = \ \frac{m_{1,A}(T)}{\sig^2_A(T)} \ , \quad 0<s<T \ .
\ee
Next we note from (\ref{AM2}), (\ref{E4})  that for any $T_0>0$, there are constants $C_1,C_2,C_3>0$ such that
\begin{multline} \label{M4}
 \frac{C_1s}{T-s} \ \le \ g_{1,A}(s,T)  \ \le \  \frac{C_2s}{T-s} \ , \quad
 0 \ \le \ \frac{\pa g_{2,A}(s,T)}{\pa s}  \ \le \  C_3T \ ,  \\ 
 0 \ \le \ g_{2,A}(s,T) \ \le C_3sT \ ,  \quad {\rm for \ } 0<s<T, \ 0<T\le T_0 \ .
\end{multline}
 Evaluating the functional on the RHS of (\ref{B4}) at $\tau=Ty/x$ we conclude from (\ref{L4}), (\ref{M4}) that 
\be \label{N4}
q_0(x,y,T) \ \le \  \frac{m_{1,A}(T_0)}{2C_1}[1+2C_2+C_3]^2 \frac{xy}{T} \ , \quad 0<T\le T_0, \ x\ge \max\{2y,T^2\} \ .
\ee
We also have from (\ref{L4}), (\ref{M4}) that
\be \label{O4}
 \frac{g_{3,A}(\tau,T)}{2} \left[y+g_1(\tau,T)x+g_{2,A}(\tau,T)\right]^2 \
  \ge \ \frac{m_{1,A}(T)}{2\sig^2_A(T)} g_{1,A}(\tau,T)x^2 \ \ge \ \frac{C_1\tau x^2}{2m_{1,A}(T_0)T^2} \ .
\ee
It follows from (\ref{N4}), (\ref{O4}) there is a constant $C_4$ such that any minimizing $\tau=\tau(x,y,T)$ in (\ref{B4}) satisfies the inequality
\be \label{P4}
0<\tau(x,y,T) \ \le \ \frac{C_4Ty}{x} \ , \quad 0<T\le T_0, \ x\ge \max\{2y,T^2\} \ .
\ee
We have from (\ref{P3}) that 
\be \label{Q4}
\lim_{\tau\ra 0}\frac{g_{2,A}(\tau,T)}{g_{1,A}(\tau,T)} \ = \ m_{2,A}(T)-\frac{\sig^2_A(T)}{m_{1,A}(T)} \ .
\ee
We also see from (\ref{Q3}) that if $A(\cdot)$ is non-negative, the derivative of the function $s\ra g_{2,A}(s,T)/g_{1,A}(s,T), \ 0<s<T,$ is less than or equal to zero and bounded by a constant times $T$. We conclude therefore from (\ref{Q4}) there is a constant $C_5$ such that
\be \label{R4}
-C_5\tau T \ \le \ \frac{g_{2,A}(\tau,T)}{g_{1,A}(\tau,T)}-\left\{m_{2,A}(T)-\frac{\sig^2_A(T)}{m_{1,A}(T)}\right\} \ \le \ 0 \ , 
\quad 0<\tau<T\le T_0 \ .
\ee

The inequality (\ref{J4}) follows from (\ref{L4}), (\ref{P4}), (\ref{R4}).  Thus from the upper bound in (\ref{R4}) we have that 
\be \label{S4}
q_0(x,y,T) \ \le \ \inf_{0<\tau<T}\frac{g_{3,A}(\tau,T)}{2}\left[y+\frac{1}{g_{3,A}(\tau,T)}\frac{q_{\rm linear}(x,y,T)}{2y}\right]^2  \ = q_{\rm linear}(x,y,T)\ .
\ee
Similarly from the lower bound in (\ref{R4}) we have upon using (\ref{P4})  that
\begin{multline} \label{T4}
q_0(x,y,T) \ \ge \ \inf_{0<\tau<T}\frac{g_{3,A}(\tau,T)}{2}\left[y+\frac{1}{g_{3,A}(\tau,T)}\left\{\frac{q_{\rm linear}(x,y,T)}{2y}-\frac{m_{1,A}(T)}{\sig^2_A(T)}\frac{C_5C_4 T^2y}{x}\right\}\right]^2 \\
= \ q_{\rm linear}(x,y,T)-\frac{m_{1,A}(T)}{\sig^2_A(T)}\frac{2C_5C_4 T^2y^2}{x} \ ,
\end{multline} 
provided the term inside the curly braces is positive. Observe that the lower bound in (\ref{J4}) is trivial if 
$q_{\rm linear}(x,y,T)-CTy^2/x\le 0$, whence we need only consider situations where  $q_{\rm linear}(x,y,T)-CTy^2/x$ is positive. Upon choosing $C$ sufficiently large, we see that the term of (\ref{T4})  inside the curly braces is then positive. Hence we obtain the lower bound (\ref{J4}) for all $x\ge \max\{2y,T^2\}$.  

We can argue similarly to obtain the bound (\ref{K4}). Thus since $g_{2,A}(\cdot,\cdot)$ is non-negative, we have that 
\be \label{U4}
q_0(x,y,T) \ \ge \ \inf_{0<\tau<T}\frac{g_{3,A}(\tau,T)}{2}\left[y+\frac{1}{g_{3,A}(\tau,T)}\frac{m_{1,A}(T)x}{\sig^2_A(T)}\right]^2
 \ = \  \frac{2m_{1,A}(T)xy}{\sig^2_A(T)} \  .
\ee
From (\ref{M4}) we have that $g_{2,A}(s,T)\le C_3T^2, \ 0<s<T\le T_0$, for some constant $C_3$. Hence from (\ref{B4}), (\ref{L4}) we have that
\be \label{V4}
q_0(x,y,T) \ \le \ \inf_{0<\tau<T}\frac{g_{3,A}(\tau,T)}{2}\left[y+C_3T^2+\frac{1}{g_{3,A}(\tau,T)}\frac{m_{1,A}(T)x}{\sig^2_A(T)}\right]^2
 \ = \  \frac{2m_{1,A}(T)x(y+C_3T^2)}{\sig^2_A(T)} \  .
\ee
\end{proof}
We wish also to understand the behavior of $\pa q_0(x,y,T)/\pa x$ and $\pa^2 q_0(x,y,T)/\pa x^2$ as $x\ra\infty$ and as $T\ra0$. In order to do this we will show that the function $[x,T]\ra q_0(x,y,T)$, defined by (\ref{AB2}) or equivalently (\ref{B4}),  can be obtained for $[x,T]$  in a certain domain by the method of characteristics applied to solving the HJ equation (\ref{AC2}) with boundary condition (\ref{AD2}). To implement the method of characteristics, we first  observe from (\ref{AD2}) that $\pa q_0(x,y,T)/\pa T=0$ at $x=0$, whence (\ref{AC2}) yields the formula
\be \label{W4}
u_0(0,y,T) \ = \ \frac{\pa q_0(x,y,T)}{\pa x} \Big|_{x=0} \ = \ -2\la(0,y,T) \ = \  \frac{2m_{1,A}(T)}{\sig^2_A(T)}[y+g_{2,A}(T,T)] \ ,
\ee
where from (\ref{O3}) we define $g_{2,A}(T,T)=\lim_{s\ra T}g_{2,A}(s,T)=[\sig^2_A(T)-m_{2,A}(T)]/m_{1,A}(T)$. Next we differentiate the PDE (\ref{AC2}) with respect to $x$ to obtain the Burgers' equation
\begin{multline} \label{X4}
\frac{\pa u_0(x,y,T)}{\pa T}  + \left[\la(x,y,T)+ u_0(x,y,T)\right]\frac{\pa u_0(x,y,T)}{\pa x} \\
+\left[A(T)+\frac{1}{\sig^2_A(T)}\right]u_0(x,y,T) \  = \ 0 \ ,
\end{multline}
for the function $u_0(x,y,T)=\pa q_0(x,y,T)/\pa x$. We seek to solve (\ref{X4}), with the boundary condition $u_0(0,y,T)=\pa q_0(0,y,T)/\pa x$ given by (\ref{W4}), by using the method of characteristics.  If $s\ra [x(s),s], \ s>\tau$, is a characteristic with initial condition $x(\tau)=0$,  then (\ref{W4}), (\ref{X4}) yield the  ODE initial value problem,
\begin{multline} \label{Y4}
\frac{d}{ds} u_0(x(s),y,s)+\left[A(s)+\frac{1}{\sig^2_A(s)}\right]u_0(x(s),y,s) \  = \ 0 \ ,  \ \ s>\tau,  \\
u_0(x(\tau),y,\tau) \ = \  \frac{2m_{1,A}(\tau)}{\sig^2_A(\tau)}[y+g_{2,A}(\tau,\tau)] \ .
\end{multline}
The solution to (\ref{Y4}) is given by the formula
\be \label{Z4}
u_0(x(s),y,s) \ = \ \frac{2m_{1,A}(s)[y+g_{2,A}(\tau,\tau)]}{\sig^2_A(s)}  \ , \quad s>\tau \ .
\ee
It follows from (\ref{X4}), (\ref{Z4}) that the characteristics are solutions to the ODE initial value problem  
\be \label{AA4}
\frac{dx(s)}{ds} \ = \ \la(x(s),y,s)+\frac{2m_{1,A}(s)[y+g_{2,A}(\tau,\tau)]}{\sig^2_A(s)} \ , \quad s>\tau, \ x(\tau)=0 \ .
\ee
From (\ref{O2}) we see that (\ref{AA4}) is the same as
\be \label{AB4}
\frac{dx(s)}{ds} \ = \ \left[A(s)+\frac{1}{\sig^2_A(s)}\right]x(s) +\frac{m_{1,A}(s)}{\sig^2_A(s)}\left[y+2g_{2,A}(\tau,\tau)-g_{2,A}(s,s)\right] \ , \quad s>\tau, \ x(\tau)=0 \ .
\ee
The general solution to the ODE (\ref{AB4}) is given by the formula
\be \label{AC4}
x(s) \ = \ C\frac{\sig_A^2(s)}{m_{1,A}(s)}-[y+2g_{2,A}(\tau,\tau)]m_{1,A}(s) -m_{2,A}(s) \ ,
\ee
where $C$ is an arbitrary constant.  The constant $C$ is determined for the characteristic  by the initial condition $x(\tau)=0$. In the case  $A(\cdot)\equiv0$ this yields the formula $x(s)=[s/\tau-1]y, \ s>\tau,$ for the characteristic.  

We have that
\be \label{AD4}
\frac{d}{ds} g_{2,A}(s,s) \ =  \frac{A(s)\sig^2_A(s)}{m_{1,A}(s)} \ .
\ee
If we  assume $A(\cdot)$ non-negative, it follows from (\ref{AD4}) that the function $s\ra g_{2,A}(s,s)$ is increasing. This implies  the characteristics that are solutions to (\ref{AB4}) may meet, whence one cannot expect the HJ equation (\ref{AC2}) to have a classical (continuously differentiable) solution.  We can make this more precise by considering characteristics $s\ra x(\tau,s), \ s>\tau>0,$ which are solutions to (\ref{AB4}) with initial condition $x(\tau,\tau)=0$. The first variation   $D_\tau x(\tau,s)=\pa x(\tau,s)/\pa \tau, \ s>\tau>0,$ is from (\ref{AB4}), (\ref{AD4}) the solution to the initial value problem 
\begin{multline} \label{AE4}
\frac{d}{ds}D_\tau x(\tau,s) \ = \ \left[A(s)+\frac{1}{\sig^2_A(s)}\right]D_\tau x(\tau,s) +\frac{2m_{1,A}(s)}{\sig^2_A(s)}  \frac{A(\tau)\sig^2_A(\tau)}{m_{1,A}(\tau)} \ , \ s>\tau \ , \\
D_\tau x(\tau,s)\Big|_{s=\tau} \ = \ -\frac{m_{1,A}(\tau)}{\sig^2_A(\tau)}\left[y+g_2(\tau,\tau)\right] \ .
\end{multline}
We note that (\ref{AE4}) is equivalent to
\begin{multline} \label{AF4}
\frac{d}{ds}\left[ \frac{m_{1,A}(s)}{\sig^2_A(s)}  D_\tau x(\tau,s)\right] \ = \ \frac{2m_{1,A}(s)^2}{\sig^4_A(s)}  \frac{A(\tau)\sig^2_A(\tau)}{m_{1,A}(\tau)} \ , \ s>\tau \ , \\
D_\tau x(\tau,s)\Big|_{s=\tau} \ = \ -\frac{m_{1,A}(\tau)}{\sig^2_A(\tau)}\left[y+g_2(\tau,\tau)\right] \ .
\end{multline}
Since $D_\tau x(\tau,s)<0$ at $s=\tau$ and the derivative on the LHS of (\ref{AF4}) is non-negative, we can have $D_\tau x(\tau,s)=0$ for some $s>\tau$, from which point the solution to (\ref{X4}) cannot be continued by using the method of characteristics. 

When the method of characteristics does apply to obtain the solution of (\ref{X4}) with boundary data (\ref{W4}),  we may obtain a formula for $\pa u_0(x,y,T)/\pa x$ along characteristics similarly to how we obtained (\ref{Z4}) for $u_0(x,y,T)$. To see this first note from (\ref{W4}), (\ref{AD4}) that
\begin{multline} \label{AG4}
\frac{\pa u_0(x,y,T)}{\pa T}\Big|_{x=0} \ = \ \frac{d}{dT} u_0(0,y,T) \\
 = \ -\left[A(T)+\frac{1}{\sig^2_A(T)}\right]u_0(0,y,T)+ 2A(T) \ .
\end{multline} 
Setting $x=0$ in (\ref{X4}) and using (\ref{AG4}) we conclude that
\be \label{AH4}
\frac{\pa u_0(x,y,T)}{\pa x}\Big|_{x=0} \ = \ -\frac{2A(T)\sig^2_A(T)}{m_{1,A}(T)[y+g_{2,A}(T,T)]} \ .
\ee
Differentiating (\ref{X4}) with respect to $x$, we obtain a PDE for $v_0(x,y,T)=\pa u_0(x,y,T)/\pa x$,
\begin{multline} \label{AI4}
\frac{\pa v_0(x,y,T)}{\pa T}  + \left[\la(x,y,T)+ u_0(x,y,T)\right]\frac{\pa v_0(x,y,T)}{\pa x} \\
+v_0(x,y,T)^2+2\left[A(T)+\frac{1}{\sig^2_A(T)}\right]v_0(x,y,T) \  = \ 0 \ .
\end{multline}
From the method of characteristics applied to (\ref{AI4}), we obtain using (\ref{AH4})  the ODE initial value problem
\begin{multline} \label{AJ4}
\frac{d}{ds} v_0(x(s),y,s)+v(x(s),y,s)^2+2\left[A(s)+\frac{1}{\sig^2_A(s)}\right]v_0(x(s),y,s) \  = \ 0 \ ,  \ \ s>\tau,  \\
v_0(x(\tau),y,\tau) \ = \  -\frac{2A(\tau)\sig^2_A(\tau)}{m_{1,A}(\tau)[y+g_2(\tau,\tau)]} \ ,
\end{multline}
where $s\ra x(s), \ s>\tau,$ is the characteristic defined by (\ref{AB4}).  It follows from (\ref{AJ4}) that the function $s\ra 1/v_0(x(s),y,s)$ is a solution to a linear differential equation, whence we conclude that  $1/v_0(x(s),y,s), \ s>\tau,$  is of the form
\be \label{AK4}
\frac{1}{v_0(x(s),y,s)} \ = \ C\frac{\sig_A^4(s)}{m_{1,A}(s)^2}-\sig_A^2(s) \ , \quad s>\tau \ ,
\ee
for some constant $C$. Choosing $C$ in (\ref{AK4}) to satisfy the initial condition (\ref{AJ4}), we have then that
\begin{multline} \label{AL4}
\frac{\pa u_0(x(s),y,s)}{\pa x} \ = \ -K(\tau)\frac{m_{1,A}(s)^2}{\sig_A^4(s)} 
 \Bigg/\left\{1+K(\tau)\left[\frac{m_{1,A}(s)^2}{\sig_A^2(s)}-\frac{m_{1,A}(\tau)^2}{\sig_A^2(\tau)}\right]\right\} 
 \\ {\rm for \ } s>\tau, \quad {\rm where \ } K(\tau) \ = \ \frac{2A(\tau)\sig^6_A(\tau)}{m_{1,A}(\tau)^3[y+g_{2,A}(\tau,\tau)]} \ .
\end{multline}
Since the function $s\ra m_{1,A}(s)/\sig^2_A(s)$ is decreasing, we see  that the formula (\ref{AL4}) for  $\pa u_0(x,y,T)/\pa x$ can blow up to $-\infty$. This is again a consequence of the fact that we cannot in general expect a classical solution to (\ref{W4}), (\ref{X4}) when $A(\cdot)$ is non-negative. 

We assume that $A(\cdot)$ is non-negative. Observe that the condition $x(\tau)=0$ in (\ref{AB4})  implies that the constant $C$ in  (\ref{AC4}) is given by the formula
\be \label{AM4}
C \ = \ \frac{[y+2g_2(\tau,\tau)] m_{1,A}(\tau)^2}{\sig^2_A(\tau)} +\frac{m_{1,A}(\tau)m_{2,A}(\tau)}{\sig^2_A(\tau)} \
= \ \frac{[y+g_2(\tau,\tau)] m_{1,A}(\tau)^2}{\sig^2_A(\tau)} +m_{1,A}(\tau) \ .
\ee
Substituting (\ref{AM4}) into (\ref{AC4}) gives the formula for the characteristic,
\begin{multline} \label{BJ4}
x(s) \ = \ [y+g_2(\tau,\tau)] \left\{\frac{m_{1,A}(\tau)^2}{\sig^2_A(\tau)}\frac{\sig^2_A(s)}{m_{1,A}(s)}-m_{1,A}(s)   \right\} \\
+m_{1,A}(s)m_{1,A}(\tau) \left\{      \frac{\sig^2_A(s)}{m_{1,A}(s)^2} -   \frac{\sig^2_A(\tau)}{m_{1,A}(\tau)^2}    \right\}
-m_{1,A}(s)  \left\{      \frac{m_{2,A}(s)}{m_{1,A}(s)} -   \frac{m_{2,A}(\tau)}{m_{1,A}(\tau)}    \right\} \ .
\end{multline}
The first term on the RHS of (\ref{BJ4}) is bounded below by $c_1\left[s/\tau-1\right][y+g_{2,A}(\tau,\tau)]$ for $0<\tau<s\le T_0$, where the constant $c_1>0$ depends only on $T_0$ and $\sup_{0\le t\le T_0}A(t)$. 
The remaining terms can be expressed as an integral over the interval $[\tau,s]$,
\be \label{BK4}
m_{1,A}(s)\int_\tau^s \left[\frac{m_{1,A}(\tau)}{m_{1,A}(s')}- 1\right]\frac{ds'}{m_{1,A}(s')} \ = \ m_{1,A}(s)f(s) \ .
\ee
The function $s\ra f(s), \ s\ge\tau,$ is decreasing and $f(\tau)=f'(\tau)=0$.  We conclude that the characteristic $s\ra x(\tau,s), \ s>\tau,$ is an increasing function of $s$ for $s>\tau$ such that $s-\tau$ is sufficiently small. However it could decrease for $s$ large, even to $0$.  We see from (\ref{BJ4}), (\ref{BK4}) that
\begin{multline} \label{AN4}
c_1(s-\tau)\left[\frac{[y+g_{2,A}(\tau,\tau)]}{\tau}-C_1(s-\tau)\right]  \ \le \ x(\tau,s) \\
\le \ C_2(s-\tau)\left[\frac{[y+g_{2,A}(\tau,\tau)]}{\tau}-c_2(s-\tau)\right] \ , \quad 0<\tau<s\le T_0 \ ,
\end{multline}
where $c_1,c_2,C_1,C_2>0$ depend only on $T_0$ and $\sup_{0\le t\le T_0}A(t)$. It follows from (\ref{AN4}) that $x(\tau,s)>0$ for $0<\tau<s<\min\left\{\tau+[y+g_{2,A}(\tau,\tau)]/C_1\tau,T_0\right\}$. 

Next we obtain from the variation equation (\ref{AE4}) conditions   that imply  characteristics do not intersect. Thus setting $y(\tau,s)=m_{1,A}(s)D_\tau x(\tau,s)/\sig^2_A(s)$,  we have
from (\ref{AE4}), (\ref{AF4})
\be \label{AQ4}
\frac{\pa}{\pa s} y(\tau,s)\  \le \ \frac{C_3\tau}{s^2} \ ,  \ \tau<s\le T_0 \ , \quad y(\tau,\tau)\le  -C_4\frac{y+g_{2,A}(\tau,\tau)}{\tau^2} \ ,
\ee
for some positive constants $C_3,C_4$ depending only on $T_0$ and $\sup_{0\le t\le T_0}A(t)$.  Integrating (\ref{AQ4}) we conclude that
\be \label{AR4}
y(\tau,s) \ \le \  C_3\left[1-\frac{\tau}{s}\right]-C_4\frac{y+g_{2,A}(\tau,\tau)}{\tau^2} \quad {\rm for \ } \tau<s\le T_0 \ .
\ee
It follows from (\ref{AR4}) that $D_\tau x(\tau,s)<0$ for 
$0<\tau<s<\min\left\{\tau+C_4[y+g_{2,A}(\tau,\tau)]/C_3\tau,T_0\right\}$. 

Let $\La_0>0$ be a constant such that $\La_0<1/C_1,  \ \La_0< C_4/C_3$ and consider the function $T_y(\tau)=\tau+\La_0[ y+g_{2,A}(\tau,\tau)]/\tau, \ 0<\tau\le T_0$.  Evidently  $T_y(\cdot)\ge \tilde{T}_y(\cdot)$, where $\tilde{T}_y$  is the convex function $\tilde{T}_y(\tau)=\tau+\La_0y/\tau$.  The infimum of $\tilde{T}_y$ is attained at $\tau=\sqrt{\La_0 y}$ and $\inf_{0<\tau<\infty}\tilde{T}_y(\tau)=2\sqrt{\La_0 y}$.  Since (\ref{AD4}) implies that  $g_{2,A}(\tau,\tau)\le C\tau^2$, we see that 
if $2\sqrt{\La_0y}\le T_0$ then $\inf_{0<\tau<T_0}T_y(\tau)\le2[1+C']\sqrt{\La_0 y}$, where $C'\ge 0$ depends only on $T_0$ and $\sup_{0\le t\le T_0}A(t)$. 

We wish to identify the largest domain $\mathcal{D}_{y,T_0}$ contained in $\{[x,T]: \ x>0, \ 0<T<T_0\}$ such that  characteristics do not intersect within the domain. We have already seen that if $2\sqrt{\La_0y}\ge T_0$ then we may take $\mathcal{D}_{y,T_0}=\{[x,T]: \ x>0, \ 0<T<T_0\}$, so let us assume that 
$2\sqrt{\La_0y}< T_0$.  Then $\mathcal{D}_{y,T_0}$ contains  $\{[x,T]: \ x>0, \ 0<T<2\sqrt{\La_0y}\}$, so we just need to consider the situation $2\sqrt{\La_0y}<T<T_0$. Then the equation
\be \label{BL4}
\tau+\frac{\La[y+g_{2,A}(T,T)]}{\tau} \ = \ T
\ee
has two solutions provided that  $2\sqrt{\La[y+g_{2,A}(T,T)]}<T$. Since $2\sqrt{\La_0y}<T<T_0$ we may choose $\La_1<\La_0$,  depending only on $T_0$ and $\sup_{0\le t\le T_0}A(t)$, such that $4\sqrt{\La_1[y+g_{2,A}(T,T)]}<T$.  The larger solution to (\ref{BL4}) is  given by the formula
\be \label{BM4}
\tau_{1,\La,y}(T) \ = \ \frac{T}{2}+\frac{T}{2}\left\{1-\frac{4\La [y+g_{2,A}(T,T)]}{T^2}\right\}^{1/2} \ .
\ee
If $\La\le \La_1$ then $\tau_{1,\La,y}(T)$ satisfy the inequality
\be \label{BN4}
 \frac{\La [y+g_{2,A}(T,T)]}{T} \  \le \ T-\tau_{1,\La,y}(T) \ \le \frac{2\La [y+g_{2,A}(T,T)]}{\sqrt{3}T} \ .
\ee
It follows from (\ref{AD4}), (\ref{BN4}) that
\be \label{BO4}
0\le \ g_{2,A}(T,T)-g_{2,A}(\tau_{1,\La,y}(T),\tau_{1,\La,y}(T)) \ \le \  C\La [y+g_{2,A}(T,T)] \ ,
\ee
for some constant $C$ depending only on $T_0$ and $\sup_{0\le t\le T_0}A(t)$. We conclude from (\ref{AN4}), (\ref{BO4})  there is a constant $\La_2$, depending only on $T_0$ and $\sup_{0\le t\le T_0}A(t)$, such that if
\be \label{BP4}
2\sqrt{\La_0y}<T<T_0 \ \ {\rm and \ } 0<x \le \ \frac{\La_2[y+g_{2,A}(T,T)]^2}{T^2}  \  {\rm then} \  [x,T]\in \mathcal{D}_{y,T_0} \ .
\ee

We may make a similar argument to show that when $2\sqrt{\La_0y}<T<T_0$ then $[x,T]\in\mathcal{D}_{y,T_0}$ for $x$ sufficiently large.  In that case we define $\tau_{2,\La,y}$ as the smaller solution to the equation $\tau+\La y/\tau=T$. If  $\La<\La_0$ then 
\be \label{BQ4}
\tau_{2,\La,y}(T) \ = \ \frac{T}{2}-\frac{T}{2}\left\{1-\frac{4\La y}{T^2}\right\}^{1/2} \ .
\ee
If $\La_1$ satisfies $4\sqrt{\La_1y}<T$ then for $\La<\La_1$  we have that
\be \label{BR4}
 \frac{\La y}{T} \  \le \ \tau_{2,\La,y}(T) \ \le \frac{2\La y}{\sqrt{3}T} \ .
\ee
We conclude from (\ref{AN4})  there is a constant $\La_3$, depending only on $T_0$ and $\sup_{0\le t\le T_0}A(t)$, such that if
\be \label{BS4}
2\sqrt{\La_0y}<T<T_0 \ \ {\rm and \ } x \ge \ \La_3T^2  \  {\rm then} \  [x,T]\in \mathcal{D}_{y,T_0} \ .
\ee
We  define the domain $\mathcal{U}_{y,T_0}$ by
\be \label{BT4}
\mathcal{U}_{y,T_0} \ = \ \{[\tau,s]: \ 0<\tau<T_0, \ \tau<s<\min[T_y(\tau),T_0] \} \ .
\ee 
The mapping $[\tau,s]\ra[x(\tau,s),s]$ is a diffeomorphism from $\mathcal{U}_{y,T_0} $ onto a domain
$\mathcal{D}_{y,T_0}$, which has the properties (\ref{BP4}), (\ref{BS4}).   The fact that the mapping is onto follows from the intermediate value theorem since we see from (\ref{AN4}) that $\lim_{\tau\ra 0}x(\tau,s)=\infty$ for all $0<s<T_0$.  It is one-one since $D_\tau x(\tau,s)<0$ for $[\tau,s]\in\mathcal{U}_{y,T_0}$.  
 \begin{proposition}
Assume the function $A:[0,\infty)\ra\mathbb{R}$ is continuous and non-negative.
  For $[x,T]\in \mathcal{D}_{y,T_0}$ let $[\tau,T]\in \mathcal{U}_{y,T_0}$ be such that $x(\tau,T)=x$, and define $q_0(x,y,T)$ by 
\be \label{AS4}
q_0(x,y,T) \ = \  \frac{g_{3,A}(\tau,T)}{2} \left[y+g_{1,A}(\tau,T)x+g_{2,A}(\tau,T)\right]^2 \ .
\ee
Then the function $[x,T]\ra q_0(x,y,T)$ is a $C^1$ solution of the HJ equation (\ref{AC2}) on $\mathcal{D}_{y,T_0}$ and satisfies the boundary condition $\lim_{x\ra 0} q_0(x,y,T)=0, \ 0<T<T_0$. Furthermore, the function $x\ra q_0(x,y,T)$ is $C^2$ on
$\mathcal{D}_{y,T_0}$ and the derivatives $\pa q_0(x,y,T)/\pa x, \ \pa^2 q_0(x,y,T)/\pa x^2$ are given respectively by the formulas  (\ref{Z4}), (\ref{AL4}). 

Also $\tau=\tau(x,y,T)$ in (\ref{AS4})  is the unique minimizer in the variational problems (\ref{AB2}), (\ref{B4}) for $[x,T]$ with $x>0, \ 0<T <T_0$,   in the following regions: 
 (a) all $x>0$ if $2\sqrt{\La_0 y}\ge T$, otherwise  (b) $0<x\le \La y[y+g_{2,A}(T,T)]/T^2$, (c) $x\ge T^2/\La$, where $\La>0$ is chosen sufficiently small depending only on $T_0$ and $\sup_{0\le t\le T_0}A(t)$. Therefore if $[x,T]$ is  in one of the regions (a), (b), (c) the functions (\ref{B4}) and (\ref{AS4}) are identical.
\end{proposition}
\begin{proof}
All constants in the following can be chosen to depend only on $T_0$ and $\sup_{0\le t\le T_0}A(t)$.
To show regularity of the function $[x,T]\ra q_0(x,y,T)$ we first differentiate the formula (\ref{AS4}) with respect to $x$.  The resulting formula for $\pa q_0(x,y,T)/\pa x$ involves $g_{j,A}(\tau,T), \ j=1,2,3,$ and their first derivatives with respect to $\tau$. It also involves $\pa \tau(x,y,T)/\pa x=[D_\tau x(\tau,T)]^{-1}$, which we see from (\ref{AE4}), (\ref{AF4}) is a continuous function of $[\tau,T]$ and hence of $[x,T]$. We conclude  that the function $x\ra q_0(x,y,T)$ is differentiable and the function  $[x,T]\ra \pa q_0(x,y,T)/\pa x$  continuous.  We can make a similar argument to see that the function $T\ra q_0(x,y,T)$ is differentiable and the function  $[x,T]\ra \pa q_0(x,y,T)/\pa T$  continuous.  In that case we need to show the continuity of the function $[x,T]\ra\pa \tau(x,y,T)/\pa T$, which is given by the formula
\be \label{AT4}
\frac{\pa \tau(x,y,T)}{\pa T} \ = \  -\frac{D_Tx(\tau,T)}{D_\tau x(\tau,T)} \ .
\ee 
Evidently $D_Tx(\tau,T)$ is given by the RHS of (\ref{AB4}) with $s=T, x(s)=x$ and hence is a continuous function of $[x,T]$.  We conclude that the function $[x,T]\ra q_0(x,y,T)$ is $C^1$ on $\mathcal{D}_{y,T_0}$.  

To show that the function (\ref{AS4}) is a solution to the HJ equation  (\ref{AC2}), we proceed by the standard method \cite{evans}, writing (\ref{AC2}) as 
\be \label{AU4}
\frac{\pa q_0(x,y,T)}{\pa T} +  H\left(x,y,\frac{\pa q_0}{\pa x},T\right)  \ = \ 0 \ ,
\ee
where the Hamiltonian $H(x,y,p,T)$ is given by the formula
\be \label{AV4}
H(x,y,p,T) \ = \ \la(x,y,T)p+\frac{1}{2}p^2 \ .
\ee
The corresponding Hamiltonian equations of motion are given by
\be \label{AW4}
\frac{dx}{ds} \ = \ \frac{\pa H(x,y,p,s)}{\pa p} \ , \quad \frac{dp}{ds} \ = \ -\frac{\pa H(x,y,p,s)}{\pa x} \ .
\ee
We solve (\ref{AW4}) with initial conditions
\be \label{AX4}
x(\tau) \ = \ 0 \ ,  \quad  H(0,y,p(\tau),\tau) \ = \ 0  \ .
\ee
Note that the initial condition (\ref{AX4}) for $p(\cdot)$ is the same as in (\ref{Y4}).
If we solve the second equation in (\ref{AW4}) with initial condition (\ref{AX4}) we obtain the formula
\be \label{AY4}
p(\tau,s) \ = \ \frac{2m_{1,A}(s)}{\sig^2_A(s)}[y+g_2(\tau,\tau)] \ , \quad s>\tau \ ,
\ee 
corresponding to (\ref{Z4}). 
Taking $p(\tau,s)$ to be given by (\ref{AY4}), the first equation in (\ref{AW4}) becomes identical to the characteristic equation (\ref{AA4}). We define the function $w:\mathcal{U}_{y,T_0}\ra\mathbb{R}$ by  
\begin{multline} \label{AZ4}
\frac{\pa}{\pa s} w(\tau,s) \ = \ -H(x(\tau,s),p(\tau,s),s)+p(\tau,s)\frac{\pa H(x(\tau,s),p(\tau,s),s)}{\pa p}  \\
{\rm for \ } s>\tau  \ ,\quad {\rm and \ initial \ condition \ } w(\tau,\tau) \ = \ 0 \ .
\end{multline}
Then by standard theory the function $q_0(\cdot,y,\cdot)$ defined on $\mathcal{D}_{y,T_0}$ by $q_0(x(\tau,s),y,s)=w(\tau,s)$ is a solution to (\ref{AU4}) and $p(\tau,s)=\pa q_0(x(\tau,s),y,s)/\pa x$. 

We see that
\be \label{BA4}
w(\tau,s) \ = \ \frac{g_{3,A}(\tau,s)}{2} \left[y+g_{1,A}(\tau,s)x(\tau,s)+g_{2,A}(\tau,s)\right]^2  ,
\ee
by verifying that the RHS of (\ref{BA4}) is a solution to the differential equation (\ref{AZ4}), where $x(\tau,s), \ p(\tau,s)$ are given by (\ref{AC4}), (\ref{AM4}), (\ref{AY4}) and with initial condition $w(\tau,\tau)=0$.  To obtain the formula (\ref{AL4}) for $\pa^2 q_0/\pa x^2$  we observe that
\be \label{BB4}
\frac{\pa}{\pa\tau} \frac{\pa q_0(x(\tau,s),y,s)}{\pa x} \ = \ \frac{\pa^2 q_0(x(\tau,s),y,s)}{\pa x^2}\frac{\pa x(\tau,s)}{\pa \tau} \ = \ \frac{\pa p(\tau,s)}{\pa \tau} \ ,
\ee
and use the formulas  (\ref{BJ4}), (\ref{AY4}).  Note that we may choose $\La_0$ sufficiently small, depending only on $T_0$, such that the denominator in the formula (\ref{AL4}) is positive
if $[\tau,s]\in \mathcal{U}_{y,T_0}$.   

Finally  we we consider for which $[x,T]\in\mathcal{D}_{y,T_0}$ that $\tau=\tau(x,y,T)$ is the minimizer for (\ref{B4}).  We apply the standard verification theorem to paths $[x(s),s], \ \tau<s\le T$, with $x(\tau)=0, \ x(T)=x,$ which lie in $\mathcal{D}_{y,T_0}$.  Using the fact that $q_0$ is a $C^1$ solution of (\ref{AU4}), (\ref{AV4}) on  
$\mathcal{D}_{y,T_0}$and $q(0,y,\tau)=0$, we have that
\begin{multline} \label{BC4}
q_0(x,y,T) \ = \ \int_{\tau}^T \frac{d}{ds} q_0(x(s),y,s) \ ds \\
 = \ \int_{\tau}^T \frac{\pa q_0(x(s),y,s)}{\pa x}\left[\frac{dx(s)}{ds}
 -\la(x(s),y,s)\right]-\frac{1}{2}\int_{\tau}^T \left[\frac{\pa q_0(x(s),y,s)}{\pa x}\right]^2 \ ds \ \\
 \le \frac{1}{2}\int_{\tau}^T \left[\frac{dx(s)}{ds}
 -\la(x(s),y,s)\right]^2 \ ds \  .
\end{multline}
If $\tau=\tau(x,y,T)$ and $x(s)=x(\tau,s), \ \tau<s<T,$ then from (\ref{Z4}), (\ref{AA4})  we get equality in (\ref{BC4}).   Let $F_0(x,y,\tau,T)$ be the function on the RHS of (\ref{B4}). We wish to find $[x,T]$ such that  $q_0$, defined by (\ref{AS4}), satisfies $q_0(x,y,T)=\inf_{0<\tau<T}F_0(x,y,\tau,T)$. We shall identify a subset $S_{x,y,T}\subset (0,T)$ such that $F_0(x,y,\tau,T)>F_0(x,y,\tau(x,y,T),T)=q_0(x,y,T)$ for $\tau\in S_{x,y,T}$.  Hence it is necessary only to consider $\tau\in(0,T)-S_{x,y,T}$. 
  We have already observed that the variational problem (\ref{AB2}) with fixed $\tau$ is quadratic and has the unique solution 
(\ref{AK2}) given by $\Ga(\tau,s,T,x)=a(\tau,s,T)x+b(\tau,s,T), \ \tau<s<T,$ where
\begin{eqnarray} \label{BD4}
\sig^2_A(T)a(\tau,s,T) \ &=& \ m_{1,A}(s)\sig^2_A(s,T)\left[g_{1,A}(s,T)-g_{1,A}(\tau,T)\right]  \ ,  \\
\sig^2_A(T)b(\tau,s,T) \ &=& \   m_{1,A}(s)\sig^2_A(s,T)\left[g_{2,A}(s,T)-g_{2,A}(\tau,T)\right]  \ . \nonumber
\end{eqnarray}
In view of the verification result (\ref{BC4}),  if we show that  the path $s\ra \Ga(\tau,s,T,x), \ \tau<s<T$, lies in $\mathcal{D}_{y,T_0}$ when
$\tau\in(0,T)-S_{x,y,T}$, then it follows that $\tau=\tau(x,y,T)$ is the unique minimizer for (\ref{B4}). If $T<2\sqrt{\La_0y}$ then $[x,T]\in\mathcal{D}_{y,T_0}$ for all $x>0$. Since the functions $\tau\ra g_{1,A}(\tau,T), \ g_{2,A}(\tau,T), \ 0<\tau<T,$ are increasing when $A(\cdot)$ is non-negative, we see from (\ref{BD4}) that   the path   $s\ra \Ga(\tau,s,T,x), \ \tau<s<T$,  lies in $\mathcal{D}_{y,T_0}$ when $0<\tau<T$. Hence $\tau=\tau(x,y,T)$ is the unique minimizer for (\ref{B4}) when $x>0$ and $0<T<2\sqrt{\La_0y}$. 

  For $x,y\ge 0$ let $\tilde{q}_0(x,y,T)$ be defined by 
 \be \label{BU4}
 \tilde{q}_0(x,y,T) \ = \  \min_{0<\tau<T} \frac{g_{3,A}(\tau,T)}{2} \left[y+g_{1,A}(\tau,T)x\right]^2 \ .
 \ee
 We have from (\ref{B4}) that
 \be \label{BV4}
 \tilde{q}_0(x,y,T) \ \le \ \inf_{0<\tau<T}F_0(x,y,\tau,T) \ \le \ \tilde{q}_0(x,y+g_{2,A}(T,T),T) \ .
 \ee
 Using the identity (\ref{L4}) we see that the minimizing $\tau$ for the RHS of (\ref{BU4}) is given by 
 \be \label{BW4}
 g_{3,A}(\tau,T) \ = \ \frac{m_{1,A}(T)x}{\sig^2_A(T)y} \  .
 \ee
 Substituting (\ref{BW4}) into the RHS of (\ref{BU4}) yields the formula
 \be \label{BX4}
 \tilde{q}_0(x,y,T) \ = \ \frac{2m_{1,A}(T)xy}{\sig^2_A(T)} \ .
 \ee
 
 It is intuitively clear from (\ref{BV4}) that the function $q_0$ defined by (\ref{AS4}) on $\mathcal{D}_{y,T_0}$ should satisfy the inequality 
 \be \label{CU4}
 q_0(x,y,T) \ \le \ \tilde{q}_0(x,y+g_{2,A}(T,T),T) \  , \quad [x,T]\in \mathcal{D}_{y,T_0} \ .
 \ee
 To see this observe that the function $v(\cdot,\cdot)$ defined by $v(\tau,s)=\tilde{q}_0(x(\tau,s),y+g_{2,A}(s,s),s)-w(\tau,s), \ 0<\tau<s<T_0,$ with $w(\cdot,\cdot)$ as in (\ref{BA4}) satisfies the differential equation
 \be \label{CV4}
 \frac{\pa v(\tau,s)}{\pa s} \ = \ 2A(s)x(\tau,s)-\frac{2m_{1,A}(s)^2}{\sig^4_A(s)}[g_{2,A}(s,s)-g_{2,A}(\tau,\tau)]^2 \ .
 \ee
 In deriving (\ref{CV4}) we use the fact that $w(\tau,s)=q_0(x(\tau,s),y,s)$ and that the function $[x,T]\ra q_0(x,y,T)$ is a solution to the HJ equation (\ref{AC2}). 
 Integrating  (\ref{CV4}) over an interval $[\tau,T]$ with $0<\tau<T\le T_0$, we obtain using (\ref{AD4}), (\ref{AN4}) the inequality,
 \be \label{CW4}
 v(\tau,T) \ \ge \ c_1\frac{y+g_{2,A}(\tau,\tau)}{\tau}\int_\tau^T(s-\tau)A(s) \ ds-2\int_\tau^T\left[\int_\tau^s A(s') \ ds'\right]^2 \ ds \ ,
 \ee
 provided $0<s-\tau\le [y+g_{2,A}(\tau,\tau)]/2C_1\tau$.  Using the inequality
 \be \label{CX4}
 \left[\int_\tau^s A(s') \ ds'\right]^2 \ \le \ 2\left[\sup_{\tau<s'<s}A(s')\right]\int_\tau^s(s'-\tau)A(s') \ ds' \ ,
 \ee
 we see that the RHS of (\ref{CW4}) is non-negative for $0<T-\tau\le [y+g_{2,A}(\tau,\tau)]/C_2\tau$,
 provided  $C_2$ is chosen sufficiently large. 
 
  Consider now the function $f(\cdot)$ defined by
 \be \label{BY4}
  f(\la) \ = \ \frac{\la}{2}[y_1+x_1/\la]^2 \ ,  \quad {\rm where} \  x_1,y_1>0.
 \ee
 Evidently the minimizing $\la=\la_{\rm min}$ and minimizer  are given by
 \be \label{BZ4}
 \la_{\rm min} \ = \ x_1/y_1, \quad f(\la_{\rm min}) \ = \  2x_1y_1 \ .
 \ee
 We have furthermore that
 \be \label{CA4}
 f(\la_{\rm min}/8) \ = \  f(8\la_{\rm min}) \ \ge \ 5x_1y_1 \ .
 \ee
 
 We consider $[x,T]\in\mathcal{D}_{y,T_0}$ which satisfies (\ref{BP4}). Let us assume now that $g_{2,A}(T,T)\le y$. Then we have from (\ref{CU4}) that
 \be \label{CB4}
 q_0(x,y,T) \ \le \ 2\tilde{q}_0(x,y,T) \ .
 \ee
 Suppose $\tau\in(0,T)$ lies outside the region
 \be \label{CC4}
 \frac{m_{1,A}(T)x}{8\sig^2_A(T)y}  \ \le \ g_{3,A}(\tau,T) \ \le \ \frac{8m_{1,A}(T)x}{\sig^2_A(T)y}  \ .
 \ee
 Then from (\ref{BW4}), (\ref{BZ4}), upon setting $y_1=y$ and $x_1=m_{1,A}(T)x/\sig^2_A(T)$ in (\ref{CA4}), we obtain the inequality
 \be \label{CD4}
  F_0(x,y,\tau,T)   \ \ge \  5\tilde{q}_0(x,y,T)/2  \ .
 \ee
 It follows from (\ref{CB4}), (\ref{CD4}) that $\tau\in S_{x,y,T}$ if $\tau$ does not satisfy (\ref{CC4}).  Observe from (\ref{BP4}) that $x\le 4\La_2y^2/T^2$,
 which implies that $Tx/y\le 4\La_2y/T\le\La_2T/\La_0$.  Using (\ref{L4}), (\ref{M4}),  (\ref{CC4})  it follows on choosing $\La_2$ sufficiently small, that  $\tau\in(0,T)-S_{x,y,T}$ satisfies  inequalities 
 \be \label{CE4}
 \frac{T}{2}<\tau<T \quad {\rm and \ } C_1\frac{Tx}{y} \ \le \ T-\tau \ \le \ C_2\frac{Tx}{y} \  ,
 \ee 
for some constants $C_1,C_2>0$.  We have from (\ref{E4}), (\ref{F4}), (\ref{BD4})  there are constants $c_3,C_3>0$, depending only on $T_0$, such that
\begin{multline} \label{CF4}
c_3(s-\tau)\frac{x}{T-\tau}  \ \le \ \Ga(\tau,s,T,x) \\
 \le \ C_3(s-\tau)\left\{\frac{x}{T-\tau}+T-s\right\} \ , \quad 0<\tau<s<T \ .
\end{multline}
Hence  we have from (\ref{CE4}), (\ref{CF4}) that if $\tau\in(0,T)-S_{x,y,T}$  then
\be \label{CG4}
\Ga(\tau,s,T,x) \ \le \ C_3\left\{x+\frac{1}{4}(T-\tau)^2\right\} \ \le \ C_3 x\left\{1+\frac{C^2_2T^2x}{4y^2}\right\} \ \le \ C_3 x\left\{1+C^2_2\La_2\right\} \ .
\ee
We conclude there exists $\La>0$, depending only on $T_0$, such that if $x\le \La y^2/T^2$ and $\tau\in(0,T)-S_{x,y,T}$  then the path $s\ra \Ga(\tau,s,T,x), \ \tau<s<T$,  lies in $\mathcal{D}_{y,T_0}$. It follows that $\tau=\tau(x,y,T)$ is the unique minimizer for the function $\tau\ra F_0(x,y,\tau,T), \ 0<\tau<T,$ when $x\le \La y^2/T^2$.

Next we consider the case $g_{2,A}(T,T)> y$ and consider $[x,T]\in\mathcal{D}_{y,T_0}$ which satisfies (\ref{BP4}).  We may estimate $\tau(x,y,T)$ from (\ref{AN4}). Thus we have from (\ref{AN4}), (\ref{BO4}) that
\be \label{CH4}
\frac{Tx}{2C_2[y+g_{2,A}(T,T)]} \ \le\  T-\tau(x,y,T) \ \le \ \frac{2Tx}{c_1[y+g_{2,A}(T,T)]} \quad {\rm if \ } \ x\le \ \frac{\La[y+g_{2,A}(T,T)]^2}{T^2} \ ,
\ee
provided $\La\le \La_2$ is chosen sufficiently small.  From (\ref{L4}), (\ref{BO4}), (\ref{CU4}) we see that (\ref{CH4}) implies
\be \label{CI4}
\tilde{q}_0(x,\{y+g_{2,A}(T,T)\}/3,T) \ \le \ F_0(x,y,\tau(x,y,T),T) \ \le \ \tilde{q}_0(x,y+g_{2,A}(T,T),T) \ ,
\ee
if $\La$ is sufficiently small.  We have already observed that if $q_0$ is defined by (\ref{B4}) then $\lim_{y\ra 0}q_0(x,y,T)=0$.  
It follows  from the lower bound (\ref{CI4}) that $\inf_{0<\tau<T}F_0(x,y,\tau,T)\ne F_0(x,y,\tau(x,y,T),T)$ if $y$ is sufficiently small. 

We show  for sufficiently small $\La>0$  that
\be \label{CO4}
{\rm if \ } x\le \ \frac{\La y[y+g_{2,A}(T,T)]}{T^2} \quad {\rm then \ } \inf_{0<\tau<T}F_0(x,y,\tau,T)= F_0(x,y,\tau(x,y,T),T) \ .
\ee
To see this observe from (\ref{E4}) that
 \be \label{CJ4}
\frac{c_3T}{(T-s)^2}\int_s^T(T-s')A(s') \ ds' \ \le \frac{\pa g_{2,A}(s,T)}{\pa s} \ \le \ \frac{C_3T}{(T-s)^2}\int_s^T(T-s')A(s') \ ds' \ ,
 \ee 
 where $c_3,C_3$ are constants. Similarly we have from (\ref{O3}) that
 \be \label{CK4}
c_4\int_0^T sA(s) \ ds \ \le \  g_{2,A}(T,T) \ \le \ C_4\int_0^T sA(s) \ ds \  .
 \ee
We define $\del, \ 0<\del<1,$ by
\be \label{CL4}
\int_0^{(1-\del)T} sA(s) \ ds \ = \ \frac{1}{2}\int_0^T sA(s) \ ds \ .
\ee
It follows from (\ref{CK4}), (\ref{CL4}) that $\del$ satisfies the inequality
\be \label{CR4}
\del \ \ge \ \frac{g_{2,A}(T,T)}{2\|A\|_\infty C_4T^2} \  .
\ee
Integrating (\ref{CJ4}) over the interval $0<s<\tau$ we have from the lower bound the inequality,
\be \label{CM4}
g_{2,A}(\tau,T) \ \ge \ c_3\left[\int_0^\tau sA(s) \ ds+\frac{\tau}{T-\tau}\int_\tau^T(T-s)A(s) \ ds\right] \ .
\ee

We assume first that $\del\ge1/2$, whence   (\ref{CK4}), (\ref{CM4}) imply that $g_{2,A}(\tau,T)\ge c_3g_{2,A}(T,T)/2C_4$ if $\tau\ge T/2$.    We have then from 
(\ref{L4}), (\ref{M4}) that
\be \label{CP4}
F_0(x,y,\tau,T) \ \ge \  \frac{g_{3,A}(\tau,T)g_{2,A}(\tau,T)y}{2}  \ \ge \ \frac{  c_5(T-\tau)g_{2,A}(T,T)y}{T^2}  \quad {\rm if \ } \tau \ge T/2 \ ,
\ee
for some constant $c_5>0$.  If $0<\tau<T/2$ then we again see from (\ref{CK4}), (\ref{CM4}) that 
$g_{2,A}(\tau,T)\ge  c_3\tau g_{2,A}(T,T)/2C_4(T-\tau)$. In this case we obtain the inequality 
\be \label{CQ4}
F_0(x,y,\tau,T) \ \ge \ \frac{  c_6g_{2,A}(T,T)y}{T}  \quad {\rm if \ } \tau < T/2 \ ,
\ee
where $c_6>0$ is constant. It is easy to see from (\ref{BX4}), the upper bound (\ref{CI4}) and (\ref{CQ4}) that if $0<\tau<T/2$ then  $F_0(x,y,\tau,T)>F_0(x,y,\tau(x,y,T),T)$ provided $x$ satisfies (\ref{CO4}) with $\La>0$ sufficiently small.  Similarly from (\ref{CP4}) we conclude that 
$F_0(x,y,\tau,T)>F_0(x,y,\tau(x,y,T),T)$ if $T-\tau\ge C_7Tx/y$, for some constant $C_7$.   Hence if we show that the paths $s\ra \Ga(\tau,s,T,x), \ \tau<s<T,$ lie in $\mathcal{D}_{y,T_0}$ if $T-\tau< C_7Tx/y$ then it follows that $\inf_{0<\tau<T}F_0(x,y,\tau,T)=F_0(x,y,\tau(x,y,T),T)$. We see from (\ref{BP4}) (\ref{CF4}), (\ref{CG4}) that this is the case provided $x$ satisfies (\ref{CO4}) with $\La>0$ sufficiently small. 

Next we assume that $\del<1/2$, whence we have from (\ref{CL4}), (\ref{CM4}) that
\be \label{CS4}
g_{2,A}(\tau,T) \ \ge \ c_3\frac{\tau\del}{2(T-\tau)}\int_0^TsA(s) \ ds \  \quad {\rm if \ } T-\tau\ge \del T \ .
\ee
Then using (\ref{L4}), (\ref{M4}) again together  with (\ref{CK4}), (\ref{CR4}) we conclude from (\ref{CS4}) the inequality
\be \label{CT4}
F_0(x,y,\tau,T) \ \ge \ c_8\frac{g_{2,A}(T,T)^2y}{T^3} \ \quad {\rm if \ } T-\tau\ge \del T \ ,
\ee
where $c_8>0$ is constant.  It follows from (\ref{BX4}), the upper bound (\ref{CI4}) and (\ref{CT4}) that if $T-\tau\ge \del T$ then  $F_0(x,y,\tau,T)>F_0(x,y,\tau(x,y,T),T)$ provided $x$ satisfies (\ref{CO4}) with $\La>0$ sufficiently small.  If $T-\tau<\del T$ then
$g_{2,A}(\tau,T)\ge c_3g_{2,A}(T,T)/2C_4$, whence the inequality (\ref{CP4}) holds provided $T-\tau<\del T$.  We  may argue now as in the previous paragraph to conclude that  $\inf_{0<\tau<T}F_0(x,y,\tau,T)=F_0(x,y,\tau(x,y,T),T)$ if $x$ satisfies (\ref{CO4}) with $\La>0$ sufficiently small. 

Finally we show that for  $2\sqrt{\La_0y}<T$ and $\La>0$  sufficiently small,  if  $x\ge T^2/\La$ then $\inf_{0<\tau<T}F_0(x,y,\tau,T)=F_0(x,y,\tau(x,y,T),T)$. 
Observe that in this case $x\ge\max\{2y,T^2\}$, whence (\ref{P4}) gives a bound on  any minimizing $\tau$ for the function $\tau\ra F_0(x,y,\tau,T), \ 0<\tau<T$. Since $x\ge T^2/\La$, this  implies $0<\tau\le C_4\La y/T$. Hence if we show that all paths  $s\ra \Ga(\tau,s,T,x), \ \tau<s<T,$ lie in $\mathcal{D}_{y,T_0}$ if $0<\tau\le C_4\La y/T$ it follows that $\inf_{0<\tau<T}F_0(x,y,\tau,T)=F_0(x,y,\tau(x,y,T),T)$. We see from (\ref{AN4}) there exists $\La_3>0$ such that that if 
$0<\tau\le \La_3 y/T=\tau^*$ then the characteristic $ s\ra x(\tau,s), \ \tau<s\le T,$ lies in $\mathcal{D}_{y,T_0}$ Furthermore  $x(\tau^*,s)\le C_5T(s-\tau)$, where $C_5$ is constant.  Since $x\ge T^2/\La$ the lower bound (\ref{CF4}) implies that $\Ga(\tau,s,T,x)\ge c_3T(s-\tau)/\La,  \tau<s<T$. We conclude that if
$\La<c_3/C_5$ then the path $s\ra \Ga(\tau,s,T,x), \ \tau<s<T,$ lies in $\mathcal{D}_{y,T_0}$.
\end{proof}
\begin{rem}
The interval (b) in the statement of Proposition  4.2  is more or less an optimal interval for which the method of characteristics yields the minimizer  in (\ref{B4}). 
One can see this  by choosing $A(\cdot)$ to have support in a small neighborhood of $T$. Thus for any $\del, \ 0<\del<1/2,$ we set the function $A(\cdot)=A_\del(\cdot)$, where $A_\del(s)=A(T)[1-(T-s)/\del T]$ for $0\le T-s\le \del T $ and $A_\del(s)=0$ for  $T-s>\del T$.  Then the ratio $g_{2,A}(\tau,T)/g_{2,A}(T,T)\simeq 1/N$ if $T-\tau\simeq N\del T$. 
 \end{rem}
\begin{corollary}
Assume the function $A:[0,\infty)\ra\mathbb{R}$ is continuous and non-negative. Then for any $T_0>0$, there exist  constants $C_1,C_2>0$, depending only on $T_0$ and $\sup_{0\le t\le T_0}A(t)$, such that   the function $(x,T)\ra q_0(x,y,T)$ defined by (\ref{AS4}) satisfies the inequalities
\begin{multline} \label{BF4}
0 \le  \frac{\pa q_0(x,y,T)}{\pa x}-\frac{2m_{1,A}(T)y}{\sig^2_A(T)} \le  \frac{C_1Ty^2}{x^2}  , \quad
-\frac{C_1Ty^2}{x^3}  \le  \frac{\pa^2 q_0(x,y,T)}{\pa x^2}  \le  0 \ ,
\\ {\rm for \ }  [x,T]\in\mathcal{D}_{y,T_0}, \ x\ge C_2\max\{2y,T^2\}, \ 0<T<T_0 \ .
\end{multline}
In addition for any $T_0>0$ there is a constant $C>0$ such that
\begin{multline} \label{BG4}
0  \le  \frac{\pa q_0(x,y,T)}{\pa x}-\frac{2m_{1,A}(T)y}{\sig^2_A(T)}  \le  CT , \quad  -\frac{CT}{y}  \le \frac{\pa^2 q_0(x,y,T)}{\pa x^2}\le 0,
\\ {\rm for \ }  [x,T]\in\mathcal{D}_{y,T_0},  \ x,y>0,   \  0<T<T_0 \ .
\end{multline}
\end{corollary}
\begin{proof}
Since  $[x,T]\in\mathcal{D}_{y,T_0}$ we may  use the formulas (\ref{Z4}), (\ref{AL4})  to show (\ref{BF4}), (\ref{BG4}).  For (\ref{BF4}) we use the inequality (\ref{P4}) for $\tau(x,y,T)$.  The first inequality follows   from (\ref{Z4})  and the fact that $0\le g_{2,A}(\tau,\tau)\le C\tau^2$, where $C$ is constant.  To obtain the second inequality we observe that the function $K(\tau)$ of (\ref{AL4}) satisfies an inequality 
$K(\tau)\le C_3\tau^3/y$, where $C_3$ is constant.  Hence for $\tau=\tau(x,y,T)$ we have
using (\ref{P4}) that
\begin{multline} \label{BH4}
\frac{K(\tau)m_{1,A}(\tau)^2}{\sig^2_A(\tau)} \ \le \ \ \frac{C_3m_{1,A}(T_0)^2\tau^2}{y} \\
 \le \frac{C_3m_{1,A}(T_0)^2C_4^2T^2y}{x^2} \ \le   \frac{C_3m_{1,A}(T_0)^2C_4^2}{2C_2^2}\quad {\rm for \ }  x\ge C_2\max\{2y,T^2\}, \ 0<T\le T_0 \ .
\end{multline}
We choose $C_2$ large enough so that the final expression on the RHS of (\ref{BH4}) is less than $1/2$.   We have then from (\ref{AL4}) the lower bound 
\be \label{BI4}
\frac{\pa^2 q(x,y,T)}{\pa x^2} \ \ge \ -\frac{2C_3m_{1,A}(T_0)^2\tau^3}{T^2y} \ ,
\ee
whence the second inequality of (\ref{BF4}) follows on using the bound (\ref{P4})  for $\tau=\tau(x,y,T)$ in (\ref{BI4}). 
To prove (\ref{BG4}) we again use the formulas (\ref{Z4}), (\ref{AL4}), observing that $\tau(x,y,T)<T$. 
\end{proof}

\vspace{.1in}
\section{Uniform bounds on $q_\ve$ and its derivatives}
In this section our goal is to show that the bounds on $q_0$ and its first two space derivatives obtained in Proposition 4.1 and Corollary 4.1 may be extended to $q_\ve$ with $\ve>0$.  First we prove  results for $\pa q_\ve(x,y,T)/\pa x$ analogous to the bounds on $q_\ve(x,y,T)$ obtained in Proposition 3.3. In fact the lower bound in (\ref{A5}) implies the lower bound in (\ref{AB3}), and the upper bound in (\ref{A5}) implies (\ref{AC3}).  
\begin{proposition}
Assume the function $A:[0,\infty)\ra\mathbb{R}$ is continuous non-negative, and let $q_\ve(x,y,T)$ be defined by (\ref{R2}), (\ref{U2}). Then
\be \label{A5}
\frac{2m_{1,A}(T)y}{\sig^2_A(T)}  \     \le \  \frac{\pa q_\ve(x,y,T)}{\pa x} \ \le  \ 2\left[1-\frac{m_{2,A}(T)}{\sig^2_A(T)}\right] + \frac{2m_{1,A}(T)y}{\sig^2_A(T)} \ , \quad x,y,T>0 \ .
\ee
\end{proposition}
\begin{proof}
Letting $u_\ve(x,y,T)=\pa q_\ve(x,y,T)/\pa x$, we note from Proposition 3.1 that the function $(x,T)\ra u_\ve(x,y,T)$ is continuous in the domain $\{(x,T): \ x,T>0\}$, with continuous derivatives $\pa u_\ve(x,y,T)/\pa x, \ \pa u_\ve(x,y,T)/\pa T, \ \pa^2 u_\ve(x,y,T)/\pa x^2$.  Furthermore $(x,T)\ra u_\ve(x,y,T)$ is continuous up to the boundary $\{(x,T): \ x=0,T>0\}$. Differentiating (\ref{V2}) with respect to $x$, we see using (\ref{O2}) that $u_\ve(x,y,T)$ is a solution to the diffusive Burger's PDE
\begin{multline} \label{B5}
\frac{\pa u_\ve(x,y,T)}{\pa T}  + \left[\la(x,y,T)+ u_\ve(x,y,T)\right]\frac{\pa u_\ve(x,y,T)}{\pa x} \\
+\left[A(T)+\frac{1}{\sig^2_A(T)}\right]u_\ve(x,y,T) \  = \ \frac{\ve}{2}\frac{\pa^2 u_\ve(x,y,T)}{\pa x^2} \ , \quad x,y,T>0 \ .
\end{multline}
It follows from (\ref{B5}) that
\begin{multline} \label{C5}
\frac{\pa v_\ve(x,y,T)}{\pa T}  + \left[\la(x,y,T)+ u_\ve(x,y,T)\right]\frac{\pa v_\ve(x,y,T)}{\pa x} 
 \  = \ \frac{\ve}{2}\frac{\pa^2 v_\ve(x,y,T)}{\pa x^2} \ , \\
 {\rm where \ } v_\ve(x,y,T) \ = \ \frac{\sig^2_A(T)u_\ve(x,y,T)}{m_{1,A}(T)} \ .
\end{multline}
From Proposition 3.2 it follows that the function $(x,T)\ra v_\ve(x,y,T)$ is non-negative.  Since $q_\ve(0,y,T)=0$ we also have from the lower bound (\ref{AB3}) of Proposition 3.3 that $v_\ve(0,y,T)\ge 2y, \ y,T>0$.  

Using Ito's lemma and the martingale optional sampling theorem, we have that 
\be \label{D5}
v_\ve(x,y,T) \ = \ E\left[ v_\ve(X^*_\ve(s\vee\tau^*_{\ve,x,T,K}),y,s\vee\tau^*_{\ve,x,T,K}) \ |  \ X^*_\ve(T)=x\right] \ , \quad 0<s<T \ ,
\ee
where the stopping time $\tau^*_{\ve,x,T,K}$ is defined in the proof of Lemma 2.3. In view of the non-negativity of $v_\ve$ and lower bound on $v_\ve(0,y,\cdot)$, we conclude from (\ref{D5})  that
\begin{multline} \label{E5}
v_\ve(x,y,T) \ \ge \ E\left[ v_\ve(0,y,\tau^*_{\ve,x,T}) \ ;  \ \tau^*_{\ve,x,T}=\tau^*_{\ve,x,T,K}>s\right] \\
\ge \ 2y P\left(       \tau^*_{\ve,x,T}=\tau^*_{\ve,x,T,K}>s           \right) \ , \quad 0<s<T \ .
\end{multline}
Observe that
\be \label{F5}
\{\tau^*_{\ve,x,T}>s\} \ = \ \{\tau^*_{\ve,x,T}=\tau^*_{\ve,x,T,K}>s\}\cup\{\tau^*_{\ve,x,T,K}>s, \ X^*_\ve(\tau^*_{\ve,x,T,K})=K\} \ .
\ee
As in the proof of Lemma 2.3 we use the  inequality $X^*_\ve(\cdot)\le X_\ve(\cdot)$, where $X_\ve(\cdot)$ is given by (\ref{P2}), to show that
$\limsup_{K\ra\infty}P\left( \tau^*_{\ve,x,T,K}>s, \ X^*_\ve(\tau^*_{\ve,x,T,K})=K           \right)\le \limsup_{K\ra\infty} P\left(  \sup_{s<s'<T}   X_\ve(s')>K \right)=0$. Letting $K\ra\infty$ in (\ref{E5}) we then have from (\ref{F5}) that $v_\ve(x,y,T)\ge 2yP\left(  \tau^*_{\ve,x,T}>s   \right)$ for $0<s<T$. Since Lemma 2.3 implies that $\lim_{s\ra 0} P\left(  \tau^*_{\ve,x,T}>s   \right)=1$ we conclude that $v_\ve(x,y,T)\ge 2y$, whence the lower bound in (\ref{A5}). 

To get the upper bound in (\ref{A5}) we define for $h>0$ a function $q_{\ve,h}$ by
\be \label{G5}
q_{\ve,h}(x,y,T) \ = \ q_\ve\left(x+\frac{\sig^2_A(T)}{m_{1,A}(T)}h,y,T\right) \quad x,T>0 \ .
\ee
The function $(x,T)\ra q_{\ve,h}(x,y,T)$ is also a solution to (\ref{V2}), whence the function $v_{\ve,h}(x,y,T)=\left[q_{\ve,h}(x,y,T)-q_\ve(x,y,T)\right]/h$ is a solution to the PDE
\begin{multline} \label{H5}
\frac{\pa v_{\ve,h}(x,y,T)}{\pa T}  +\la(x,y,T)\frac{\pa v_{\ve,h}(x,y,T)}{\pa x} \\
+ \frac{1}{2}\left[\frac{\pa q_{\ve,h}(x,y,T)}{\pa x}+\frac{\pa q_\ve(x,y,T)}{\pa x}\right]\frac{\pa v_{\ve,h}(x,y,T)}{\pa x} 
 \  = \ \frac{\ve}{2}\frac{\pa^2 v_{\ve,h}(x,y,T)}{\pa x^2} \ , \quad x,T>0 \  .
\end{multline}
Let $X^*_{\ve,h}(\cdot)$ denote solutions to the backwards in time  SDE (\ref{Y2}) with $\mu_\ve$ given by 
\be \label{I5}
\mu_\ve(x,y,T) \ = \ \la(x,y,T)+ \frac{1}{2}\left[\frac{\pa q_{\ve,h}(x,y,T)}{\pa x}+\frac{\pa q_\ve(x,y,T)}{\pa x}\right] \ .
\ee
Letting $\tau^*_{\ve,h,x,T,K}$ be the first exit time for $X^*_{\ve,h}(s), \ s<T,$ with terminal condition $X^*_{\ve,h}(T)=x$ from the interval 
$(0,K)$ we have again from Ito's lemma and the martingale optional sampling theorem the identity
\be \label{J5}
v_{\ve,h}(x,y,T) \ = \ E\left[ v_{\ve,h}(X^*_{\ve,h}(s\vee\tau^*_{\ve,h,x,T,K}),y,s\vee\tau^*_{\ve,h,x,T,K}) \ |  \ X^*_{\ve,h}(T)=x\right] \ , \quad 0<s<T \ .
\ee
We may simplify the expression (\ref{J5}) by taking $K\ra \infty$.  Observe that
\begin{multline} \label{K5}
 \left|E\left[ v_{\ve,h}(K,y,\tau^*_{\ve,h,x,T,K}); \ \tau^*_{\ve,h,x,T,K}>s,  X^*_\ve(\tau^*_{\ve,h,x,T,K})=K   \ |  \ X^*_{\ve,h}(T)=x\right]\right|
 \\
  \le \sup_{s<s'<T} |v_{\ve,h}(K,y,s')|P\left(\sup_{s<s'<T} X^*_{\ve,h}(s')>K \ | \ X^*_{\ve,h}(T)=x\right)  \ .
\end{multline}
Since the drift $\mu_\ve$ of (\ref{I5}) satisfies $\mu_\ve(x,y,T)\ge \la(x,y,T)$ we see that the probability on the RHS of (\ref{K5}) is bounded by 
$P\left(\sup_{s<s'<T} X_\ve(s')>K \ | \ X_\ve(T)=x\right)$, where $X_\ve(\cdot)$ is given by (\ref{P2}). We may bound this latter probability by using the reflection principle (\ref{AT2}), whence the probability converges to zero as $K\ra \infty$ like $\exp[-cK^2]$ for some constant $c>0$.  From Proposition 3.3 we see that $ \sup_{s<s'<T} |v_{\ve,h}(K,y,s')|$ is bounded linearly in $K$ as $K\ra \infty$. We conclude that the RHS of (\ref{K5}) converges to $0$ as $K\ra\infty$.  Letting $K\ra\infty$ in (\ref{J5}) we have then that
\begin{multline} \label{L5}
v_{\ve,h}(x,y,T) \ = \ E\left[ v_{\ve,h}(0,y,\tau^*_{\ve,h,x,T}); \  \tau^*_{\ve,h,x,T}>s \ |  \ X^*_{\ve,h}(T)=x\right] \\
+ E\left[ v_{\ve,h}(X^*_{\ve,h}(s),y,s);  \ \tau^*_{\ve,h,x,T}<s \ |  \ X^*_{\ve,h}(T)=x\right] \ , \quad 0<s<T \ ,
\end{multline}
where $\tau^*_{\ve,h,x,T}$ is the first hitting time at $0$ for the diffusion $X^*_{\ve,h}(s), \ s<T,$ with terminal condition $X^*_{\ve,h}(T)=x$.

It is easy to bound from above the first expectation on the RHS of (\ref{L5}) by using Proposition 3.3. Thus we have from the upper bound (\ref{AC3})  the inequality $v_{\ve,h}(0,y,\tau) \ \le \  2g_{2,A}(\tau,\tau)+2y, \tau>0,$ where $g_{2,A}(\tau,\tau)=[\sig^2_A(\tau)-m_{2,A}(\tau)]/m_{1,A}(\tau)$ has derivative (\ref{AD4}). Since $A(\cdot)$ is non-negative, whence the function $\tau\ra g_{2,A}(\tau,\tau)$ is increasing,  we see that the first expectation on the RHS of (\ref{L5}) is bounded above by $2g_{2,A}(T,T)+2y$ for all $s$ satisfying $0<s<T$.  We shall show that the limit of the second expectation on the RHS of (\ref{L5})  converges to $0$ as $s\ra 0$. The upper bound in (\ref{A5}) follow  then by first letting $s\ra 0$ in (\ref{L5}) and then $h\ra 0$. 

Using the bound (\ref{I3}) of Proposition 3.2, we see that the second expectation on the RHS of (\ref{L5}) is bounded in absolute value by
\be \label{M5}
\frac{C(T)y}{sh}  E\left[ X^*_{\ve,h}(s)+sh+s^2;  \ \tau^*_{\ve,h,x,T}<s \ |  \ X^*_{\ve,h}(T)=x\right]  \ , \quad 0<s<T,
\ee
where the constant $C(T)$ depends only on $T$. In order to estimate the expression (\ref{M5}) we use the lower bound (\ref{A5}) which has been already proven. Let $X_{\ve,{\rm linear}}(\cdot)$ denote solutions to the SDE (\ref{Y2}) with $\mu_\ve$ given by 
\be \label{O5}
\mu_\ve(x,y,T) \ = \ \la(x,y,T)+ \frac{\pa q_{\rm  linear}(x,y,T)}{\pa x} \ = \la(x,-y,T) \ .
\ee
If $ X^*_{\ve,h}(T)= X_{\ve,{\rm linear}}(T)=x$ then $ X^*_{\ve,h}(s)\le X_{\ve,{\rm linear}}(s)$ for all $\tau^*_{\ve,h,x,T}<s<T$. Letting
$\tau_{\ve,{\rm linear},x,T}$ be the first hitting time at $0$ for $X_{\ve,{\rm linear}}(s), \ s<T,$ with $X_{\ve,{\rm linear}}(T)=x$, we see that
 $\tau^*_{\ve,h,x,T}\ge \tau_{\ve,{\rm linear},x,T}$. Hence we have that
 \be \label{P5}
  E\left[ X^*_{\ve,h}(s);  \ \tau^*_{\ve,h,x,T}<s \ |  \ X^*_{\ve,h}(T)=x\right] \ \le \ 
  E\left[ X_{\ve,{\rm linear}}(s);    \ \tau_{\ve,{\rm linear},x,T}<s \ |   \   X_{\ve,{\rm linear}}(T)=x   \right] \ .
 \ee
 
 Since the drift (\ref{O5}) is linear the SDE (\ref{Y2}) can be explicitly solved in this case and the solution is obtained by replacing $y$ by $-y$ in the formula (\ref{P2}).  Thus we have that
 \be \label{Q5}
X_\ve(s) \ = \ x_{\rm class}(s)-\sqrt{\ve}\frac{\sig^2_A(s)}{m_{1,A}(s)}Z(s), \quad s<T \ ,
\ee
where $Z(\cdot)$ is defined in (\ref{P2}), and from (\ref{L2}) we have that
\begin{multline} \label{R5}
\sig_A^2(T)x_{\rm class}(s) \ = \  xm_{1,A}(s,T)\sig_A^2(s) -ym_{1,A}(s)\sig_A^2(s,T) \\
+ \ m_{1,A}(s,T)m_{2,A}(s,T)\sig_A^2(s)-m_{2,A}(s)\sig_A^2(s,T) \ .
\end{multline}
We see from (\ref{R5}) that  $\lim_{s\ra 0} x_{\rm class}(s)=-y$, whence there exists $s_0$ with $0<s_0<T$ such that $ x_{\rm class}(s)\le -y/2$ for $0<s\le s_0$.  It follows then from (\ref{Q5}) that
\be \label{S5}
P\left( \tau_{\ve,{\rm linear},x,T}<s   \right) \ \le \ P\left( \inf_{s<s'<s_0} s'Z(s')> cy/\sqrt{\ve}                     \right) \ , \quad 0<s\le s_0 \ ,
\ee
for some constant $c>0$ depending only on $s_0$.  The variables $sZ(s), \ s<T,$ are Gaussian with zero mean and variance bounded above by $C(T)s$ for some constant $C(T)$ depending only on $T$. Hence the probability on the RHS of (\ref{S5}) is bounded above by $\exp[-c/s]$ for some constant $c>0$. Observing also that   $\sup_{0<s<T}E\left[ X_{\ve,{\rm linear}}(s)^2 \ | \ X_{\ve,{\rm linear}}(T)=x   \right] <\infty$, 
we conclude from (\ref{S5}) and  the Schwarz inequality applied to the RHS of (\ref{P5}) that the expression (\ref{M5}) converges to $0$ as $s\ra 0$. 
\end{proof}
\begin{proposition}
Assume the function $A:[0,\infty)\ra\mathbb{R}$ is continuous and non-negative. Then for any $T_0>0$, there exists a constant $C>0$, depending only on $T_0$ and $\sup_{0\le t\le T_0}A(t)$,  such that   the function $(x,T)\ra q_\ve(x,y,T)$ satisfies the inequality
\be \label{T5}
-\frac{CTy^2}{x}  \ \le \ q_\ve(x,y,T)-q_{\rm linear}(x,y,T)  \quad {\rm for \ } xy\ge \ve T, \ x\ge \max\{2y,T^2\}, \ 0<T<T_0 \ ,
\ee
and also the inequality
\be \label{U5}
q_\ve(x,y,T)-\frac{2m_{1,A}(T)xy}{\sig^2_A(T)} \ \le \ CTx \quad {\rm for \ } x>0, \ 0<T<T_0 \ . 
\ee
\end{proposition}
\begin{proof}
All constants in the following can be chosen to depend only on $T_0$ and $\sup_{0\le t\le T_0}A(t)$.
We consider the stochastic integral $s\ra M(s)$ defined similarly to (\ref{AX2}) but with $q_{\rm linear}$ in place of $q_\ve$ and  $\mu_\ve$ in (\ref{Y2}) given by $\mu_\ve=\mu^ *_\ve$ of (\ref{AA2}).  Then similarly to (\ref{AZ2}) we obtain the inequality
 \begin{multline} \label{V5}
 q_{\rm linear}(x,y,T) \ \le E[q_{\rm linear}(X^*_\ve(s\vee\tau^*_{\ve,x,T,K}),y,s\vee\tau^*_{\ve,x,T,K})] \\
 + E\left[\frac{1}{2}\int_{s\vee\tau^*_{\ve,x,T,K}}^T[\mu^*_\ve(X^*_\ve(s),y,s)-\la(X^*_\ve(s),y,s)]^2 \ ds \ \Big| \ X^*_\ve(T)=x \  \right] \ , \quad 0<s\le T \ ,
 \end{multline}
 where $X^*_\ve(\cdot)$ is the solution to the SDE (\ref{Y2}) with $\mu_\ve=\mu_\ve^*$. The stopping time  $\tau^*_{\ve,x,T,K}$ in (\ref{V5}) is the first exit time of $X^*_\ve(s), \ s<T, $ with $X^*_\ve(T)=x$ from the interval $(0,K)$. As in the proof of  Lemma 2.3 we use the inequality
 $X^*_\ve(\cdot)\le X_\ve(\cdot)$, where $X_\ve(\cdot)$ is given by (\ref{P2}). Letting $K\ra \infty$ in (\ref{V5}) we have then using (\ref{BP2}) of Lemma 2.3 the inequality
 \be \label{W5}
  q_{\rm linear}(x,y,T) \ \le E[q_{\rm linear}(X^*_\ve(s\vee\tau^*_{\ve,x,T}),y,s\vee\tau^*_{\ve,x,T})]+q_\ve(x,y,T) \ , \quad 0<s<T \ ,
 \ee
 with $\tau^*_{\ve,x,T}$ the stopping time defined in Lemma 2.3. Using the inequality $q_{\rm linear}(x,y,\tau)\le Cy[x/\tau+\tau], \ 0<\tau\le T_0,$ where $C$ is constant,  we have from (\ref{W5}) the inequality
  \be \label{X5}
  q_{\rm linear}(x,y,T)-q_\ve(x,y,T) \ \le CyE[\tau^*_{\ve,x,T}]+\frac{Cy}{s}E\left[X^*_\ve(s)+s^2; \     \tau^*_{\ve,x,T}<s \  \right] \ , \quad 0<s<T \ .
 \ee
 Just as in the proof of Proposition 5.1 we use the lower bound (\ref{A5}) to show that the second expectation on the RHS of (\ref{X5}) converges to $0$ as $s\ra 0$. 
 
 In order to bound $E[\tau^*_{\ve,x,T}]$ we use the identity (\ref{BS2}), observing since $A(\cdot)$ is non-negative that the sum of the last two terms on the RHS are non-negative. We have then upon  applying the Schwarz inequality to the RHS of (\ref{BS2})  that  for any $T_0>0$ one has
\begin{multline} \label{Y5}
\frac{C_1}{\sqrt{\tau^*_{\ve,x,T}}}  \left\{\frac{1}{2}\int_{\tau^*_{\ve,x,T}}^T[\mu^*_\ve(X^*_\ve(s),y,s)-\la(X^*_\ve(s),y,s)]^2 \ ds\right\}^{1/2} \\
 \ge \  \frac{c_1x}{T} -\sqrt{\ve}Z(\tau^*_{\ve,x,T})
 \quad {\rm for \ } 0<T\le T_0 \ ,
\end{multline}
where  $C_1,c_1>0$ are constants.  It follows from (\ref{Y5}) that if $|Z(\tau^*_{\ve,x,T})|\le c_1x/2\sqrt{\ve}T$ then
\be \label{Z5}
\tau^*_{\ve,x,T} \ \le  \ \left(\frac{2TC_1}{c_1x}\right)^2\frac{1}{2}\int_{\tau^*_{\ve,x,T}}^T[\mu^*_\ve(X^*_\ve(s),y,s)-\la(X^*_\ve(s),y,s)]^2 \ ds \ .
\ee
We conclude from (\ref{BP2}) of Lemma 2.3 and (\ref{Z5}) that
\be \label{AA5}
E\left[\tau^*_{\ve,x,T}; \ |Z(\tau^*_{\ve,x,T})|\le c_1x/2\sqrt{\ve}T\right]\ \le  \ \left(\frac{2TC_1}{c_1x}\right)^2 q_\ve(x,y,T) \ .
\ee
From (\ref{BY2}), (\ref{BZ2}) we have that
\be \label{AB5}
E\left[\tau^*_{\ve,x,T}; \ |Z(\tau^*_{\ve,x,T})|>a\right] \ \le \  E[\tau_a] \ \le \ \frac{C_2}{a^2} , \quad a>0, \ 0<T\le T_0,
\ee
where  $C_2$ is constant. It follows from (\ref{AB5}) that
\be \label{AC5}
E\left[\tau^*_{\ve,x,T}; \ |Z(\tau^*_{\ve,x,T})|> c_1x/2\sqrt{\ve}T\right] \ \le \  
\frac{4C_2\ve T^2}{c_1^2x^2} \ , \quad 0<T\le T_0 \ .
\ee
We conclude from (\ref{AA5}), (\ref{AC5}) and Proposition 3.2  that
\be \label{AD5}
E\left[\tau^*_{\ve,x,T}\right] \ \le \  \frac{C_3Ty}{x} \quad {\rm if \ } x\ge 2y, \ x\ge T^2, \ xy>\ve T, \ 0<T\le T_0,
\ee
where  $C_3$ is constant. The lower bound (\ref{T5})  follows from (\ref{X5}), (\ref{AD5}) on letting $s\ra 0$ in (\ref{X5}). 

To prove (\ref{U5}), we first observe that if $x\ge y$ the inequality follows from the inequality $q_\ve(x,y,T)\le q_{\rm linear}(x,y,T)$. Hence we may assume $0<x<y$. 
 We consider the stochastic integral $s\ra M(s)$ defined similarly to (\ref{AX2}) but with  $\mu_\ve$ in (\ref{Y2}) given by (\ref{O5}).  Arguing as in the proof of Proposition 5.1 we have analogously to (\ref{AZ2}) the inequality
  \be \label{AE5}
 q_\ve(x,y,T) \ \le \
  E\left[\frac{1}{2}\int_{\tau_{\ve,{\rm linear},x,T}}^T[\mu_\ve(X_\ve(s),y,s)-\la(X_\ve(s),y,s)]^2 \ ds \ \Big| \ X_\ve(T)=x \  \right] \ .
 \ee
Next we consider the stochastic integral $s\ra M(s)$ defined similarly to (\ref{AX2}) but with $q_{\rm linear}$ in place of $q_\ve$ and  $\mu_\ve$ in (\ref{Y2})  again given by (\ref{O5}). Then using Ito's formula and the martingale optional sampling theorem  we have the identity
\begin{multline} \label{AF5}
q_{\rm linear}(x,y,T) \ = \ E\left[q_{\rm linear}\left(0,y,\tau_{\ve,{\rm linear},x,T}\right)\right] \\
+  E\left[\frac{1}{2}\int_{\tau_{\ve,{\rm linear},x,T}}^T[\mu_\ve(X_\ve(s),y,s)-\la(X_\ve(s),y,s)]^2 \ ds \ \Big| \ X_\ve(T)=x \  \right]  \ .
\end{multline}
Observe next that
\begin{multline} \label{AG5}
q_{\rm linear}(x,y,T) -E\left[q_{\rm linear}\left(0,y,\tau_{\ve,{\rm linear},x,T}\right)\right] \\
 \le \ \frac{2m_{1,A}(T)xy}{\sig^2_A(T)} +
C_4yE\left[T-\tau_{\ve,{\rm linear},x,T}\right] \ ,  \quad 0<T\le T_0 \ ,
\end{multline}
for some constant $C_4$. The inequality (\ref{U5}) follows from (\ref{AE5})-(\ref{AG5}) if we can show that
\be \label{AH5}
E\left[T-\tau_{\ve,{\rm linear},x,T}\right] \ \le \ \frac{C_5Tx}{y} \ ,  \quad 0<x<y, \ 0<T\le T_0 \ ,
\ee
for a constant $C_5$. 

We show that (\ref{AH5}) holds if $y\ge C_6T^2$ for some constant $C_6$.  We choose $C_6$ such that the drift $\mu_\ve$ defined by (\ref{O5})  satisfies the inequality $\mu_\ve(x,y,s)\ge y/C_5s, \ 0<s\le T_0$ for some constant $C_5>0$.  This follows from (\ref{O2}) since  $\sig^2_A(s)-m_{2,A}(s)\le C_7s^2, \ 0<s\le T_0,$ where $C_7$ is constant. Then the LHS of (\ref{AH5}) is bounded above by $u_\ve(x)=
E\left[T-\tau^*_{\ve,{\rm linear},x,T}\right] $, where $\tau^*_{\ve,{\rm linear},x,T}$ is the first exit time from $(0,\infty)$ for the diffusion $X_\ve(s), \ s<T,$ which is the solution to (\ref{Y2}) with terminal condition $X_\ve(T)=x$ and drift $\mu_\ve (x,y,T)=y/C_5T$. Now  $u_\ve(\cdot)$ is the solution to the boundary value problem,
\be \label{AJ5}
-\frac{\ve}{2}\frac{d^2u_\ve(x)}{dx^2}+\frac{y}{C_5T} \frac{du_\ve(x)}{dx} \ = \ 1 \ ,  \ \ x>0, \quad u_\ve(0)=0 \ .
\ee
The solution to (\ref{AJ5}) is the linear function $u_\ve(x)=C_5Tx/y$, whence we obtain the upper bound (\ref{AH5}). 

To finish the proof of (\ref{U5}) we need to deal with the case $0<x<y<C_6T^2$.  In this case (\ref{U5}) reduces to the inequality $q_\ve(x,y,T)\le C_8Tx, \ 0<T\le T_0$, where $C_8$ is constant. We consider again the stochastic integral $s\ra M(s)$ defined similarly to (\ref{AX2}) but with  $\mu_\ve$ in (\ref{Y2})
given by $\mu_\ve(x,y,s)=\la(x,y,s)+C_9T, \ 0<s<T\le T_0$, where the constant $C_9>0$  is chosen sufficiently large.   Let $\tau_{\ve,x,T}$ be the first exit time for the diffusion $X_\ve(s), \ s<T,$ of (\ref{Y2}) with $X_\ve(T)=x$. Then  a similar inequality to (\ref{AE5}) holds, whence we have
\be \label{AK5}
q_\ve(x,y,T) \ \le \  \frac{C_9^2T^2}{2}E[T-\tau_{\ve,x,T}] \ , \quad 0<x\le y\le C_6T^2, \ 0<T\le T_0 \ .
\ee
 We also see  as before that $E[T-\tau_{\ve,x,T}]\le C_{10}x/T, \ 0<T\le T_0,$, where $C_{10}$ is constant.  The result follows.
\end{proof}
\begin{rem}
The inequality (\ref{U5}) also follows from the upper bound in (\ref{A5}) by integration over the interval $[0,x]$. In our proof in Proposition 5.2 we 
use the fact that the optimizing $\tau$ in (\ref{B4}) is close to $T$ as $x\ra 0$. 
\end{rem}
We have already shown in Proposition 5.1 that all  of the bounds on the derivative $\pa q_\ve(x,y,T)/\pa x$ with $\ve=0$ in Corollary 4.1, with the exception of the upper bound (\ref{BF4}),    extend to $\ve>0$. Next we extend the upper bound (\ref{BF4}) to $\ve>0$. 
\begin{proposition}
Assume the function $A:[0,\infty)\ra\mathbb{R}$ is continuous and non-negative. Then for any $T_0>0$ there exists a constant $C>0$,  depending only on $T_0$ and $\sup_{0\le t\le T_0}A(t)$, such that   the function $(x,T)\ra q_\ve(x,y,T)$ satisfies the inequality
\be \label{AL5}
\frac{\pa q_\ve(x,y,T)}{\pa x}-\frac{2m_{1,A}(T)y}{\sig^2_A(T)}  \ \le \  \frac{CTy^2}{x^2}  , \quad {\rm for \ }xy\ge \ve T, \   x\ge\max\{2y,CT^2\}, \ 0<T\le T_0 \ .
\ee
\end{proposition}
\begin{proof}
We use the identity (\ref{D5}). The upper bound (\ref{A5})  implies that for any $T_0>0$ there is a constant $C_1$ such that $v_\ve(x,y,\tau)\le 2y+C_1\tau^2, \ 0<\tau\le T_0$. Hence we may let $K\ra \infty$ and $s\ra 0$ in (\ref{D5}) to obtain the identity
\be \label{AM5}
v_\ve(x,y,T) \ = \ E\left[ v_\ve(0,y,\tau^*_{\ve,x,T}) \ |  \ X^*_\ve(T)=x\right] \ , \quad 0<s<T \ .
\ee
 The inequality (\ref{AL5})  follows then from (\ref{AM5})  if we can show that
\be \label{AN5}
E\left[ (\tau^*_{\ve,x,T})^2\right] \ \le \ \frac{C_2T^2y^2}{x^2} \ , \quad x\ge\max\{2y,C_2T^2\}, \ 0<T\le T_0 \ ,
\ee
for a constant $C_2$. To prove (\ref{AN5}) we use the upper bound (\ref{A5}).  Analogously to (\ref{O5}) we consider solutions $
X^*_{\ve,{\rm linear}}(\cdot)$ to the SDE (\ref{Y2})  with $\mu_\ve$ given by 
\be \label{AO5}
\mu_\ve(x,y,T) \ = \ \la(x,-y,T)+ 2\left[1-\frac{m_{2,A}(T)}{\sig^2_A(T)}\right] \ .
\ee
Letting
$\tau^*_{\ve,{\rm linear},x,T}$ be the first hitting time at $0$ for $X^*_{\ve,{\rm linear}}(s), \ s<T,$ with $X^*_{\ve,{\rm linear}}(T)=x$, we see that
 $\tau^*_{\ve,x,T}\le \tau^*_{\ve,{\rm linear},x,T}$. In order to prove (\ref{AN5}) it will be sufficient therefore to estimate $E\left [ \left(\tau^*_{\ve,{\rm linear},x,T}\right)^2   \right]$.
 
 Since the drift (\ref{AO5}) is linear the SDE (\ref{Y2}) can be again explicitly solved, and the solution is given by  
 \be \label{AP5}
X_\ve(s) \ = \ x^*_{\rm class}(s)-\sqrt{\ve}\frac{\sig^2_A(s)}{m_{1,A}(s)}Z(s), \quad s<T \ ,
\ee
where $Z(\cdot)$ is defined in (\ref{P2}), and $x^*_{\rm class}(\cdot)$ is obtained from (\ref{R5}) by switching the signs of the terms which do not involved $x$ or $y$. Thus we have that
\begin{multline} \label{AQ5}
\sig_A^2(T)x^*_{\rm class}(s) \ = \  xm_{1,A}(s,T)\sig_A^2(s) -ym_{1,A}(s)\sig_A^2(s,T) \\
- \ m_{1,A}(s,T)m_{2,A}(s,T)\sig_A^2(s)+m_{2,A}(s)\sig_A^2(s,T) \ .
\end{multline}
We have from (\ref{AN2}), (\ref{E4}) and (\ref{AQ5})   that
\be \label{AR5}
x^*_{\rm class}(s) \ \ge \ \ c_3\frac{sx}{T}-\left[1-\frac{s}{T}\right]\left\{C_3y+C_4sT\right\} \ , \quad 0<s< T, \ \ \ 0<T\le T_0, 
\ee
for some constants $C_3,c_3,C_4>0$.  We conclude from (\ref{AR5}) that
\be \label{AS5}
 \tau^*_{0,{\rm linear},x,T} \ \le \ T\frac{2C_3y}{c_3x+2C_3y} \quad {\rm if \ } x\ge \frac{2C_4T^2}{c_3} \ .
\ee

We may extend the inequality (\ref{AS5}) to $\ve>0$ by considering for $n=1,2,\dots,$ events $A_n$ where  $A_n=\{\tau^*_{\ve,{\rm linear},x,T}>nTy/x\}$. Assuming $x\ge\max\{2y,2C_4T^2/c_3\}$, we have  from (\ref{AR5}) the inequality $x^*_{\rm class}(\tau^*_{\ve,{\rm linear},x,T})\ge c_3\tau^*_{\ve,{\rm linear},x,T}(x/4T)$ on the event $A_n$  provided $n\ge 4C_3/c_3$.    Hence from (\ref{AP5}) we have on $A_n$ with $n\ge 4C_3/c_3$ the inequality $Z\left(   \tau^*_{\ve,{\rm linear},x,T} \right)<-c_5x/\sqrt{\ve} T$, where $c_5>0$ depends only on $T_0$.  From (\ref{BZ2}) we have that
\be \label{AT5}
P(\sup_{s<s'<T}|Z(s')|>a) \ \le \  \frac{C_6}{(a^2s)^4} \ , \quad a>0, \ 0<s<T\le T_0 \ ,
\ee
where $C_6>0$ depends only on $T_0$.  Choosing integers $n_0,n_1$ such that $n_0\ge 4C_3/c_3$ and $n_1\ge x/y$, we have from (\ref{AT5}) that
\begin{multline} \label{AU5}
E\left[\left(      \tau^*_{\ve,{\rm linear},x,T}    \right)^2\right] \ \le \ \frac{T^2 y^2}{x^2}\left[n^2_0+ \sum_{n_0\le n\le n_1} (n+1)^2P(A_n)\right] \\
\le \  \frac{T^2 y^2}{x^2}\left[n^2_0+ \frac{C_6 }{c_5^8}\left(\frac{\ve T}{xy}\right)^4\sum_{n_0\le n\le n_1}\frac{(n+1)^2}{n^4} \right]  \ .
\end{multline}
 The inequality (\ref{AN5}) follows from (\ref{AU5}) provided $xy\ge \ve T$. 
\end{proof}
Finally we show that the function $x\ra q_\ve(x,y,T)$ is concave, thereby extending the upper bound  on $\pa^2q_\ve(x,y,T)/\pa x^2$  with $\ve=0$ in Corollary 4.1 to $\ve>0$. Because of the singularity in the drift $[x,T]\ra\la(x,y,T), \ x,y,T>0,$ of (\ref{O2}) as $T\ra0$,  we use an approximation method.  
\begin{lem}
Assume the function $A:[0,\infty)\ra\mathbb{R}$ is continuous and non-negative. Then for any $\del>0$ there is a unique classical solution 
$[x,T]\ra q_{\ve,\del}(x,y,T), \ x>0, \ T>\del,$ to the PDE (\ref{V2})  with boundary and initial conditions
\be \label{AV5}
q_{\ve,\del}(0,y,T) \ = \ 0, \ T>\del, \quad q_{\ve,\del}(x,y,\del) \ = \ \frac{2m_{1,A}(\del)xy}{\sig^2_A(\del)} \ ,  \ \ x>0 \ .
\ee
Furthermore, the function $[x,T]\ra q_{\ve,\del}(x,y,T)$ satisfies the inequalities
\be \label{AW5}
 \frac{2m_{1,A}(T)xy}{\sig^2_A(T)} \ \le \ q_{\ve,\del}(x,y,T) \ \le \ -2\la(0,y,T)x \ , \quad  x>0, \ T>\del \ ,
\ee
\be \label{BA5}
\frac{2m_{1,A}(T)y}{\sig^2_A(T)}  \     \le \  \frac{\pa q_{\ve,\del}(x,y,T)}{\pa x} \ \le  \ 2\left[1-\frac{m_{2,A}(T)}{\sig^2_A(T)}\right] + \frac{2m_{1,A}(T)y}{\sig^2_A(T)} \ , \quad x>0, \ T>\del \ .
\ee
Let $q_\ve(x,y,T)$ be defined by (\ref{R2}), (\ref{U2}). Then  $\lim_{\del\ra 0}[q_{\ve,\del}(x,y,T)-q_\ve(x,y,T)]=0$, and the limit is uniform in all sets $\{[x,T]: \ x>0, \ T_0<T<T_1\}$ with $ 0<T_0<T_1<\infty$.  
\end{lem}
\begin{proof}
We  proceed as in the proof of Proposition 3.3 by setting $v_{\ve,\del}(x,y,T)=\exp\left[-q_{\ve,\del}(x,y,T)/\ve\right]$.  If $[x,T]\ra v_{\ve,\del}(x,y,T) $ is a solution to the PDE (\ref{S2}) in the region $x>0,T>\del$ with boundary and initial conditions
\be \label{AX5}
v_{\ve,\del}(0,y,T) \ = \ 1, \ T>\del, \quad v_{\ve,\del}(x,y,\del) \ = \ \exp\left[-\frac{2m_{1,A}(\del)xy}{\sig^2_A(\del)\ve}\right] \ ,  \ \ x>0 \ ,
\ee
then $q_{\ve,\del}(x,y,T)=-\ve\log v_{\ve,\del}(x,y,T)$ is a solution to (\ref{V2}) with boundary and initial conditions (\ref{AV5}).  Since the drift $[x,T]\ra\la(x,y,T)$ is linear in $x$ and continuous in $T$ for $T\ge \del$, standard regularity theory implies that $[x,T]\ra v_{\ve,\del}(x,y,T)$ is a classical solution to (\ref{S2}), (\ref{AX5}), from whence we conclude that $[x,T]\ra q_{\ve,\del}(x,y,T)$ is a classical solution to (\ref{V2}), (\ref{AV5}).   The proof of (\ref{AW5}) proceeds as in the proof of Proposition 3.3 by using the maximum principle.  With (5.47) established, the proof of (\ref{BA5}) then follows along the same lines as the proof of Proposition 5.1. 

To prove the convergence of $q_{\ve,\del}$ as $\del\ra 0$ we define the function $u_{\ve,\del}(x,y,T)  =  q_\ve(x,y,T)-q_{\ve,\del}(x,y,T)$ .
Since both functions $[x,T]\ra q_\ve(x,y,T)$ and  $[x,T]\ra q_{\ve,\del}(x,y,T)$ are solutions to (\ref{V2}) it follows that the function 
$[x,T]\ra u_{\ve,\del}(x,y,T)=q_\ve(x,y,T)-q_{\ve,\del}(x,y,T)$ is a solution to the PDE
\begin{multline} \label{AY5}
\frac{\pa u_{\ve,\del}(x,y,T)}{\pa T} \ = \ \frac{\ve}{2}\frac{\pa^2 u_{\ve,\del}(x,y,T)}{\pa x^2} \\
-\left\{ \la(x,y,T)+\frac{1}{2}\left[\frac{\pa q_\ve(x,y,T)}{\pa x}+\frac{\pa q_{\ve,\del}(x,y,T)}{\pa x}\right]\right\} \frac{\pa u_{\ve,\del}(x,y,T)}{\pa x}  \ ,
\end{multline}
in the region$\{[x,T]: \ x>0, \ T>\del\}$.  It follows from (\ref{AV5}) and the upper bound (\ref{AB3}) of Proposition 3.3 that the boundary and initial conditions satisfy
\be \label{AZ5}
u_{\ve,\del}(0,y,T) \ = \ 0, \ T>\del, \quad 0  \ \le \   u_\ve(x,y,\del) \ \le  \ C\del y \ ,  \ \ x>0 \ ,
\ee
where the constant $C$ may be chosen uniformly in any interval $0<\del<\del_0<\infty$.  From (\ref{AZ5}) and the maximum principle applied to (\ref{AY5})  we conclude that $0\le u_{\ve,\del}(x,y,T)\le C\del y$ for $x>0,T>\del$. 
\end{proof}
\begin{lem}
Assume the function $A:[0,\infty)\ra\mathbb{R}$ is continuous  non-negative, and  $q_{\ve,\del}(x,y,T), \ x,y>0, \ T>\del,$   the function defined in Lemma 5.1.
Then for any $T_0>\del$ there exist constants $C, M>0$, depending on $\ve,\del,y,T_0$  and $\sup_{0\le t\le T_0}A(t)$, such that 
\be \label{BB5}
\left|\frac{\pa^2 q_{\ve,\del}(x,y,T)}{\pa x^2}\right| \ \le \  C \quad {\rm for \ } x\ge M, \ \del<T\le T_0 \ .
\ee 
\end{lem}
\begin{proof}
Similarly to the derivation of (\ref{C5}) we see that the function
\be \label{BC5}
v_{\ve,\del}(x,y,T) \ = \ \frac{\sig^2(T)}{m_{1,A}(T)} \frac{\pa q_{\ve,\del}(x,y,T)}{\pa x}
\ee
is a solution to the PDE
\begin{multline} \label{BD5}
\frac{\pa v_{\ve,\del}(x,y,T)}{\pa T}  + \left[\la(x,y,T)+ \frac{m_{1,A}(T)}{\sig^2_A(T)}v_{\ve,\del}(x,y,T)\right]\frac{\pa v_{\ve,\del}(x,y,T)}{\pa x} 
 \  = \ \frac{\ve}{2}\frac{\pa^2 v_{\ve,\del}(x,y,T)}{\pa x^2} \ 
\end{multline}
in the region $\{[x,T]: \ x>0, \ T>\del\}$ with constant  initial condition $2y$ on the half line $\{[x,\del]: \ x>0\}$.  We make a change of variable to eliminate the linear drift $\la(\cdot,\cdot,\cdot)$ from (\ref{BD5}). To see this consider the PDE
\be \label{BE5}
\frac{\pa w(x,T)}{\pa T} +[\al(T)x+\beta(T)]\frac{\pa w(x,T)}{\pa x} \ = \ \frac{\ve}{2}\frac{\pa^2 w(x,T)}{\pa x^2} \ , \quad x\in\mathbb{R}, \ T>\del \ .
\ee
If we make the transformation $w(x,T)=u(z,t)$ where
\begin{multline} \label{BF5}
z=\exp\left[-\int_\del^T \al(s) \ ds\right]x-\int_\del^T\beta(s)\exp\left[-\int_\del^s \al(s') \ ds'\right] \ ds \ , \\
t=\int_\del^T\exp\left[-2\int_\del^s \al(s') \ ds'\right] \ ds  \ .
\end{multline}
then $u$ is a solution to the heat equation
\be \label{BG5}
\frac{\pa u(z,t)}{\pa t} \ = \ \frac{\ve}{2}\frac{\pa^2 u(z,t)}{\pa ^2 z} \ , \quad z\in\mathbb{R}, \ t>0 \ .
\ee
Now writing $\la(x,y,T)=\al(T)x+\beta(T)$ and setting $ v_{\ve,\del}(x,y,T)=u(z,t)$ according to the change of variables (\ref{BF5}), we see from (\ref{BD5}) that $u$ is a solution to the Burgers' equation
\be \label{BH5}
\frac{\pa u(z,t)}{\pa t}+\ga(t)u(z,t)\frac{\pa u(z,t)}{\pa z} \ = \ \frac{\ve}{2}\frac{\pa^2 u(z,t)}{\pa ^2 z} \ , 
\ee
where $\ga(\cdot)$ is the function
\be \label{BI5}
\ga(t) \ = \ \exp\left[2\int_\del^T \al(s) \ ds\right]\frac{m_{1,A}(T)}{\sig^2_A(T)} \ , \quad T\ge \del \ .
\ee

For any $z_0\in\mathbb{R}, \ t_0>0$ we define the domain $\mathcal{D}(z_0,\ve)=\{[z,t]: \ |z-z_0|<\sqrt{\ve t_0}, \ 0<t<t_0\}$.  The Dirichlet Green's function for the heat equation (\ref{BG5}) on the domain $\mathcal{D}(z_0,\ve)$ is simply a space translation and  dilation of the Green's function on the domain  $\mathcal{D}(0,1)$. This latter Green's function can be obtained by the method of images. Thus for $t>0$ let $z\ra G(z,t)$ be the pdf of the Gaussian variable with mean $0$ and variance $t$, so
\be \label{BK5}
G(z,t) \ = \ \frac{1}{\sqrt{2\pi t} }\exp\left[-\frac{z^2}{2t}\right] \  , \quad z\in\mathbb{R} \ .
\ee
Then the Dirichlet Green's function $G_D(z,z',t)$   for $\mathcal{D}(0,1)$ is given by the series
\be \label{BL5}
G_D(z,z',t) \ = \ \sum_{m=0}^\infty p(m)G(z-z_m,t) \ , 
\ee
where $z_0=z'$ and $z_m, \ m=1,2,\dots,$ are reflections of $z'$ in the boundaries $z'=\pm\sqrt{t_0}$ with parities $p(m)=\pm 1$.  The function
\be \label{BM5}
u(z,t) \ = \ \int_{-\sqrt{t_0}}^{\sqrt{t_0}}G_D(z,z',t)u_0(z') \ dz'  \ ,  \quad [z,t]\in \mathcal{D}(0,1)
\ee
is then a solution to (\ref{BG5}) with $\ve=1$. It satisfies the initial condition $u(z,0)=u_0(z), \ |z|<\sqrt{t_0},$ and boundary condition  $u(z,t)=0, \ z=\pm\sqrt{t_0}, \ 0<t<t_0$. 

Letting $t=t_0$ correspond to $T=T_0$ in (\ref{BF5}), we see from (\ref{BA5}) of Lemma 5.1  there exist a constant $C_0>0$, 
depending only on $T_0$,  and a constant $M_0$, depending only on $\del,T_0$, such that the solution $u$ to  the Burgers' equation (\ref{BH5}) satisfies
\be \label{BN5}
|u(z,t)| \ \le \ \ C_0+2y \quad {\rm for \ } z\ge M_0, \ 0<t\le t_0 \ .
\ee
We may integrate (\ref{BH5}) on the domain $\mathcal{D}(z_0,\ve)$  with $z_0>M_0+\sqrt{\ve t_0}$ by using the Green's function (\ref{BL5}). We  obtain the integral equation 
\begin{multline} \label{BO5}
u(z,t) \ = \ \int_{-\sqrt{t_0}}^{\sqrt{t_0}}G_D((z-z_0)/\sqrt{\ve},z', t)u(\sqrt{\ve}z'+z_0,0) \ dz' \\
+ 2\sum_{z'=\pm\sqrt{t_0}} p(z')\int_0^t \frac{\pa G_D((z-z_0)/\sqrt{\ve},z',(t-s))}{\pa z'} u(\sqrt{\ve} z'+z_0,s) \ ds \\
+\frac{1}{2\sqrt{\ve}} \int_0^t  \int_{-\sqrt{t_0}}^{\sqrt{t_0}}\frac{\pa G_D((z-z_0)/\sqrt{\ve},z',(t-s))}{\pa z'}\ga(s) u(\sqrt{\ve} z'+z_0,s)^2 dz' \ ds  \ , 
\quad [z,t]\in \mathcal{D}(z_0,\ve) \ ,
\end{multline}
where $p(z')=-1$ if $z'=\sqrt{t_0}$ and $p(z')=1$ if $z'=-\sqrt{t_0}$.  We can use the representation (\ref{BO5}) and the bound (\ref{BN5}) to obtain a bound on $\pa u(z,t)/\pa z$ at $z=z_0, \ 0<t\le t_0$, which is independent of $z_0$ as $z_0\ra\infty$.  Observe from (\ref{AV5})  that $u(\cdot,0)\equiv 2y$ is constant. Hence if $u_1(z,t)$ denotes the first term on the RHS of (\ref{BO5}) we have from (\ref{BK5}), (\ref{BL5}) the inequality
\be \label{BP5}
\left|\frac{\pa u_1(z,t)}{\pa z}\right| \ \le \ \frac{C_1y}{\sqrt{\ve}} \ , \quad |z-z_0|<\frac{\sqrt{\ve t_0}}{2}, \ 0<t\le t_0 \ ,
\ee
where $C_1$ depends only on $t_0$.  Letting $u_2(z,t)$ be the second (boundary) term on the RHS of (\ref{BO5}), we may use (\ref{BN5}) to bound 
$\pa u_2(z,t)/\pa z$ at $z=z_0$. Thus we have that
\be \label{BQ5}
\left|\frac{\pa u_2(z,t)}{\pa z}\right| \ \le \ \frac{C_2(C_0+2y)}{\sqrt{\ve}} \ , \quad |z-z_0|<\frac{\sqrt{\ve t_0}}{2}, \ 0<t\le t_0 \ ,
\ee
where $C_2$ depends only on $t_0$.

To bound the derivative of the third term on the  RHS of (\ref{BO5}) we define an operator $\mathcal{L}$ on functions  $w:\mathcal{D}(0,1)\ra\mathbb{R}$ by
\be \label{BR5}
\mathcal{L} w(z,t) \ = \ \int_0^t\int_{-\sqrt{t_0}}^{\sqrt{t_0}}\frac{\pa G_D(z,z',t-s)}{\pa z'}\ga (s)w(z',s) \ dz' \ ds  \ ,  \quad [z,t]\in \mathcal{D}(0,1)
\ee
We see from (\ref{BK5}), (\ref{BL5}) there is a constant $C_3$, depending only on $t_0$ such that
\be \label{BS5}
\left|\frac{\pa^2 G_D(z,z',t)}{\pa z\pa z'}\right| \ \le \ \frac{C_3}{t}G(z-z',2t) \ , \quad [z,t],[z',t]\in \mathcal{D}(0,1) \ .
\ee
It follows from (\ref{BS5}) that
\be \label{BT5}
\left|\frac{\pa G_D(z_1,z',t)}{\pa z'}-\frac{\pa G_D(z_2,z',t)}{\pa z'}\right| \ \le \ \frac{C_3|z_1-z_2|}{t}\int_0^1G(\la(z_1-z')+(1-\la)(z_2-z'),2t) \ d\la  \ .
\ee
Let $U_{z_1,z_2}=\{z'\in\mathbb{R}:  \ |z'|<\sqrt{t_0}, \ |z_1-z'|\ge 2|z_1-z_2|\}$. It is evident that 
\be \label{BU5}
|z_1-z_2| \ \le \ |\la(z_1-z')+(1-\la)(z_2-z')|  \quad {\rm for \ } 0<\la<1, \ z'\in U_{z_1,z_2} \ .
\ee
Next we write
\be \label{BV5}
\mathcal{L}w(z_1,t)-\mathcal{L}w(z_2,t) \ = \ F_1(z_1,z_2,t)+F_2(z_1,z_2,t) \ , \quad {\rm where}
\ee
\be \label{BW5}
F_1(z_1,z_2,t) \ = \ \int_0^t\int_{U_{z_1,z_2}}\left[    \frac{\pa G_D(z_1,z',t-s)}{\pa z'}-\frac{\pa G_D(z_2,z',t-s)}{\pa z'}\right]\ga (s)w(z',s) \ dz' \ ds  \ .
\ee
It follows from (\ref{BT5}), (\ref{BU5}) that  for all $\al$ satisfying $0<\al<1$ there is a constant $C_4$ such that
\be \label{BX5}
|F_1(z_1,z_2,t)| \ \le \ C_4|z_1-z_2|^\al\|w\|_\infty\int_0^t\frac{ds}{ s^{(1+\al)/2}} \
= \ \frac{2C_4}{1-\al}|w\|_\infty\sqrt{ t}\left(\frac{|z_1-z_2|}{\sqrt{ t}}\right)^\al \ .
\ee
To bound $F_2(z_1,z_2,t)$ we use the inequalities
\be \label{BY5}
\left|\frac{\pa G_D(z,z',t)}{\pa z'}\right| \ \le \  \frac{C_5}{\sqrt{ t}}G(z-z',2 t)  \ , \ [z,t],[z',t]\in \mathcal{D}(0,1) \ ,
\ee
\be \label{BZ5}
\int_{|z'|<a} G(z',2 t) \ dz' \ \le \ C_6\min\left\{\frac{a}{\sqrt{ t}},1\right\} \ ,
\ee
where $C_5,C_6$ are constants depending only on $t_0$. 
We have from (\ref{BY5}), (\ref{BZ5}) that
\begin{multline} \label{CA5}
\int_0^t\int_{[-\sqrt{t_0},\sqrt{t_0}]-U_{z_1,z_2}}   \left|\frac{\pa G_D(z_1,z',t-s)}{\pa z'}-\frac{\pa G_D(z_2,z',t-s)}{\pa z'}\right|    \ dz' \ ds \\
 \ \le \ 
2C_5C_6\int_0^t \frac{1}{\sqrt{t-s)}}\min\left\{\frac{3|z_1-z_2|}{\sqrt{t-s}},1\right\}  \ ds \\
 = \ 
6C_5C_6|z_1-z_2|\int_0^{ t/9|z_1-z_2|^2} \min\left\{\frac{1}{\sqrt{s'}},1\right\} \frac{ds'}{\sqrt{s'}} \
\le \ \frac{C_7}{1-\al}\sqrt{t}\left(\frac{|z_1-z_2|}{\sqrt{ t}}\right)^\al \quad {\rm if \ } \frac{|z_1-z_2|}{\sqrt{ t}} \ \le \  1 \ .
\end{multline}
We conclude from (\ref{BV5}), (\ref{BX5}), (\ref{CA5}) that
\be \label{CB5}
|\mathcal{L}w(z_1,t)-\mathcal{L}w(z_2,t)| \ \le \ C_\al\|w\|_\infty\sqrt{ t}\left(\frac{|z_1-z_2|}{\sqrt{t}}\right)^\al \quad {\rm if \ } \frac{|z_1-z_2|}{\sqrt{ t}} \ \le \  1 \ ,
\ee
where the constant $C_\al$ depends only on $\al, t_0$. 
It follows from (\ref{BR5}), (\ref{BY5})  that $\|\mathcal{L}w\|_\infty\le C\sqrt{ t}\|w\|_\infty$ for some universal constant $C$. Hence (\ref{CB5}) holds for all $[z_1,t],[z_2,t]\in\mathcal{D}(0,1)$. 

Next for a continuous function $f:[-\sqrt{t_0},\sqrt{t_0}]\ra\mathbb{R}$  and $\al$ satisfying $0<\al<1$  define the $\al$ H\"{o}lder norm of $f$ by 
\be \label{CC5}
\|f\|_{0,\al} \ = \  \sup\{|f(z)|: z\in[-\sqrt{t_0},\sqrt{t_0}] \ \} +\sup\left\{ \frac{|f(z_1)-f(z_2)|}{|z_1-z_2|^\al}  : z\in[-\sqrt{t_0},\sqrt{t_0}]  \right\} \ .
\ee
We may bound the derivative of $(\pa/\pa z) \mathcal{L}w(z,t)$ in terms of the norms (\ref{CC5}) for $w(\cdot,s), \ 0<s<t$.  To see this observe from (\ref{BR5})  that \begin{multline} \label{CD5}
\frac{\mathcal{L}w(z+h,t)-\mathcal{L}w(z,t)}{h} \ = \\
 \int_0^t\int_{-\sqrt{t_0}}^{\sqrt{t_0}}\int_0^1\frac{\pa^2 G_D(z+\la h,z',t-s)}{\pa z\pa z'}d\la  \ \ga (s)w(z',s) \ dz' \ ds  \\
 = \  \int_0^t\int_{-\sqrt{t_0}}^{\sqrt{t_0}}\int_0^1\frac{\pa^2 G_D(z+\la h,z',t-s)}{\pa z\pa z'}  \ \ga (s)[w(z',s)-w(z+\la h,s)] \ d\la \ dz' \ ds \ .
\end{multline}
We have from (\ref{BS5}), (\ref{CD5}) that
\be \label{CE5}
\left|\frac{\mathcal{L}w(z+h,t)-\mathcal{L}w(z,t)}{h}\right| \ \le \ C\int_0^t \frac{ds}{(t-s)^{1-\al/2}} \|w(\cdot,s\|_{0,\al} \ ds \ , \quad 0<\al<1 \ , 0<t\le t_0,
\ee
for some constant $C_8$ depending on $t_0$. 

Let $u_3(z,t)$ with $[z,t]\in\mathcal{D}(z_0,\ve)$ be the third term on the RHS of (\ref{BO5}).  Then 
$u_3(z,t)=\mathcal{L}w_1((z-z_0)/\sqrt{\ve},t)$ with $[z,t]\in\mathcal{D}(z_0,\ve)$, where  $w_1$ is given by the formula
\be \label{CF5}
w_1(z,t) \ = \ \frac{1}{2\sqrt{\ve}} u^2(\sqrt{\ve}z+z_0,t) \ , \quad [z,t]\in\mathcal{D}(0,1) \ .
\ee
We have then from (\ref{BN5}), (\ref{CB5}) that
\be \label{CG5}
\|\mathcal{L}w_1(\cdot,t)\|_{0,\al} \ \le \ \frac{C_\al}{2\sqrt{\ve}} [C_0+2y]^2 t^{(1-\al)/2} \ , \quad 0<t<t_0, \ 0<\al<1 \  .
\ee
Now defining $w_2;\mathcal{D}(0,1)\ra\R$ by $w_2((z-z_0)/\sqrt{\ve},t)=u(z_0+(z-z_0)/2,t)$, we have from (\ref{BP5}), (\ref{BQ5}), (\ref{CG5}) the bound 
\be \label{CH5}
\|w_2(\cdot,t)\|_{0,\al} \ \le \ C_9[C_0+2y]+ \frac{C_\al}{2\sqrt{\ve}} [C_0+2y]^2 t^{(1-\al)/2} \ , \quad 0<t<t_0, \ 0<\al<1 \  ,
\ee
for some constant $C_9$ depending only on $t_0$.  

The inequality (\ref{CH5}) show that the solution $u(\cdot,\cdot)$ of (\ref{BH5}) is H\"{o}lder continuous in the domain $\mathcal{D}(z_0,\ve/4)$. We represent $u(z,t)$ again as in (\ref{BO5})  but with $t_0$ replaced by $t_0/4$, whence (\ref{CH5}) implies that $u(\cdot,\cdot)$  is H\"{o}lder continuous in the domain $\mathcal{D}(z_0,\ve)$. Letting $w_3(z,t)=w_2(z,t)^2/2\sqrt{\ve}$, then we see that
\be \label{CI5}
u_3(z,t)=\mathcal{L}w_3((z-z_0)/\sqrt{\ve},t) \ ,  \quad \|w_3(\cdot,t)\|_{0,\al} \ \le \ \frac{1}{\sqrt{\ve}}\|w_2(\cdot,t)\|^2_{0,\al} \ .
\ee
It follows from (\ref{CE5}), (\ref{CH5}), (\ref{CI5}) that
\be \label{CJ5}
\left|\frac{\pa u_3(z,t)}{\pa z}\right| \ \le \  C \quad {\rm for \ } |z-z_0|<\sqrt{\ve t_0}/2,  \ 0<t<t_0/2 \ ,
\ee
for some constant depending on $\ve,y,t_0$, but not on $z_0$.  The inequality (\ref{BB5}) follows from (\ref{BP5}), (\ref{BQ5}), (\ref{CJ5}). 
\end{proof}
\begin{theorem}
Assume the function $A:[0,\infty)\ra\mathbb{R}$ is continuous non-negative, and let $q_\ve(x,y,T)$ be defined by (\ref{R2}), (\ref{U2}). Then for all $y,T>0$ the function $x\ra q_\ve(x,y,T)$ is concave.
\end{theorem}
\begin{proof}
We show that the function  $x\ra q_{\ve,\del}(x,y,T)$ is concave, and then the result follows from Lemma 5.1 by letting $\del\ra 0$.  We define the function 
$w_{\ve,\del}$ by
\be \label{CK5}
w_{\ve,\del}(x,y,T) \ = \ \frac{\sig^4_A(T)}{m_{1,A}(T)^2}\frac{\pa^2 q_{\ve,\del}(x,y,T)}{\pa x^2} \ .
\ee
By differentiating the PDE (\ref{C5})  we see from Proposition 3.1 that the function $[x,T]\ra w_{\ve,\del}(x,y,T)$ is a classical solution of the PDE 
\begin{multline} \label{CL5}
\frac{\pa w_{\ve,\del}(x,y,T)}{\pa T}  + \left[\la(x,y,T)+ \frac{\pa q_{\ve,\del}(x,y,T)}{\pa x}\right]\frac{\pa w _{\ve,\del}(x,y,T)}{\pa x} \\
+\frac{m_{1,A}(T)^2}{\sig^4_A(T)}w_{\ve,\del}(x,y,T)^2 \  = \ \frac{\ve}{2}\frac{\pa^2 w_{\ve,\del}(x,y,T)}{\pa x^2} \ .
\end{multline}
We proceed as in the proof of Proposition 5.1 using Ito's lemma and the martingale optional sampling theorem. 
Thus  similarly to (\ref{D5}) we have the representation 
\begin{multline} \label{CM5}
w_{\ve,\del}(x,y,T) \ = \ E\left[ w_{\ve,\del}(X^*_{\ve,\del}(\del\vee\tau^*_{\ve,\del,x,T,K}),y,\del\vee\tau^*_{\ve,\del,x,T,K}) \ |  \ X^*_{\ve,\del}(T)=x\right]  \\
- E\left[      \int^T_{  \del\vee\tau^*_{\ve,\del,x,T,K}}        \frac{m_{1,A}(s)^2}{\sig^4_A(s)}w_{\ve,\del}(X^*_{\ve,\del}(s),y,s)^2   \ ds            \right] \ ,
\end{multline}
where $X^*_{\ve,\del}(\cdot)$ is the solution to the SDE (\ref{Y2}) with drift $\mu_\ve(x,y,T)$ given by the coefficient of $\pa w _{\ve,\del}(x,y,T)/\pa x$ in (\ref{CL5}). 
The stopping time $\tau^*_{\ve,\del,x,T,K}$ is the first exit time of $X^*_{\ve,\del}(s), \ s\le T,$ with $X^*_{\ve,\del}(T)=x$ from the interval $(0,K)$.  From (\ref{AV5}) we have that $w_{\ve,\del}(\cdot,y,\del)\equiv 0$, whence the first term on the RHS of (\ref{CM5}) can be written as
\begin{multline} \label{CN5}
E\left[    w_{\ve,\del}(0,y,\tau^*_{\ve,\del,x,T,K}) ; \   \tau^*_{\ve,\del,x,T,K}>\del, X_{\ve,\del}(\tau^*_{\ve,\del,x,T,K} )=0           \right] \\
+ E\left[    w_{\ve,\del}(K,y,\tau^*_{\ve,\del,x,T,K}) ; \   \tau^*_{\ve,\del,x,T,K}>\del, X_{\ve,\del}(\tau^*_{\ve,\del,x,T,K} )=K           \right] \ .
\end{multline}
Using Lemma 5.1 and arguing as in the proof of Corollary 3.1, we see that $w_{\ve,\del}(0,y,s)\le0$ for $s>\del,y>0$. Hence the first term in (\ref{CN5}) is non-positive. The second term converges to $0$ as $K\ra\infty$.  In order to prove this we  use Lemma 5.2, which yields a uniform upper bound on 
$|w_{\ve,\del}(K,y,s)|,\  \del<s<T,$ as $K\ra\infty$.  Then we follow the corresponding  argument around (\ref{K5}) in the proof of Proposition 5.1.  By letting $K\ra\infty$ in (\ref{CM5}) we conclude that $w_{\ve,\del}(x,y,T)\le 0$ for $x\ge 0,y>0,T>\del$. 
\end{proof}

\vspace{.1in}
\section{Convergence of the function  $\pa q_\ve(x,y,T)/\pa x$ as $\ve\ra0$} 
In this section we assume the function $A(\cdot)$ is non-negative, whence the results of Proposition 4.2 and Corollary 4.1 imply that the function 
$x\ra q_0(x,y,T)$ of (\ref{AB2}), (\ref{B4}) is $C^1$ for certain ranges of $[x,T]$. We will show that $\lim_{\ve\ra0}\pa q_\ve(0,y,T)/\pa x=\pa q_0(0,y,T)/\pa x$. In view of the upper bound (\ref{AC3}), we only need to prove for small $x$ a lower bound for $q_\ve(x,y,T)$ in terms of $q_0(x,y,T)$ and a correction term which 
goes to $0$ as  $\ve\ra 0$. We already obtained such a lower bound in Lemma 2.4. Our starting point was the inequality (\ref{CE2}), which leads to the inequality (\ref{CF2}).  However  the second term on the RHS of (\ref{CF2}) is not sufficient for our purposes since we need  the correction to be bounded by a constant times $x$ as $x\ra 0$.  Instead of (\ref{CE2}) we observe  from (\ref{CB2})-(\ref{CE2}) that
\be \label{A6}
\frac{1}{2}\int_{\tau^*_{\ve,x,T}}^T\left[\mu^*_\ve(X^*_\ve(s),y,s)-\la(X^*_\ve(s),y,s)\right]^2 \ ds \ \ge \ 
q_0(x,y,T)-q_0(\sqrt{\ve}Z_\ve,y,\tau^*_{\ve,x,T}) \ .
\ee
The function  $[x,T]\ra q_0(x,y,T)$ is defined by (\ref{AB2}) for $x,T>0$, and for $x<0,T>0$ by
\begin{multline} \label{B6}
q_0(x,y,T) \ = \\
- \min\left\{\frac{1}{2}\int^\tau_T\left[\frac{dx(s)}{ds}-\la(x(s),y,s)\right]^2 \ ds \ \Big| \ \tau>T, \ x(T)=x, \ x(\cdot)<0, \  \ x(\tau)=0 \right\} \ .
\end{multline}
Assuming the function $[x,T]\ra q_0(x,y,T)$ defined by (\ref{AB2}), (\ref{B6}) is sufficiently differentiable at $[0,T]$, we can do a Taylor expansion,
\begin{multline} \label{C6}
q_0(\sqrt{\ve}Z_\ve,y,\tau^*_{\ve,x,T}) \  \ = \ \frac{\pa q_0(0,y,T)}{\pa x}\sqrt{\ve}Z_\ve+\ve Z_\ve^2\int_0^1\int_0^1\la \  d\la \ d\mu  \frac{\pa^2 q_0(\la\mu\sqrt{\ve}Z_\ve ,y,T)}{\pa x^2} \\
-[T-\tau^*_{\ve,x,T}]\int_0^1d\la \frac{\pa q_0(\sqrt{\ve}Z_\ve,y,\la\tau^*_{\ve,x,T}+(1-\la)T)}{\pa t} \ .
\end{multline} 
Then we can estimate the expectation of the correction term in (\ref{A6}) by estimating the expectation of each term on the RHS of (\ref{C6}). 

We may obtain a formula similar to (\ref{B4})  for  $q_0(x,y,T), \ x<0,$ defined by (\ref{B6}).  If the minimization in (\ref{B6}) is for {\it fixed} $\tau>T$ then there is a unique minimizing trajectory $x(s), \ T<s<\tau,$ given by (\ref{AK2}), where $\ga(\tau)$ is chosen so that $x(\tau)=0$.  The functions $s\ra g_{1,A}(s,T), \ g_{2,A}(s,T)$, which were defined in (\ref{AM2}), (\ref{AN2})  for $0<s<T$ may be extended by the same formulas to $s>T$. Similarly we may extend the function $s\ra g_{3,A}(s,T)$  by
using (\ref{L4}).   Note that the functions $s\ra g_{1,A}(s,T), \ g_{3,A}(s,T), \ s>T,$ are negative. We have then from (\ref{B6}) that
 \be \label{D6}
q_0(x,y,T) \ = \ -\min_{\tau>T} \frac{|g_{3,A}(\tau,T)|}{2} \left[y+g_{1,A}(\tau,T)x+g_{2,A}(\tau,T)\right]^2 \ .
\ee 
When $A(\cdot)\equiv0$ the formula (\ref{D6}) becomes
\be \label{E6}
q_0(x,y,T) \ = \ -\min_{\tau>T} \frac{\tau-T}{2\tau T} \left[y+\frac{\tau x}{(T-\tau)}\right]^2 \ 
= \ -\frac{1}{2T}\left[-2xy+\min_{\al>1}\{\al x^2+y^2/\al\} \right] \ .
\ee
Hence we have that
\begin{eqnarray} \label{F6}
q_0(x,y,T) \ &=& \ \frac{2xy}{T}  \ \ {\rm if \ } -y<x<0, \\
q_0(x,y,T) \ &=& \ -\frac{(x-y)^2}{2T}  \ \ {\rm if \ } x<-y \ . \label{G6}
\end{eqnarray} 
For $-y<x<0$ the minimizing $\tau>T$ in (\ref{D6}) is given by  $\tau(x,y,T) \ = \ yT/(x+y)$. Otherwise the minimum is obtained  by letting $\tau\ra \infty$.  The function $[x,T]\ra q_0(x,y,T)$ defined by (\ref{F6}), (\ref{G6})  is a $C^1$ solution to the HJ equation (\ref{AC2}) in the region  $\{[x,T]: \ x<0,T>0\}$.  However the second derivative $\pa^2 q_0(x,y,T)/\pa x^2$ is discontinuous across the boundary $\{[x,T]: x=-y, T>0\}$. The characteristics which yield the function (\ref{F6}) are the same as in the situation $x>0$ studied in $\S4$, and are given by $x(\tau,s)=(s-\tau)y/\tau, \ s>0$. Then we have
\be \label{H6}
q_0(x,y,T) \ = \  \frac{1}{2}\int_\tau^T p(\tau, s)^2 \ ds \ = \ 2\int_\tau^T \frac{y^2}{s^2} \ ds \ = \ \frac{2xy}{T} \ , \quad x=x(\tau,T) \ .
\ee
This set of characteristics covers the region $\{[x,T]: \ x>-y, \ T>0\}$ without intersecting, but all characteristics converge to the point $[-y,0]$.  The characteristics which yield the function (\ref{G6})  are given by $x(\la,s)=\la, \ \la<-y, \ s>0$.  In that case
\be \label{I6}
q_0(x,y,T) \ = \ \frac{1}{2}\int_\infty^T p(\la,s)^2 \ ds \ = \ \frac{1}{2}\int_\infty^T \frac{(\la-y)^2}{s^2} \ ds \ = \ -\frac{(x-y)^2}{2T} \ , \quad x=x(\la,T) \ .
\ee
This set of characteristics covers the region $\{[x,T]: \ x<-y, \ T>0\}$, also without intersecting.

We shall show that the function $[x,T]\ra q_0(x,y,T)$ defined by (\ref{AB2}), (\ref{B6}) is differentiable for $[x,T]$ in a neighborhood of the initial line $\{[0,T]: \ T>0\}$, by proving 
 it may be obtained via the method of characteristics.  To do this we extend the domain $\mathcal{D}_{y,T_0}$ defined just prior to (\ref{BL4}). It follows from (\ref{AF4}) that the characteristics $s\ra x(\tau,s)$ do not meet if $0<s<\tau\le T_0$.  Hence we may extend the domain $\mathcal{D}_{y,T_0}$  of $\S4$ to include 
 the set $\{[x,T]: \ 0<T<T_0,  \ x(T_0,T)<x\le 0\}$, and similarly extend the region $\mathcal{U}_{y,T_0}$.  The inequality (\ref{AN4}) for the characteristic $s\ra x(\tau,s), \ \tau\le  s\le T_0,$ continues to hold for $0<s<\tau$.  More precisely, we have from (\ref{BJ4}), (\ref{BK4}) that 
\begin{multline} \label{T6}
C_1(s-\tau)\left[\frac{[y+g_{2,A}(\tau,\tau)]}{\tau}+(\tau-s)\right]  \ \le \ x(\tau,s) \\
\le \ c_1(s-\tau)\left[\frac{[y+g_{2,A}(\tau,\tau)]}{\tau}+c_2(\tau-s)\right] \ , \quad 0<s<\tau\le T_0 \ ,
\end{multline}
where $C_1,c_1,c_2$ depend only on $T_0$ and $\sup_{0\le t\le T_0} A(t)$. 
Hence there is a constant $\La_3>0$, depending only on $T_0$ and $\sup_{0\le t\le T_0} A(t)$,  such that if
\be \label{J6}
0<T\le 3T_0/4 \ \ {\rm and \ } \La_3[y+g_{2,A}(T_0,T_0)] \ < x \ \le 0,  \ \ {\rm then \ } [x,T]\in\mathcal{D}_{y,T_0} \ .
\ee
We extend the results of Proposition 4.2 to include $[x,T]$ in the region (\ref{J6}). 
\begin{lem}
The results of Proposition 4.2, with $q_0(x,y,T)$ defined by (\ref{AS4}), continue to hold in the extended region $\mathcal{D}_{y,T_0}$.  Also $\tau=\tau(x,y,T)$ in (\ref{AS4})  is the unique minimizer in the variational problems (\ref{B6}), (\ref{D6}) for $[x,T]$ with $x<0, \ 0<T <T_0/2$,   in the following regions: 
 (a) $-\La[y+g_{2,A}(T,T)]<x<0$ if $4\sqrt{\La_0 y}/3\ge T$, otherwise  (b) $-\La [y+g_{2,A}(T,T)]^2/T^2<x<0$, where $\La>0$ is chosen sufficiently small depending only on $T_0$ and $\sup_{0\le t\le T_0} A(t)$. Therefore if $[x,T]$ is  in one of the regions (a), (b),  the functions (\ref{D6}) and (\ref{AS4}) are identical.
\end{lem}
\begin{proof}
All constants in the following can be chosen to depend only on $T_0$ and $\sup_{0\le t\le T_0} A(t)$.
It is clear from (\ref{T6}) that the characteristics $s\ra x(\tau,s), \ 0<s<\tau,$ satisfy $x(\tau,s)<0$, and from (\ref{AF4}) that  $D_\tau x(\tau,s)<0$. The differentiability properties of the function $[x,T]\ra q_0(x,y,T)$ and the fact that it is a solution to the HJ equation (\ref{AC2}) follow as in the proof of Proposition 4.2. We also have from (\ref{BC4}) that  
\be \label{K6}
-q_0(x,y,T) \ \le \  \frac{1}{2}\int^{\tau}_T \left[\frac{dx(s)}{ds}-\la(x(s),y,s)\right]^2 \ ds \  ,
\ee
for any path $s\ra x(s), \ T<s<\tau<T_0,$ in $\mathcal{D}_{y,T_0}$  with $x(T)=x<0, \ x(\tau)=0$.  Equality holds in (\ref{K6})  if $x(\cdot)$ is the characteristic.  

Let $F_0(x,y,\tau,T)$ be the function on the RHS of (\ref{D6}).  
We wish to find $[x,T]\in\mathcal{D}_{y,T_0}$ with $x<0$ such that $q_0$, defined by (\ref{AS4}), satisfies $-q_0(x,y,T)=\inf_{\tau>T} F_0(x,y,\tau,T)$. To do this we first observe from (\ref{E4}) that since $A(\cdot)$ is non-negative, the function $s\ra g_{2,A}(s,T), \ s>0,$ is increasing and hence non-negative.  We also have from 
(\ref{H2}) that the function $s\ra \sig^2_A(s,T), \ s>T,$ is negative and decreasing with $\lim_{s\ra T}\sig^2_A(s,T)=0$.  Letting  $\lim_{s\ra\infty}\sig^2_A(s,T)=\sig^2_A(\infty,T)$,  one sees in the case $A(\cdot)\equiv 0$   that $\sig^ 2_A(\infty,T)=-\infty$. It is however possible for some non-negative $A(\cdot)$ that  $\sig^ 2_A(\infty,T)>-\infty$. We have then from (\ref{D4}) that the function $s\ra g_{1,A}(s,T), \ s>T,$ is  increasing with $\lim_{s\ra T}g_{1,A}(s,T)=-\infty$,  and $g_{1,A}(s,T)<-1/m_{1,A}(T), \ s>T$.  It follows that  $\lim_{s\ra\infty}g_{1,A}(s,T)=g_{1,A}(\infty,T)\le -1/m_{1,A}(T)$.

Similarly to (\ref{BU4}) we consider $[x,T]\in\mathcal{D}_{y,T_0}$ which satisfies (\ref{J6}), and define $\tilde{q}_0(x,y,T)$  by 
 \be \label{L6}
 -\tilde{q}_0(x,y,T) \ = \  \min_{\tau>T} \frac{|g_{3,A}(\tau,T)|}{2} \left[y+g_{1,A}(\tau,T)x\right]^2 \ .
 \ee
  Using the identity (\ref{L4}) we see that the minimizing $\tau$ for the RHS of (\ref{L6}) is given by 
 \be \label{M6}
 g_{3,A}(\tau,T) \ = \ \frac{m_{1,A}(T)x}{\sig^2_A(T)y} \  , \quad {\rm if \ } \frac{y}{g_{1,A}(\infty,T)}<x<0 \ .
 \ee
  Substituting (\ref{M6}) into the RHS of (\ref{L6}) then yields the formula
 \be \label{N6}
 \tilde{q}_0(x,y,T) \ = \ \frac{2m_{1,A}(T)xy}{\sig^2_A(T)} \ ,
 \ee
which is the same as (\ref{BX4}). We should however note that the RHS of (\ref{N6}) is negative in this case. 
 Following the argument in the  proof of Proposition 4.2, we see that  if $\tau>T$ lies outside the region
 \be \label{O6}
 \frac{8m_{1,A}(T)x}{\sig^2_A(T)[y+g_{2,A}(T,T)]}  \ \le \ g_{3,A}(\tau,T) \ \le \ \frac{m_{1,A}(T)x}{8\sig^2_A(T)[y+g_{2,A}(T,T)]}  \ ,
 \ee
$F_0$ satisfies the  inequality
 \be \label{P6}
  F_0(x,y,\tau,T)   \ \ge \  -5\tilde{q}_0(x,y+g_{2,A}(T,T),T)/2  \ .
 \ee
 In concluding (\ref{P6}) we have used the fact that the function $s\ra g_{2,A}(s,T), \ s>T,$ is increasing. 
 
 We require $x<0$ to be sufficiently close to $0$ so that if $\tau>T$ with $0<T\le T_0/2$ lies in the region (\ref{O6}) then $\tau\le 3T/2\le3T_0/4$.  This is the case if $x>-\La_4[y+g_{2,A}(T,T)]$, where $\La_4>0$ is constant.    Next we determine how small $|x|$ needs to be so that  the minimizing paths $s\ra\Ga(\tau,s,T,x)=a(\tau,s,T)x+b(\tau,s,T),  \ T<s<\tau,$ for fixed $\tau$ defined by (\ref{BD4})   lie in $\mathcal{D}_{y,T_0}$.  To see this first note that the functions $s\ra a(\tau,s,T), \ b(\tau,s,T), \ T<s<\tau,$ are non-negative. Hence if $x<0$ is sufficiently small one may have $\Ga(\tau,s,T,x)>0$ for some $s\in(T,\tau)$. We see from (\ref{E4}), (\ref{F4}), (\ref{BD4})  there are constants $C_1,c_1,C_2>0$ such that
\be \label{Q6}
 C_1(\tau-s)\frac{x}{\tau-T}  \ \le \ \Ga(\tau,s,T,x) \ \le\  (\tau-s)\left[\frac{c_1x}{\tau-T}+C_2(s-T)\right]  \ , \quad T<s<\tau<T_0 \ .
\ee
Let $\La_0$ be the constant defined just after (\ref{AR4}), whence if $T<4\sqrt{\La_0 y}/3$ then  $\mathcal{D}_{y,T_0}$ contains the domain $\{[x,s]:  \ x>0, \ 0<s<3T/2\}$.  We assume that $\tau>T$ lies in the region (\ref{O6}) and  that $-\La_4[y+g_{2,A}(T,T)]<x<0$, whence $\tau\le 3T_0/4$. Choosing $\La_4$   to also satisfy the inequality $C_1\La_4<\La_3$, we see from (\ref{J6}), (\ref{Q6}) that if $T<4\sqrt{\La_0 y}/3$  then the path  $s\ra\Ga(\tau,s,T,x), \ T<s<\tau,$  lies in $\mathcal{D}_{y,T_0}$. In the case $T\ge4\sqrt{\La_0 y}/3$ we  observe that if $\tau>T$ satisfies (\ref{O6}) and $-\La_4[y+g_{2,A}(T,T)]<x<0$ then $\tau-T$ satisfies an inequality
\be \label{R6}
\tau-T \ \le \  \frac{C_3Tx}{[y+g_{2,A}(T,T)]} \ , \quad {\rm where \ } C_3  \ {\rm is  \ constant. }  
\ee
Hence the RHS of (\ref{Q6}) is negative provided $x<0$ satisfies the inequality
\be \label{S6}
|x| \ \le \ \frac{c_1[y+g_{2,A}(T,T)]^2}{C_2C_3^2T^2}  \ .
\ee
We conclude in this case that the path  $s\ra\Ga(\tau,s,T,x), \ T<s<\tau,$  lies in $\mathcal{D}_{y,T_0}$ provided $[x,T]$ satisfies (\ref{S6}).

Finally we need to show that if $\tau=\tau(x,y,T)$  then  $ F_0(x,y,\tau,T) <  -5\tilde{q}_0(x,y+g_{2,A}(T,T),T)/2$, which is the same as $q_0(x,y,T)\ge 5\tilde{q}_0(x,y+g_{2,A}(T,T),T)/2$. We first observe from (\ref{CV4}), (\ref{T6})  that since $x(\tau,s)<0, \ 0<s<\tau,$ we have  $\pa v(\tau,s)/\pa s \le 0$ for $T<s<\tau$ and $v(\tau,\tau)=0$. We conclude that
\be \label{Y6}
q_0(x,y,T) \ \le \ \tilde{q}_0(x,y+g_{2,A}(T,T),T) \ ,  \quad [x,T]\in\mathcal{D}_{y,T_0} \ , \ x<0 \ .
\ee
We assume now that $[x,T]$ lies in the domain   $\{[x,T]: \ 0<T<T_0/2,  \ -\La[y+g_{2,A}(T,T)]<x<0\}$, where $\La$  satisfies $0<\La\le \La_3$. It follows from (\ref{J6}) that $[x,T]\in\mathcal{D}_{y,T_0}$. Letting $x=x(\tau,T)$ we have from (\ref{T6}) that
\be \label{U6}
\tau-T \ \le \  \frac{\tau |x|}{c_1[y+g_{2,A}(\tau,\tau)]} \ \le \ \frac{\La \tau}{c_1} \ .
\ee
Choosing $\La\le c_1/3$ we see from (\ref{U6}) that $\tau-T\le \La T/[c_1-\La]\le T/2$. Observe from (\ref{AD4}) that
\be \label{V6}
1 \ \le \ \frac{y+g_{2,A}(\tau,\tau)}{y+g_{2,A}(T,T)} \ \le 1+C_4\La  
\ee
if  $T<4\sqrt{\La_0 y}/3$, where $C_4$ is constant. It follows then from (\ref{CV4}), (\ref{T6}), (\ref{V6}), upon using a lower bound for the integral of the RHS of (\ref{CV4}) on the interval $T<s<\tau$  similar to the one in (\ref{CW4}),  that 
\be \label{W6}
q_0(x,y,T) \ \ge \ [1+C_5\La]\tilde{q}_0(x,y+g_{2,A}(T,T),T) \ ,
\ee
where $C_5$ is constant. 
We assume  that $T\ge4\sqrt{\La_0 y}/3$ and that $[x,T]$  satisfies the inequality
\be \label{X6}
0<T<T_0/2, \quad -\frac{\La[y+g_{2,A}(T,T)]^2}{T^2} \ < \ x \ <0 \ .
\ee
Then the inequality (\ref{U6}) continues to hold, whence we see from (\ref{AD4}), (\ref{X6}) that (\ref{V6}) also holds. Similarly to before we see that
(\ref{W6}) holds in this case also. Now we choose $\La$ so that $C_5\La\le 1$.
\end{proof}
Observe from (\ref{BP3}) that we expect
\be \label{Z6}
\frac{\pa q_\ve(0,y,T)}{\pa x}-\frac{\pa q_0(0,y,T)}{\pa x} \ = \ \ve \frac{\pa^2 q_0(0,y,T)}{\pa x^2}\Big/\frac{\pa q_0(0,y,T)}{\pa x}+o(\ve) \quad {\rm as \ } \ve\ra 0.
\ee
This evidently suggests that the LHS of (\ref{Z6})  is $O(\ve)$ as $\ve\ra 0$.  Furthermore, we see from (\ref{W4}) that $\pa q_0(0,y,T)/\pa x\simeq y/T$, and 
from (\ref{AH4}) that $\pa^2 q_0(0,y,T)/\pa x^2\simeq T/y$ if $y\ge T^2$. Hence we expect the LHS of (\ref{Z6}) to be bounded by a constant times $\ve T^2/y^2+o(\ve)$.  
\begin{proposition}
Assume the function $A:[0,\infty)\ra\mathbb{R}$ is continuous non-negative, and let $q_\ve(x,y,T)$ be defined by (\ref{R2}), (\ref{U2}), and $q_0(x,y,T)$ by 
(\ref{AB2}), (\ref{B4}). Then for any $T_0>0$ there are constants $C_1,C_2$, depending only on $T_0$ and $\sup_{0\le t\le T_0} A(t)$, such that 
\be \label{AA6}
\left|\frac{\pa q_\ve(0,y,T)}{\pa x}-\frac{\pa q_0(0,y,T)}{\pa x}\right| \ \le \ \frac{C_2\ve }{y}\left[\frac{T^2}{y}+\frac{\ve T}{y^2}\right] \ , \quad 0<T\le T_0 \ ,
\ee
provided $0<T\le T_0, y\ge C_1 T^2,  \ve T<y^2$.
\end{proposition} 
\begin{proof} 
All constants in the following can be chosen to depend only on $T_0$ and $\sup_{0\le t\le T_0} A(t)$.
We define stopping times for the martingale  $s\ra Z(s), \ s<T,$ of (\ref{P2}). For $\del>0$ let $\tau_\del$ be given by
\be \label{AB6}
\tau_\del \ = \ \inf\left\{s: 0<s<T, \  \left|\frac{\sig^2_A(T)}{m_{1,A}(T)}Z(s')\right|<\del \ {\rm for \ } s<s'<T \ \right\} \ .
\ee
It is easy to see that $\tau_\del>0$ with probability $1$.  Since $\tau^*_{\ve,x,T}$, defined in the statement of Lemma 2.3, is also a stopping time for the martingale (\ref{P2}), it follows that
$\tilde{\tau}=\tilde{\tau}_{\ve,\del,x,T}=\tau_\del\vee\tau^*_{\ve,x,T}\vee(T/2)$ is a stopping time.
On  taking  expectations in (\ref{A6}) we have  from (\ref{BP2}) of Lemma 2.3 that for any $\del>0$,
\begin{multline} \label{AC6}
q_\ve(x,y,T) \ \ge \ \left[1-P\left(\tau^*_{\ve,x,T}<T/2 \right)-P\left(\tau^*_{\ve,x,T}<\tau_\del \right)\right] q_0(x,y,T)\\
 -E\left[ q_0(\sqrt{\ve}Z_\ve,y,\tau^*_{\ve,x,T}); \ \tau^*_{\ve,x,T}>\max\{\tau_\del,T/2\} \ \right] \ .
\end{multline}
Since the function $s\ra \sig^2_A(s)/m_{1,A}(s), \ s>0,$ is increasing, we see from (\ref{CE2}) that $|Z_\ve|\le \sig^2_A(T)|Z(\tau^*_{\ve,x,T})|/m_{1,A}(T)$. Hence if $\del$ is small enough we may use the Taylor expansion (\ref{C6}) to estimate the second expectation on the RHS of (\ref{AC6}).  
It follows from Proposition 4.2 and Lemma 6.1 that it is sufficient for $\del$ to lie in the interval
\be \label{AD6}
\sqrt{\ve} \del \le \La\min\left\{\frac{ y}{T^2},1\right\}[y+g_{2,A}(T/2,T/2)] \ , \quad {\rm where \ } \ \La \ {\rm depends \ only \ on \ } \ T_0 \ .
\ee

We show that
\be \label{AE6}
\lim_{x\ra0} P\left(\tau^*_{\ve,x,T}<T/2 \right) \ = \  0, \quad  \lim_{x\ra 0} P\left(\tau^*_{\ve,x,T}<\tau_\del \ \right)  \ = \  0 \ .
\ee
To prove the first limit in (\ref{AE6}) we use the inequality $\pa q_\ve(x,y,T)/\pa x\ge 0$, whence it follows that $X^*_\ve(s)\le X_\ve(s), \ 0<s<T,$ where $X_\ve(\cdot)$ is defined by (\ref{P2}). This was already observed just prior to Lemma 2.3.  We have from (\ref{L2}) that one can choose a constant $\nu$ with $0<\nu<1/2$, such that  $y_{\rm class}(s)<2x$ if $0<T-s<T\nu x/(y+T^2), $ provided $x<y+T^2, \ 0<T\le T_0$.   The first limit in (\ref{AE6}) follows if we can show that
\be \label{AF6}
\lim_{x\ra0}P\left(\sqrt{\ve}\inf_{T-T\nu x/(y+T^2)<s<T} Z(s)>-\frac{2m_{1,A}(T/2)x}{\sig^2_A(T/2)}\right) \ = \ 0 \ ,
\ee
since $P\left(\tau^*_{\ve,x,T}<T-T\nu x/(y+T^2)\right)$ is smaller than the probability in (\ref{AF6}). The limit in (\ref{AF6}) follows from the reflection principle. Similarly to (\ref{BZ2}) we have that the probability in (\ref{AF6}) is bounded in terms of a probability for the standard normal variable $Y$ by
\be \label{AG6}
P\left(|Y|<2x/\sqrt{\ve}\sig(x)\right)\ , \quad{\rm where \ }  \sig(x)^2 \ \ge \ \frac{c\nu Tx}{(y+T^2)} \ ,
\ee
 with $c>0$ a constant.  Since the probability (\ref{AG6}) is bounded by a constant times $\sqrt{x}$ the limit (\ref{AF6}) follows.  To prove the second limit
 in (\ref{AE6}) we argue similarly, using the inequality $P\left(\tau^*_{\ve,x,T}<\tau_\del\ \right)\le P\left(\tau^*_{\ve,x,T}< T-T\nu x/(y+T^2)\ \right)+
 P\left(\tau_\del>T-T\nu x/(y+T^2) \ \right)$.  Using the reflection principle again we see that \\
 $\lim_{x\ra 0}P\left(\tau_\del>T-T\nu x/(y+T^2) \ \right)=0$.
 
 We estimate the contribution of the first term in the Taylor expansion (\ref{C6}) to the expectation in (\ref{AC6}).  To do this we use the inequality
 \be \label{AH6}
\left|Z_\ve - \frac{\sig^2_A(T)}{m_{1,A}(T)}Z(\tau^*_{\ve,x,T})\right| \ \le \ C_1\left[T-\tau^*_{\ve,x,T}\right]\left|Z(\tau^*_{\ve,x,T})\right|  \ ,
\ee
 where $C_1$ is a constant.  Using the fact that 
 \be \label{AI6}
 s\ra Z(s)^2-\int_s^T \frac{m_{1,A}(s')^2}{\sig_A^4(s')}  \ ds' \ , \quad 0<s<T \ ,
 \ee
is a martingale, we have from (\ref{AI6}) and the optional stopping theorem that
\begin{multline} \label{AJ6}
E\left[ Z(\tau^*_{\ve,x,T})^2  \ ; \ \tau^*_{\ve,x,T}>T/2\ \right] \ \le \  \frac{C_2}{T^2}\left\{ E\left[T-\tau^*_{\ve,x,T}\right] +TP\left(\tau^*_{\ve,x,T}<T/2\right)\right\} \\
\le \frac{3C_2}{T^2} E\left[T-\tau^*_{\ve,x,T}\right] \ ,
\end{multline} 
where $C_2$ is constant.  In order to bound the RHS of (\ref{AJ6})  we use the lower bound (\ref{A5}) of Proposition 5.1. Recalling the definition of
$\tau_{\ve,{\rm linear},x,T}$ after (\ref{O5}),  we have  that $E\left[T-\tau^*_{\ve,x,T}\right] \ \le \ E\left[T-\tau_{\ve,{\rm linear},x,T}\right]$. We obtain then from (\ref{AH5}) an upper bound for $E\left[T-\tau^*_{\ve,x,T}\right]$, provided $y\ge C_3T^2$ where $C_3$ is constant. We may also obtain an inequality $E\left[(T-\tau^*_{\ve,x,T})^2\right]\le v_\ve(x)$, where  $v_\ve(\cdot)$ is the solution to a boundary value problem. Thus
\be \label{AK6}
-\frac{\ve}{2}\frac{d^2v_\ve(x)}{dx^2}+\frac{y}{C_5T} \frac{dv_\ve(x)}{dx} \ = \ 2u_\ve(x) \ ,  \ \ x>0, \quad v_\ve(0)=0 \ ,
\ee
where $u_\ve(\cdot)$ is the solution to (\ref{AJ5}).  Evidently we have that
\be \label{AL6}
v_\ve(x) \ = \ \left( \frac{C_5T }{y}\right)^3x\left[\ve+  \frac{xy}{C_5T}\right] \ .
\ee
It follows from (\ref{AH6}), upon using the Schwarz inequality and (\ref{AJ6})-(\ref{AL6}), that 
\be \label{AM6}
\sqrt{\ve} \limsup_{x\ra 0} \frac{1}{x} E\left[ \ \left|Z_\ve - \frac{\sig^2_A(T)}{m_{1,A}(T)}Z(\tau^*_{\ve,x,T})\right| ; \ \tau^*_{\ve,x,T}>T/2 \  \right] \ \le C_4\ve \frac{T}{y^2} \ ,
\ee
where $C_4$ is constant.

Applying the optional sampling theorem to the martingale $s\ra Z(s), \ 0<s<T,$ we have that
\begin{multline} \label{AN6}
\frac{\sig^2_A(T)}{m_{1,A}(T)}\left|E\left[ Z(\tau^*_{\ve,x,T})  \ ; \ \tau^*_{\ve,x,T}>\max\{T/2,   \tau_\del\} \ \right] \right| \\ \le \
\frac{\sig^2_A(T)}{m_{1,A}(T)}E[|Z(T/2)|;  \ \tau^*_{\ve,x,T}< T/2]+ \del P\left(\tau^*_{\ve,x,T}<\tau_\del\right) \  .
\end{multline}
To bound the first term on the RHS of (\ref{AN6}) we observe  that 
\be \label{AO6}
 \sqrt{\ve} \frac{\sig^2_A(T)}{m_{1,A}(T)}E[|Z(T/2)|;  \ \tau^*_{\ve,x,T}< T/2] \ \le \  C_5E\left[  |X_\ve(T/2)-x_{\rm class}(T/2)|; \    \tau_{\ve,{\rm linear},x,T}<T/2\right]  \    ,             
\ee
where $X_\ve(\cdot)$ is given by (\ref{Q5}) and $C_5$ is a constant.  From (\ref{R5}) and the Chebyshev inequality we have, using the inequality $y\ge C_3T^2$,  that
\be \label{AP6}
\left|x_{\rm class}(T/2)|\right| P\left( \tau_{\ve,{\rm linear},x,T}<T/2\right) \ \le \ \frac{C_6yv_\ve(x)}{T^2} \ ,
\ee
where $C_6$ is a constant.  Letting $X^*_\ve(s), \ s<T,$ be the diffusion with drift $\mu_\ve(x,y,T)=y/C_5T$ defined just after (\ref{AH5}),  we have that
\be \label{AQ6}
E\left[  X_\ve(T/2); \    \tau_{\ve,{\rm linear},x,T}<T/2\right]  \ \le \ E\left[  X^*_\ve(T/2); \    \tau^*_{\ve,{\rm linear},x,T}<T/2\right] \ .
\ee
Let $[x,t]\ra u_\ve(x,t), \ x>0, \ t>0,$ be the solution to the PDE
\be \label{AR6}
\frac{\pa u_\ve(x,t)}{\pa t}  \ = \ -\mu \frac{\pa u_\ve(x,t)}{\pa x} +\frac{\ve}{2}\frac{\pa^2u_\ve(x,t)}{\pa x^2} \ , \quad x>0, \ t>0, 
\ee
with boundary and initial conditions
\be \label{AS6}
u_\ve(0,t)= 0, \ \ t>0, \quad u_\ve(x,0)=x, \ x>0  .
\ee
Then one has
\be \label{AT6}
u_\ve(x,T-t) \ = \ E\left[  X^*_\ve(t); \    \tau^*_{\ve,{\rm linear},x,T}<t\right] \ , \quad t<T, \quad {\rm when \ } \mu=\frac{y}{C_5T} \ .
\ee
The solution to (\ref{AR6}), (\ref{AS6}) is given by
\be \label{AU6}
u_\ve(x,t) \ = \ \int_0^\infty G_{\ve,D}(x,x',t) x' \ dx' \ ,
\ee
where the Green's function $G_{\ve,D}$ has the formula
\be \label{AV6}
G_{\ve,D}(x,x',t) \ = \ \frac{1}{\sqrt{2\pi\ve t}}\exp\left[-\frac{(x-x'-\mu t)^2}{2\ve t}\right]\left\{1-\exp\left[-\frac{2xx'}{\ve t}\right]\right\} \ .
\ee
Using the inequality $1-e^{-z}\le z, \ z\ge 0,$ we have from (\ref{AU6}), (\ref{AV6}) that
\
\begin{multline} \label{AW6}
\limsup_{x\ra 0}\frac{1}{x} u_\ve(x,t) \\
 \le \  \frac{2}{\sqrt{2\pi}(\ve t)^{3/2}}\exp\left[-\frac{\mu^2 t}{2\ve}\right]\int_0^\infty e^{-\mu x'/\ve} x'^2 \ dx' \ = \ 
 \frac{4}{\sqrt{2\pi}} \left(\frac{\ve}{\mu^2 t} \right)^{3/2}\exp\left[-\frac{\mu^2 t}{2\ve}\right]\  \ .
\end{multline}
We conclude from (\ref{AL6}) and (\ref{AP6})-(\ref{AW6}), that the first term on the RHS of (\ref{AN6}) is bounded as
\be \label{AX6}
\sqrt{\ve}\limsup_{x\ra 0}\frac{1}{x}\frac{\sig^2_A(T)}{m_{1,A}(T)}E[|Z(T/2)|;  \ \tau^*_{\ve,x,T}< T/2] \ \le  \frac{C_7\ve T}{y^2} \quad {\rm if \ } \ve T\le y^2 , 
\ee
where $C_7$ is a constant. 

To bound the second term on the RHS of (\ref{AN6}) we use the inequality $ P\left(\tau^*_{\ve,x,T}<\tau_\del\right)\le P\left(\tau^*_{\ve,x,T}<\nu T\right)+
 P\left(\nu T<\tau^*_{\ve,x,T}<\tau_\del\right)$, where $1/2<\nu<1$ and $\nu$ is a suitably chosen constant.   We may then bound $\sqrt{\ve}\del P\left(\tau^*_{\ve,x,T}<\nu T\right)$ from the inequalities obtained in the previous paragraph. Assuming $y\ge T^2$, it follows from (\ref{AD6}) that we may take $\sqrt{\ve}\del=\La y$, whence  $\limsup_{x\ra 0}x^{-1} \sqrt{\ve}\del P\left(\tau^*_{\ve,x,T}<\nu T\right)\le C_8\ve T/y^2$ for some constant $C_8$. We estimate the second probability as
 \begin{multline} \label{AY6}
 P\left(\nu T<\tau^*_{\ve,x,T}<\tau_\del\right) \ \le \ P\left(\tau_\del>\nu T, \ \tau_{\ve,{\rm linear},x,T}<\tau_\del\right) \\
 \le \ P\left(\sup_{ \tau_{\ve,{\rm linear},x,T}\vee(\nu T)<s<T} |X_\ve(s)-x_{\rm class}(s)|>c_8\sqrt{\ve}\del \ \right) \ ,
 \end{multline}
 where $X_\ve(\cdot), \ x_{\rm class}(\cdot)$ are given in (\ref{Q5}), (\ref{R5}), and $c_8>0$ is a constant. We choose $\nu$ so that $-c_8\La y/2\le x_{\rm class}(s)\le x$ for $\nu T<s<T$.  Hence if $\nu T<s<T, \ 0<x<c_8\La y$ and $X_\ve(s)-x_{\rm class}(s)<-c_8\La y$ then $X_\ve(s)<0$. Since $X_\ve(s)>0$ for
 $\tau_{\ve,{\rm linear},x,T}<s<T$ it follows that if $s$ satisfies $\tau_{\ve,{\rm linear},x,T}\vee(\nu T)<s<T$ and $ |X_\ve(s)-x_{\rm class}(s)|>c_8\La y$, then
 $X_\ve(s)>c_8\La y/2$.  We see therefore, using the inequality $X^*_\ve(\cdot)\ge X_\ve(\cdot)$,  that the probability on the RHS of (\ref{AY6}) is bounded above by
 \be \label{AZ6}
   w_\ve(x) \ = \ P\left(\sup_{ \tau^*_{\ve,{\rm linear},x,T}<s<T} X^*_\ve(s)>c_8\La y/2 \right) \ .
 \ee
 The function $w_\ve:[0,c_8\La y/2]\ra [0,1]$ is the solution to the boundary value problem, 
\be \label{BA6}
-\frac{\ve}{2}\frac{d^2w_\ve(x)}{dx^2}+\frac{y}{C_5T} \frac{dw_\ve(x)}{dx} \ = \ 0 \ ,  \ \ 0<x<c_8\La y/2, \quad w_\ve(0)=0, \ w_\ve(c_8\La y/2)=1 \ ,
\ee
which has solution
\be \label{BB6}
w_\ve(x) \ = \ \left\{\exp\left[\frac{2xy}{C_5\ve T}\right]-1\right\}\Big/  \left\{\exp\left[\frac{c_8\La y^2}{C_5\ve T}\right]-1\right\} \ .
\ee
Taking limits in (\ref{BB6}) we see that
\be \label{BC6}
\La y\lim_{x\ra 0}\frac{1}{x}w_\ve(x) \ = \ \frac{2\La y^2}{C_5\ve T} \Big/  \left\{\exp\left[\frac{c_8\La y^2}{C_5\ve T}\right]-1\right\} \ \le  \frac{C_9\ve T}{y^2}  \ ,
\ee
for some constant $C_9$, provided $\ve T<y^2$. 

We conclude now a bound on the contribution of the first term in the Taylor expansion (\ref{C6}) to the RHS of (\ref{AC6}). Using the formula (\ref{W4}) and the inequality $y\ge C_3 T^2$,  we have that
\be \label{BD6}
\limsup_{x\ra 0}\frac{1}{x} \frac{\pa q_0(0,y,T)}{\pa x} \sqrt{\ve}\left|E[Z_\ve: \ \tau_\del\vee(T/2)<\tau^*_{\ve,x,T} \ ] \right|  \ \le \ \frac{C_{10}\ve}{y} \ ,
\ee
where $\del=\La y$ and $C_{10}$ is a constant.  To bound the contribution of the second term in the Taylor expansion (\ref{C6}) we  use (\ref{AL4}) to obtain the inequality 
$\left|\pa^2 q_0(x,y,T)/\pa x^2 \right|\ \le C_{11} T/y$, provided $|x|<\La y$ and $y\ge C_3 T^2$.  We have then from (\ref{AH5}), (\ref{AJ6}) that
\begin{multline} \label{BE6}
\limsup_{x\ra 0}\frac{1}{x}\ve \left| E\left[Z_\ve^2\int_0^1\int_0^1\la \  d\la \ d\mu  \frac{\pa^2 q_0(\la\mu\sqrt{\ve}Z_\ve ,y,T)}{\pa x^2};  \ \tau_\del\vee(T/2)<\tau^*_{\ve,x,T} \ \right]\right| \\
 \le \  \frac{C_{11}\ve T}{y}  \limsup_{x\ra 0}\frac{1}{x} E\left[Z_\ve^2;  \ \tau_\del\vee(T/2)<\tau^*_{\ve,x,T} \ \right] \ \le \ \frac{C_{12}\ve T^2}{y^2} \ ,
\end{multline}
where $C_{12}$ is a constant.  To bound the contribution of the final term in the Taylor expansion (\ref{C6}) we use the fact that $\pa q_0(0,y,t)/\pa t=0$, whence 
\be \label{BF6}
\frac{\pa q_0(x,y,t)}{\pa t} \ = \ x\int_0^1\frac{\pa^2 q_0(\mu x,y,t)}{\pa t\pa x'} \ d\mu \ ,
\ee
where we have from (\ref{AC2}) that
\be \label{BG6}
\frac{\pa^2 q_0(x',y,t)}{\pa t\pa x'}  \ = \ -\left[A(t)+\frac{1}{\sig_A^2(t)}\right]\frac{\pa q_0(x',y,t)}{\pa x'}-\left[\la(x',y,t)+\frac{\pa q_0(x',y,t)}{\pa x'}\right]\frac{\pa^2 q_0(x',y,t)}{\pa x'^2} \ .
\ee
It follows from (\ref{Z4}), (\ref{AL4})   there are constants $C_{11},C_{12}$ such that
\begin{eqnarray}  \label{BH6}
\left|\left[A(t)+\frac{1}{\sig_A^2(t)}\right]\frac{\pa q_0(x',y,t)}{\pa x'}\right| \ &\le& \ \frac{C_{11}y}{T^2} \ ,  \\
\left|\left[\la(x',y,t)+\frac{\pa q_0(x',y,t)}{\pa x'}\right]\frac{\pa^2 q_0(x',y,t)}{\pa x'^2}\right| \ &\le& \   C_{12} \label{BI6} \ ,
\end{eqnarray} 
provided $0<x'<\La y, \ T/2<t<T, $ and $y\ge C_3 T^2$.  Hence we may estimate the expectation of the final term in the Taylor expansion by combining our estimate on the expectation of the RHS of (\ref{AH6}) with   (\ref{BF6})-(\ref{BI6}). We obtain an inequality 
\begin{multline} \label{BJ6}
\limsup_{x\ra 0}\frac{1}{x} \left|E\left[       (T-\tau^*_{\ve,x,T})\int_0^1d\la \frac{\pa q_0(\sqrt{\ve}Z_\ve,y,\la\tau^*_{\ve,x,T}+(1-\la)T)}{\pa t} \ ;   \tau_\del\vee(T/2)<\tau^*_{\ve,x,T}        \right] \right|  \\
\le \ \frac{C_{13}\ve}{y} \quad {\rm for \ } \del=\La y, \ y\ge C_3T^2 \ .
\end{multline} 
It follows then from (\ref{BD6}), (\ref{BE6}), (\ref{BJ6}) that the LHS of (\ref{AA6}) is bounded by $C_{14}\ve/y$ provided $y\ge C_3 T^2$ and $\ve T\le y^2$. 

This bound on the LHS of (\ref{AA6}) may be improved by noting a cancellation in the Taylor expansion (\ref{C6}).  To see this we use the identity
\be \label{BK6}
\frac{\sig^2_A(T)}{m_{1,A}(T)} - \frac{\sig^2_A(\tau)}{m_{1,A}(\tau)} \ = \ \int_\tau^T \left[A(t)+\frac{1}{\sig_A^2(t)}\right] \frac{\sig^2_A(t)}{m_{1,A}(t)} \ dt \ .
\ee
Hence  using (\ref{BF6}), (\ref{BG6}), (\ref{BK6}) we have  that the sum of the first and third terms in the Taylor expansion (\ref{C6}) may be written as  
\begin{multline} \label{BL6}
\frac{\pa q_0(0,y,T)}{\pa x}\sqrt{\ve}Z_\ve
-\int_\tau^T dt \frac{\pa q_0(\sqrt{\ve}Z_\ve,y,t)}{\pa t} \\
= \ \frac{\pa q_0(0,y,T)}{\pa x}\sqrt{\ve}\frac{\sig^2_A(T)}{m_{1,A}(T)} Z(\tau) 
-\sqrt{\ve}\int_\tau^T dt \left[A(t)+\frac{1}{\sig_A^2(t)}\right]  \times \\
\left\{\frac{\sig^2_A(t)}{m_{1,A}(t)}\frac{\pa q_0(0,y,T)}{\pa x}-\frac{\sig^2_A(\tau)}{m_{1,A}(\tau)}\int_0^1d\mu\frac{\pa q_0(\mu \sqrt{\ve}Z_\ve,y,t)}{\pa x}\right\} Z(\tau) \\
+\sqrt{\ve}Z_\ve\int_\tau^Tdt\int_0^1d\mu\left[\la(\mu\sqrt{\ve}Z_\ve,y,t)+\frac{\pa q_0(\mu\sqrt{\ve}Z_\ve,y,t)}{\pa x'}\right]\frac{\pa^2 q_0(\mu\sqrt{\ve} Z_\ve,y,t)}{\pa x'^2} \ ,
\end{multline}
where $\tau=\tau^*_{\ve,x,T}$.  

To bound the expectation of the first term on the RHS of (\ref{BL6}) we observe that
\begin{multline} \label{BM6}
\limsup_{x\ra 0}\frac{1}{x}  \sqrt{\ve}\frac{\sig^2_A(T)}{m_{1,A}(T)}\left|E[Z(\tau^*_{\ve,x,T}): \ \tau_\del\vee(T/2)<\tau^*_{\ve,x,T} \ ] \right|  \\
 \le \ \limsup_{x\ra 0}\frac{1}{x} \left[u_\ve(x,T/2)+C_{15}y\left\{w_\ve(x)+P\left(\tau^*_{\ve,{\rm linear},x,T}<\nu T\right)\right\}\right] \ ,
\end{multline}
where we assume $y\ge C_3T^2, \ve T\le y^2,  \sqrt{\ve}\del=\La y$. Then from  (\ref{AW6}) and (\ref{BC6})  we see that $\limsup_{x\ra 0}x^{-1}[u_\ve(x,T/2)+yw_\ve(x)]\le C_{16}(\ve T/y^2)^2 $.  Instead of using the bound $P\left(\tau^*_{\ve,{\rm linear},x,T}<\nu T\right)\le v_\ve(x)/(1-\nu)^2T^2$ as in (\ref{AP6}),  we use the identity $P\left(\tau^*_{\ve,{\rm linear},x,T}<\nu T\right)=u_\ve(x,(1-\nu)T), $ where $[x,t]\ra u_\ve(x,t), \ x,t>0,$ is the solution to the PDE (\ref{AR6}) with boundary and  initial conditions $u_\ve(0,t)=0, \ u_\ve(x.0)=1$.  Similarly to (\ref{AW6}) we have now that 
\begin{multline} \label{BN6}
\limsup_{x\ra 0}\frac{1}{x} u_\ve(x,t) \\
 \le \  \frac{2}{\sqrt{2\pi}(\ve t)^{3/2}}\exp\left[-\frac{\mu^2 t}{2\ve}\right]\int_0^\infty e^{-\mu x'/\ve} x' \ dx' \ = \ 
 \frac{2}{\mu t\sqrt{2\pi}} \left(\frac{\ve}{\mu^2 t} \right)^{1/2}\exp\left[-\frac{\mu^2 t}{2\ve}\right]\  \ .
\end{multline}
It follows from (\ref{BN6}) that $y\limsup_{x\ra 0}x^{-1}P\left(\tau^*_{\ve,{\rm linear},x,T}<\nu T\right)\le C_{17}(\ve T/y^2)^2$. We conclude from (\ref{BM6})  that the 
$\limsup$ as $x\ra 0$ of $x^{-1}$ times the expectation
of the first term on the RHS of (\ref{BL6}) is bounded by $C_{18}\ve^2T/y^3$.  

Using (\ref{BI6}) and arguing as in the previous paragraph, we see from (\ref{AJ6}), (\ref{AL6}) and the Schwarz inequality  that the $\limsup$ as $x\ra 0$ of $x^{-1}$ times the expectation of the third term on the RHS of (\ref{BL6}) is bounded by $C_{19}\ve T^2/y^2$.  
The expectation of the second term on the RHS of (\ref{BL6}) is bounded by 
\begin{multline} \label{BO6}
\frac{C_{20}\sqrt{\ve}}{T}\frac{\pa q_0(0,y,T)}{\pa x} E\left[(T-\tau^*_{\ve,x,T})^2|Z(\tau^*_{\ve,x,T})|;   \tau_\del\vee(T/2)<\tau^*_{\ve,x,T} \right]  \\
+  C_{21}\ve T\sup_{0<x'<\sqrt{\ve}\del,T/2<t<T}\left|\frac{\pa^2 q_0(x',y,t)}{\pa x'^2} \right|E\left[(T-\tau^*_{\ve,x,T})Z(\tau^*_{\ve,x,T})^2;   \tau_\del\vee(T/2)<\tau^*_{\ve,x,T} \right] \ .
\end{multline}
We bound the first expectation  in (\ref{BO6}) by using the Schwarz inequality and the inequality $E\left[(T-\tau^*_{\ve,x,T})^4\right]\le E\left[(T-\tau^*_{\ve,{\rm linear}, x,T})^4\right]$. This latter expectation can be estimated by considering the function 
\be \label{BP6}
u_{\ve,\al}(x) \ = \ E\left[   \exp\left\{-\al(   T-\tau^*_{\ve,{\rm linear}, x,T})\right\}             \right] \ , \quad x,\al>0 \ .
\ee
Then $u_{\ve,\al}(\cdot)$ is the solution to a boundary value problem
\be \label{BQ6}
\frac{\ve}{2}\frac{d^2u_{\ve,\al}(x)}{dx^2}-\mu\frac{du_{\ve,\al}(x)}{dx} \ = \ \al u_{\ve,\al}(x) \ ,  \ \ x>0, \quad u_{\ve,\al}(0)=1 \ ,
\ee
with $\mu=y/C_5T$. 
 Evidently we have that
\be \label{BR6}
u_{\ve,\al} (x) \ = \ \exp\left[-\frac{x}{\ve}\left\{\sqrt{\mu^2+2\ve\al}-\mu\right\}\right] \ .
\ee
From (\ref{BR6}) we conclude that
\begin{multline} \label{BS6}
E\left[(T-\tau^*_{\ve,{\rm linear}, x,T})^4\right] \\ =  \left(\frac{\pa}{\pa \al}\right)^4 u_{\ve,\al}(x)\Big|_{\al=0} \
= \ \frac{15\ve^3 x}{\mu^7}+\frac{15\ve^2 x^2}{\mu^6}+\frac{6\ve x^3}{\mu^5}+\frac{x^4}{\mu^4} \ .
\end{multline}
It follows from (\ref{AJ6}), (\ref{BS6}) that the $\limsup$ of $x^{-1}$ times the first expectation in (\ref{BO6}) as $x\ra 0$  is bounded by  $C_{22}\ve^2T/y^3$. 

We also use the Schwarz inequality to bound the second expectation in (\ref{BO6}) by using the fact that  for $\al\in\mathbb{R}$ the function
 \be \label{BT6}
 s\ra \exp\left[ \al Z(s)-\frac{\al^2}{2}\int_s^T \frac{m_{1,A}(s')^2}{\sig_A^4(s')}  \ ds'\right] \ , \quad 0<s<T \ ,
 \ee
is also a martingale. On differentiating (\ref{BT6}) twice with respect to $\al$ and setting $\al=0$ we see that (\ref{AI6}) is a martingale.  On differentiating four times we have that
 \be \label{BU6}
 s\ra Z(s)^4-6Z(s)^2\int_s^T \frac{m_{1,A}(s')^2}{\sig_A^4(s')}  \ ds' +3\left(\int_s^T \frac{m_{1,A}(s')^2}{\sig_A^4(s')}  \ ds'\right)^2 , \quad 0<s<T \ ,
 \ee
is a martingale.  Hence we have that
\begin{multline} \label{BV6}
E\left[ Z(\tau^*_{\ve,x,T})^4  \ ; \ \tau^*_{\ve,x,T}>T/2\ \right] \\
 \le \  C_{23}\left\{\frac{1}{T^4}E\left[(T-\tau^*_{\ve,x,T})^2\right] +\frac{1}{T}E[Z(T/2)^2;  \ \tau^*_{\ve,x,T}< T/2] \right\} \ .
\end{multline} 
We have already seen that the first expectation on the RHS of (\ref{BV6}) is bounded by a constant times $v_\ve(x)$ of (\ref{AL6}).  To bound the second expectation we proceed in a similar way to how we bounded the expectation in (\ref{AO6}). We have that
\begin{multline} \label{BW6}
 \ve \frac{\sig^4_A(T)}{m_{1,A}(T)^2}E[Z(T/2)^2;  \ \tau^*_{\ve,x,T}< T/2] \\ 
 \le \  C_{24}E\left[  [X_\ve(T/2)-x_{\rm class}(T/2)]^2; \    \tau_{\ve,{\rm linear},x,T}<T/2\right]  \    ,             
\end{multline}
and the expectation on the RHS of (\ref{BW6}) can be bounded using solutions to the PDE (\ref{AR6}).  Thus we have that the $\limsup$ of $x^{-1}$ times the  RHS of (\ref{BW6}) as $x\ra 0$ is bounded by $C_{24}\ve^2  T^2/y^3$, whence the $\limsup$ of $x^{-1}$ times the  RHS of (\ref{BV6}) as $x\ra 0$ is bounded by 
$C_{25}\ve  /Ty^3$. We conclude that  the $\limsup$ of $x^{-1}$ times the second expectation in (\ref{BO6}) as $x\ra 0$  is bounded by  $C_{26}\ve^2T^3/y^4\le C_{27}\ve^2T/y^3$.
\end{proof}

\end{document}